\documentclass{amsart}
\usepackage{graphicx}
\usepackage{color,amssymb}
\usepackage{stmaryrd,esint}
\newcommand{\dt}{\partial_t}
\newcommand{\dx}{\partial_x}
\newcommand{\Pint}{\underline{P}_{\rm i}}
\newcommand{\eps}{\varepsilon}
\newcommand{\av}[1]{\langle#1\rangle}
\newcommand{\avi}[1]{\langle#1\rangle}
\newcommand{\jump}[1]{\llbracket#1\rrbracket}
\newcommand{\jumpi}[1]{\llbracket#1\rrbracket}
\newcommand{\abs}[1]{\vert#1\vert}
\newcommand{\uq}{{\underline{q}}}

\newcommand{\uh}{{\underline{h}}}
\newcommand{\uu}{{\underline{u}}}
\newcommand{\udelta}{{\underline{\delta}}}

\newcommand{\cE}{{\mathcal E}}

\newcommand{\cK}{{\mathcal K}}

\newcommand{\dsp}{\displaystyle}

\newcommand{\mfe}{{\mathfrak e}}
\newcommand{\mfE}{{\mathfrak E}}
\newcommand{\mfF}{{\mathfrak F}}
\newcommand{\mfsw}{{\mathfrak f}_{\rm sw}}

\newcommand{\RR}{{\mathbb R}}

\newcommand{\HH}{{\mathbb H}}

\newcommand{\Hardya}{{\mathcal H}^2({\mathbb C}_a)}

\newtheorem{theorem}{Theorem}[section] 
\newtheorem{proposition}{Proposition}[section]
\newtheorem{lemma}{Lemma}[section]

\newtheorem{definition}{Definition}[section]
\newtheorem{corollary}{Corollary}[section]

 \theoremstyle{remark}
\newtheorem{notation}{Notation}
\newtheorem{remark}{Remark}[section]

\usepackage{ulem}

\numberwithin{equation}{section}

\title[Freely floating object with the Boussinesq equations]{Freely floating objects on a fluid governed by the Boussinesq equations}
\author{G. Beck and  D. Lannes}
\address{D\'epartement de math\'ematiques et applications, \'Ecole normale sup\'erieure, CNRS, PSL University, 5 rue d'Ulm, 75005 Paris, France}
\email{beck@dma.ens.fr}
\thanks{G. B. was supported by the Del Duca fondation}
\address{Institut de Mathématiques de Bordeaux, CNRS UMR 5251 et Université de Bordeaux, 351 Cours de la Libération, 33405 Talence Cedex}
\email{David.Lannes@math.u-bordeaux.fr}
\thanks{D. L is partially supported by the ANR-18-CE40-0027 Singflows, the ANR-17-CE40-0025 NABUCO and the Del Duca fondation}

\begin{document}

\begin{abstract}
We investigate here the interactions of waves governed by a Boussinesq system with a partially immersed body allowed to move freely in the vertical direction. We show that the whole system of equations can be reduced to a transmission problem for the Boussinesq equations with transmission conditions given in terms of the vertical displacement of the object and of the average horizontal discharge beneath it; these two quantities are in turn determined by two nonlinear ODEs with forcing terms coming from the exterior wave-field. Understanding the dispersive contribution to the added mass phenomenon allows us to solve these equations, and a new dispersive hidden regularity effect  is used to derive uniform estimates with respect to the dispersive parameter. We then derive an abstract general Cummins equation describing the motion of the solid in the return to equilibrium problem and show that it takes an explicit simple form in two cases, namely, the nonlinear non dispersive and the linear dispersive cases; we show in particular that the decay rate towards equilibrium is much smaller in the presence of dispersion. The latter situation also involves an initial boundary value problem for a nonlocal scalar equation that has an interest of its own and for which we consequently provide a general analysis.
\end{abstract}

\maketitle


\section{Introduction}

\subsection{General setting}

Water waves have been studied quite intensively in the last decades, from the theoretical, modeling, and numerical viewpoints. Even though considerable progress has been made, the water waves equations (or free surface Euler equations) remain too complex to be used for most applications and reduced asymptotic models are used instead. In coastal regions in particular, {\it shallow water} models are used. These models are simpler because they take advantage of the vertical structure of the velocity field in shallow water to get rid of the dependance on the vertical variable (see the review paper \cite{Lannes_SW}): the equations are therefore $d$ dimensional instead of $d+1$ dimensional for the water waves problem ($d$ is the horizontal dimension), and the problem is no longer a free boundary problem. For instance, in dimension $d=1$ and at first order with respect to the so-called shallowness parameter (see Appendix \ref{appND}), one finds the nonlinear shallow water equations
\begin{equation}\label{intNSW}
\begin{cases}
\dt \zeta+\dx q=0,\\
\dt q +\dx\big( \frac{1}{h} q^2 \big)+g h\dx \zeta=-\frac{1}{\rho}h \dx \underline{P},
\end{cases}
\end{equation}
where $\zeta$ denotes the elevation of the surface with respect to the rest state, $h=h_0+\zeta$ the total height of the water column ($h_0$ is the depth of the fluid at rest), and $q$ is the horizontal discharge, that is, the vertical integral of the horizontal velocity of the fluid; we also denoted by $\rho$ the constant density of the fluid and by $\underline{P}$ the pressure at the surface. (for instance, a constant atmospheric pressure). This system is an hyperbolic quasilinear system.\\
At second order with respect to the shallow water parameter, but under a weak nonlinearity assumption, a popular model is the Boussinesq-Abbott system
\begin{equation}\label{intBA}
\begin{cases}
\dt \zeta+\dx q=0,\\
(1-\frac{h_0^2}{3}\dx^2)\dt q +\dx\big( \frac{1}{h} q^2 \big)+g h\dx \zeta=-\frac{1}{\rho} \dx \underline{P},
\end{cases}
\end{equation}
which can be seen as a dispersive perturbation of the nonlinear shallow water equations \eqref{intNSW}. Other equivalent Boussinesq systems can be derived, and the weak-nonlinearity assumption can be lifted, leading to the more complicated Serre-Green-Naghdi equations; we refer to \cite{Lannes_SW} for more details on these models that will not be addressed in this paper.

\medbreak

Motivated by ship motion and more recently by applications to marine renewable energies (offshore wind energies, or wave energy convertors), several authors addressed the issue of the interaction of waves with a floating object. This problem adds another layer of complexity to the water waves problem because it involves several other free boundary problems (the position of the object, the motion of the contact line with the surface of the fluid) and CFD simulations such as Reynolds Averaged Navier-Stokes (RANS) simulations are far from being able to simulate a full sea state for a single floating object \cite{Claes}. A less precise and less general, but potentially much more efficient alternative is to develop an approach based on the aforementioned reduced models; this study has to be understood as a step in this direction.\\
Early studies considered infinitesimal motions and focused mainly on the stability of the equilibrium of floating bodies \cite{Appell,John1} and engineers use a phenomenological linear integro-differential equation, the so-called Cummins equations  \cite{Cummins,Wamit} to describe the motion of the floating object. In these linear models, the pressure $\Pint$ exerted by the fluid on the object is given by the (linear approximation of the) Bernoulli equation,
$$
-\frac{1}{\rho}\Pint=g \zeta_{\rm w}+\dt \Phi_{\vert_{z=\zeta_{\rm w}}},
$$
where $\zeta_{\rm w}$ is the parametrization of the wetted part of the floating body, and $\Phi$  the velocity potential of the fluid. The first term of the right-hand side is called hydrodynamic pressure, and the second one the dynamic pressure. The velocity potential necessary to compute this latter is found by solving a Poisson equation in the fluid domain with mixed boundary condition at the surface (homogeneous Dirichlet on the free surface, non homogeneous Neumann on the bottom of the object); finally the equations are complemented by the (linearized) kinematic equation $\dt \zeta=\partial_z\Phi_{\vert_{z=0}}$ for the free surface $\zeta$ and by Newton's equation for the motion of the solid. 
A simpler linear shallow water approximation was also proposed in dimension $d=1$ in \cite{John1}, basically consisting in replacing $\partial_z\Phi_{\vert_{z=0}}$ by $-\partial_x^2 \Phi_{\vert_{z=0}}$ in the kinematic equation (see Remark \ref{remJohn} below).

\medbreak
The above formulation of the problem of waves interacting with a floating body can easily be extended to the nonlinear case (see for instance \cite{Daalen} where Zakharov's Hamiltonian formulation of the water waves problem is extended in the presence of a floating object or, for instance \cite{Monteserin,Benoit} for numerical studies). We do not provide too many details here because we shall rather use the approach of \cite{Lannes_float} in which the pressure $\Pint$ exerted by the fluid on the object is understood as the Lagrange multiplier associated to the constraint that, under the object, the surface of the fluid coincides with the bottom of the object. More precisely, a formulation of the water waves equations in term of the surface elevation $\zeta$ and the horizontal discharge $Q\in \RR^d$ was proposed in \cite{Lannes_float} that reads
\begin{equation}\label{WWzetaQ}
\begin{cases}
\dt \zeta+\nabla\cdot Q=0,\\
\dt Q+\nabla\cdot \big( \int_{-h_0}^\zeta V\otimes V \big) +gh \nabla\zeta+\frac{1}{\rho}\int_{-h_0}^\zeta \nabla P_{\rm NH}=-\frac{1}{\rho} h \nabla\underline{P},
\end{cases}
\end{equation}
where $V$ is the horizontal component of the velocity field in the fluid domain, $\underline{P}$ is the pressure at the surface, and $P_{\rm NH}$ is the non hydrostatic pressure in the fluid, and whose exact expression has no importance here.
In the parts where the object is not in contact with the water (the exterior region), $\underline{P}$ is the constant atmospheric pressure and the right-hand side vanishes in the second equation. In the region located under the object (the interior region), one has $\underline{P}=\Pint$, which is the Lagrange multiplier of the aforementioned constraint that can be written, using the first equation, as
$$
\nabla\cdot Q=-\dt \zeta_{\rm w}.
$$
One must therefore handle a system of equations of "compressible" type in the exterior region, with a system of equation of "incompressible" type in the interior region; this coupling is reminiscent of what happens in other contexts for congested flows \cite{Perrin}. Note that there are other ways of exploiting the fact that the pressure is a Lagrange multiplier, as in \cite{Bokhove} for instance where a discrete constrained variational numerical scheme is proposed for the simulation of wave-buoy interactions in shallow water. 

\medbreak

The interest of this approach, whose relevance has been confirmed by comparisons with numerical simulations solving the fully nonlinear equations \cite{Monteserin}, is that it is quite flexible; indeed, instead of the full water waves equations in $(\zeta,Q)$ formulation \eqref{WWzetaQ}, it is possible to implement it with simpler asymptotic models, and even to numerical schemes. In this paper, we shall analyze the equations obtained when this method is applied with the nonlinear shallow water equations \eqref{intNSW} and, more specifically,  with the Boussinesq equations \eqref{intBA}. \\
The equations obtained in the case of the nonlinear shallow water equations have been studied and solved in \cite{IguchiLannes}; the problem is surprisingly difficult because of the dynamics of the contact points at the transition between the interior and the exterior region. The problem can be reduced to a free boundary hyperbolic transmission problem reminiscent of the one obtained for the stability of shocks \cite{Majda,Metivier2001,BenzoniSerre}, but with the Rankine-Hugoniot condition replaced by a fully nonlinear condition.  One way to circumvent this difficulty is to consider floating objects with vertical sidewalls; in this case, the horizontal position of the contact points is no longer a free boundary problem as it is known if the position of the object is known. This situation was considered for a solid allowed to move in vertical translation only in \cite{Lannes_float} in dimension $d=1$, and in \cite{Bocchi} when $d=2$ with a radial symmetry. In \cite{Tucsnak} the same situation was considered with $d=1$ in the presence of viscosity. This approach has also been used to model a wave energy convertor named the oscillating water column in \cite{Bocchi2}. These references deal with the exact nonlinear shallow water equations with a constraint accounting for the presence of the floating body; it can be interesting, especially for numerical simulations, to relax this constraint and use techniques similar to those used to study low-Mach regimes in gases; this is the approach followed in \cite{Parisot1,Parisot2}.\\
The equations obtained in the case of the Boussinesq equations have been much less studied. The reason is that while initial boundary value problems are quite well understood for hyperbolic systems (\cite{Majda,Metivier2001,BenzoniSerre}, as well as \cite{IguchiLannes} for more recent and precise results in the specific case $d=1$), there is no such theory for dispersive perturbations of such systems, such as the Boussinesq-Abbott equations \eqref{intBA}. In \cite{Jiang} and \cite{Bosi}, the authors propose a system of two Boussinesq systems (one in the interior region and one in the exterior region), while $\Pint$ is numerically solved so that these two sets of equations are compatible but the formulation used there does not allow to write the simple explicit elliptic equation on $\Pint$ used here and associated with the constraint on the surface. The approach consisting in using constrained Boussinesq equations to model the presence of a floating object has only been treated in \cite{BLM}, with $d=1$ and with a fixed object. Moreover, the Boussinesq system used there is a variant of \eqref{intBA}, physically less interesting but mathematically more convenient because in the case of a fixed object  the problem can then be reduced to a transmission problem with {\it linear} transmission conditions; as we shall see, working with \eqref{intBA} and/or a non fixed object leads to more complicated {\it nonlinear} transmission conditions. One of the main features of \cite{BLM} is that it shows the role of dispersive boundary layers associated with the dispersive term of the Boussinesq equations and that we will of course have to deal with here. \\
The goal of this paper is to treat the interaction of waves governed by the Boussinesq-Abbott system \eqref{intBA} with an object with vertical sidewalls allowed to move freely in the vertical direction. To this end, we need to address several issues. For the modeling aspects, if the elliptic equation for $\Pint$ is quite straightforward to derive, the boundary conditions necessary to solve it are not clear and require some work; one also needs to understand the coupling with Newton's equations that govern the motion of the solid, and in particular the influence of the dispersive terms on the added mass phenomenon. The formulation of the whole set of equations involved as a quite simple transmission problem for the Boussinesq-Abbott equations is also of particular interest since it is very adapted for efficient numerical simulations (work in progress) and can be used to provide a useful qualitative insight, as shown here for the return to equilibrium problem (also called "decay test", it is a standard benchmark used by engineers in particular to calibrate coefficients in the Cummins equation). For the theoretical aspects, the contribution of the dispersive effects to the added mass phenomenon (that where not treated in \cite{BLM} because the object was fixed) require special attention, and the nonlinear nature of the dispersive contribution in the transmission conditions make the derivation of uniform energy estimate much more complicated than in \cite{BLM}: we have to exploit a new type of hidden regularity at the boundary, granted by the dispersive terms and that is of independent interest for the analysis of initial boundary value problems in the presence of dispersive terms. The dispersive terms induce also nonlocal effects in the analysis of the return to equilibrium problem; this leads us to develop a general study for the analysis of initial boundary value problems for nonlocal scalar equations that exhibits interesting phenomena when the "local" limit is considered and a dispersive smoothing that can also be of interest in other contexts where nonlocal scalar equations are involved.

\subsection{Organization of the paper}

Section \ref{sectderiv} is devoted to the derivation of the wave-structure interaction equations in the framework described above. We briefly describe  (in dimensionless form) in \S \ref{sectfluid} the Boussinesq-Abbott system used to describe the propagation of the waves and provide in \S \ref{sectsolid} the dimensionless version of Newton's equations for a solid allowed to move only in the vertical direction. We also need coupling conditions between the interior and exterior regions; they are described  in \S \ref{sectWS}. The issue mentioned above for the boundary conditions on the interior pressure $\Pint$ is addressed in \S \ref{secteqint}. It is then possible to solve for $\Pint$ and to reduce the equations in the interior region  to a set of two ODEs on the vertical displacement $\delta$ and on the horizontal average discharge $\av{q_{\rm i}}$, with source terms accounting for the coupling with the exterior wave field. \\
These elements are used in Section \ref{sectWSITP} to reduce all the equations involved in the wave-structure interaction problem under consideration to a transmission problem, see \S \ref{sectderivWS}. The mathematical structure of this transmission problem is investigated in  \S \ref{sectStudytoy} and \S \ref{sectReform} where a reformulation of the equations is proposed to exhibit the nontrivial contribution of the dispersive terms to the added mass phenomenon. We then show in \S \ref{sectreducODE} that the whole system can be reduced to an infinite dimensional ODE, allowing us to establish well-posedness. In order to control the existence time as the dispersive parameter $\mu$ goes to zero,  we  address in \S \ref{sectUE} the issue of  uniform estimates and exhibit in particular a new hidden regularity phenomenon associated with the dispersive terms.

In Section \ref{sectreturn} we describe a special configuration, the return to equilibrium (or decay test), which is both of practical interest because it is a classical benchmark for engineers, and theoretically interesting because it allows one to provide details on the qualitative behavior of the solutions. In particular, we want to investigate whether the solid motion is governed by the Cummins equation, an integro-differential equation used by engineers, and whether we are able to generalize this equation to the nonlinear framework. We show in \S \ref{sectgCummins} how to derive an abstract evolution equation of Cummins-type to describe the solid motion. This abstract equation turns out to reduce to a second order nonlinear scalar ODE in the nonlinear non dispersive case and to a linear integro-differential equation in the linear dispersive case, see   \S \ref{sectnondisp} and \S \ref{sectlinear} respectively. The qualitative behavior of the solutions is commented in both cases; in particular we numerically observe and theoretically prove that the presence of dispersion makes the return to equilibrium slower. It is also shown that the waves in the exterior domain can be found by solving an initial boundary value problem for a Burgers equation in the nonlinear case,  and for a nonlocal perturbation of a linear transport equation in the dispersive case.

The nonlocal initial boundary value problem just mentioned does not fit into any general theory, and since similar problems are likely to appear in other contexts where nonlocal equations play a role, we address this issue in Section  \ref{sectNLPB}. We consider a nonlocal perturbation of a scalar transport equation (we consider both positive and negative velocity). We show the well-posedness of these problems, but under one additional compatibility condition on the data. We explain why this additional compatibility condition disappears as the dispersive parameters tends to zero and the nonlocal transport equations formally converges to the standard transport equation. We also exhibit a smoothing effect associated with these nonlocal initial boundary value problems.

Finally, the link between the equations with dimensions and their dimensionless counterparts is made in Appendix \ref{appND}.

\subsection{Notations}

The horizontal axis ${\mathbb R}$ is decomposed throughout this paper into an {\it interior region} ${\mathcal I}=(-\ell,\ell)$ and an {\it exterior region} ${\mathcal E}={\mathcal E}^+\cup {\mathcal E}^-$ with ${\mathcal E}^-=(-\infty,-\ell)$ and ${\mathcal E}^+= (\ell,\infty)$, and two {\it contact points} ${x=\pm \ell}$. For any function $f$ admitting left and right limits at $\pm \ell$, we use the following notations:
\begin{itemize}
\item restriction to the interior domain $f_{\rm i} = f |_\mathcal{I}$,
\item restriction to the exterior domain $f_{\rm e} = f |_\mathcal{E}$,
\item {\it exterior} jump $ \jump f$ and {\it interior} jump $\jumpi{f_{\rm i}}$ defined as
\begin{equation}\label{jumps}
\jump{f} =f_{\rm e}(\ell)-f_{\rm e}(-\ell), \qquad \jumpi{f_{\rm i}}=f_{\rm i}(\ell)-f_{\rm i}(-\ell),
\end{equation}
\item {\it exterior } average $\av{f}$  and {\it interior} average $\avi{f_{\rm i}}$ defined as
\begin{equation}\label{averages}
\av{f} = \frac{1}{2}\big( f_{\rm e}(\ell)+ f_{\rm e}(-\ell)\big), \qquad
\avi{f_{\rm i}}= \frac{1}{2}\big( f_{\rm i}(\ell)+ f_{\rm i}(-\ell)\big),
\end{equation}
\item {\it exterior} trace at the boundary points of a function $f$,
$$
f_\pm:=(f_{\rm e})_{\vert_{x=\pm\ell}}=\lim_{x\to (\pm \ell)^\pm} f(x).
$$
\item In dimensionless variables, the bottom of the floating object is parametrized by $\eps\zeta_{\rm w}(t,x)$, the water height at equilibrium under the object is $h_{\rm ex}(x)<1$, and $\eps\delta(t)$ denotes at time $t$ the distance of the center of mass to its equilibrium position. These quantities are related by the relation
$$
\eps\zeta_{\rm w}(t,x)=\eps\delta(t)+\big(h_{\rm eq}(x)-1\big).
$$
\end{itemize}

\medbreak

We also need to introduce the following functional spaces and notations:
\begin{itemize}
 \item if $f$ is a function of time, we sometimes use the notations $\dot f=\frac{d}{dt} f$ and $\ddot f=\frac{d^2}{dt^2} f$,
 \item For all $f\in L^2(\cE)$, we simply write $\abs{f}_2$ the associated norm,
\item for all $n\in {\mathbb N}$, we denote by $H^n(\cE)$ the standard Sobolev space on $\cE$, and define
 $$\HH^n:=H^{n+1}(\cE)\times H^{n+2} (\cE),$$
 \item for all $\eta_0 \in \mathbb{R}$, we  denote 
 $$
 \mathbb{C}_{\eta_0} := \{ s \in \mathbb{C} \, | \, \frak{Re} (s) > \eta_0 \},
 $$
 \item for all $s=\eta+i\omega\in {\mathbb C}$ ($\eta,\omega\in \RR$), we denote by $\sqrt{s}$ the square root with positive real part.
\end{itemize}

\section{Derivation and analysis of the wave-structure interaction equations}\label{sectderiv}

This section is devoted to the derivation of wave-structure interaction equations in the case of a floating object allowed to move freely in the vertical direction (see Figure \ref{fig1}), and using a nonlinear dispersive wave model to describe the propagation of the waves. To this end, we follow the strategy of \cite{Lannes_float} where it was proposed to see the pressure  $\underline{P}_{\rm i}$ exerted by the fluid on the object as the Lagrange multiplier associated with the constraint that under the object, the surface of the water coincides with the bottom of the floating object. The first step, considered in \S \ref{sectfluid}, is to choose the model used to describe the propagation of waves; we choose here the Boussinesq-Abbott system which is a nonlinear dispersive set of equations commonly used to model wave propagation. We then write in \S \ref{sectsolid} the dimensionless version of Newton's equations for a solid allowed to move only in the vertical direction.
The way these two systems of equations are coupled is described in \S \ref{sectWS}. One of the coupling conditions turns out to be that the total (fluid+solid) energy of the system has to be conserved at the order of precision of the model; it is shown in \S \ref{secteqint} that this imposes boundary conditions on  the interior pressure  $\underline{P}_{\rm i}$. These boundary conditions allow one to solve the pressure equation in the interior region (under the object); it follows that in this region, all the equations can be reduced to a set of two ODEs on the vertical displacement $\delta$ and on the horizontal average $\av{q_{\rm i}}$ of the horizontal discharge; source terms in these ODEs account for the coupling with the exterior wave field. 
\begin{figure}[h]
\begin{center}
\includegraphics[width=0.9\linewidth]{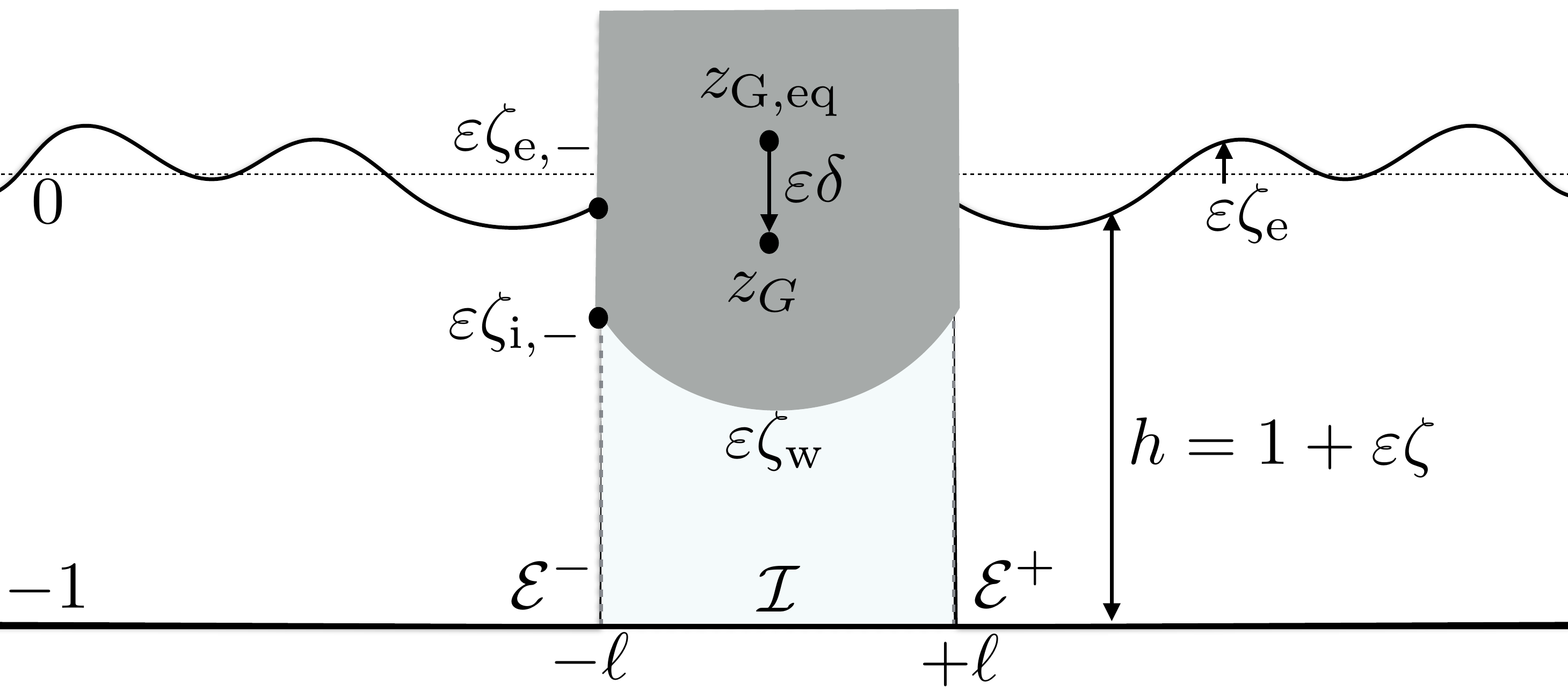}
\end{center}
\caption{The physical configuration (dimensionless variables)}\label{fig1}
\end{figure}

\subsection{The equations for the fluid}\label{sectfluid}

We consider in this paper {\it weakly nonlinear} waves in {\it shallow water}, that are known to be described with a good accuracy by Boussinesq-types systems and that are of interest for a wide range of applications. To be more precise, let us define the dimensionless nonlinearity parameter $\eps$ and the shallowness parameter $\mu$ as
$$
\eps=\frac{a}{H}, \qquad \mu=\frac{H^2}{L^2},
$$
where $a$ is the typical amplitude of the waves, $L$ their typical horizontal scale and $H$ the depth at rest. The {\it shallowness} assumption means that $\mu \ll 1$ and the statement that we have a {\it weak nonlinearity} means that $\eps=O(\mu)$. Under this latter assumption, Boussinesq systems are approximations of the full free surface Euler equations at order $O(\mu^2)$ (see for instance \cite{Lannes_book,Lannes_SW} for more details on the derivation and full justification of these models). There are actually many different Boussinesq systems that are formally equivalent since they differ only one from each other by terms of order $O(\mu^2)$, which do not affect the precision of the model. One of the most popular of these Boussinesq systems is the so-called Boussinesq-Abbott system \cite{Abbott,FBCR}, which reads, in dimensionless variables (see Appendix \ref{appND}) and when a pressure $P_{\rm atm}+\underline{P}$ (where $P_{\rm atm}$ is a constant reference value for the atmospheric pressure) is applied at the surface,
\begin{equation}\label{Abbott}
\begin{cases}
\partial_t \zeta +\partial_x q =0 
\\
(1-\frac{\mu}{3} \partial_x^2) \partial_t q + \varepsilon \partial_x \left( \frac{1}{h} q^2 \right) + h \partial_x \zeta
= -\frac{1}{\eps} h\dx \underline{P}
\end{cases}\qquad (h=1+\eps\zeta).
\end{equation}
Here, $\eps\zeta$ is the dimensionless surface elevation  with respect to its rest position, $h=1+\eps\zeta$ is the dimensionless water height, and $q$ is the horizontal discharge (the vertical integral of the horizontal velocity). In general, the pressure at the surface is a constant atmospheric pressure and  $\dx\underline{P}=0$, so that the right-hand side in the second equation of \eqref{Abbott} vanishes; we present the general system here because it is relevant in the presence of a floating body, under which the pressure is no longer equal to the atmospheric pressure (see \S \ref{secteqint} below). Note that $\underline{P}$ can be nontrivial in other contexts, for instance when one wants to study the impact of atmospheric disturbances on the waves \cite{Melinand}.\\
What makes the Boussinesq-Abbott system an interesting model is that it is a dispersive perturbation of the nonlinear shallow water equations written in conservative form
$$
\begin{cases}
\partial_t \zeta +\partial_x q =0 
\\
 \partial_t q + \varepsilon \partial_x \left( \frac{1}{h} q^2 \right) + h \partial_x \zeta
= -\frac{1}{\eps}h\dx \underline{P};
\end{cases}
$$
its drawback is that local (and global) conservation of energy is only satisfied at order $O(\eps\mu)$. Defining the local energy density $\mfe$ and the local energy flux ${\mfF}$ as
\begin{equation}
\label{defeF}
\begin{cases}
\mfe(\zeta,q)=\frac{1}{2}\zeta^2+\frac{1}{2} \frac{1}{h}q^2 +\mu \frac{1}{6h}(\dx q)^2,\\
\mfF(\zeta,q)=q\big(\zeta+\frac{1}{\eps}\underline{P}+\eps \frac{1}{2}\frac{q^2}{h^2}- \mu \frac{1}{3h} \dx\dt q \big),
\end{cases}
\end{equation}
one has indeed
\begin{equation}\label{eqNRJ}
\dt \mfe+\dx \mfF=\frac{1}{\eps}\underline{P}\dx q+   \eps\mu {\mathfrak R},
\end{equation}
with ${\mathfrak R}$ given by
\begin{equation}\label{defRmu}
{\mathfrak R}= \frac{1}{6h^2} (\dx q)^3+ \frac{1}{3h^2}q (\dt\dx q )\dx\zeta;
\end{equation}
the right-hand-side of \eqref{eqNRJ} is formally of size $O(\eps\mu)$ when $\underline{P}=0$ and therefore of size $O(\mu^2)$ in the weakly nonlinear regime $\eps=O(\mu)$. Setting $\mu=0$ in \eqref{defeF} and \eqref{eqNRJ}, one recovers the {\it exact} local conservation of energy associated with the nonlinear shallow water equations.

\begin{remark}
 In the second equation of the Boussinesq-Abbott system, one can replace $\eps\dx (\frac{1}{h}q^2) $ by $\eps \dx (q^2)$ up to a term of order $O(\eps^2)$ which is of order $O(\mu^2)$ in the weakly nonlinear regime. At order $O(\mu^2)$ and when $\underline{P}=0$, the following system is therefore formally equivalent to the Boussinesq-Abbott system \eqref{Abbott},
\begin{equation}\label{varBLM}
\begin{cases}
\partial_t \zeta +\partial_x q =0 ,
\\
(1-\frac{\mu}{3} \partial_x^2) \partial_t q + \varepsilon \partial_x (  q^2 ) + h \partial_x \zeta
= 0.
\end{cases}
\end{equation}
This system is no longer a dispersive perturbation of the nonlinear shallow water equations because the nonlinear terms are not the same, but it was used in \cite{BLM} because it satisfies an {\it exact} local conservation of energy,
\begin{equation}\label{NRJBLM}
\dt \widetilde{\mfe}+\dx \widetilde{\mfF}=0,
\end{equation}
but for a slightly different energy/flux pair,
\begin{align*}
\widetilde{\mfe}&=\frac{1}{2}\zeta^2+\frac{1}{2} q^2 + \frac{\eps}{6}\zeta^3+\frac{1}{6}\mu (\dx q)^2,\\
\widetilde{\mfF}&=q\big(\zeta+\eps \frac{2}{3} q^2+\eps \frac{1}{2}\zeta^2- \mu \frac{1}{3} \dx\dt q \big);
\end{align*}
in spite of this convenient property, we prefer to work with the Boussinesq-Abbott system \eqref{Abbott} rather than \eqref{varBLM} because, contrary to $\widetilde{\mfe}$,  $\mfe$ is an asymptotic expansion of the mechanical energy of the waves associated with the full water waves equations (see for instance \S 6.3.1 in \cite{Lannes_book}).
\end{remark}


\subsection{The equations for the solid}\label{sectsolid}

We refer to Appendix \ref{appND} for the derivation of the dimensionless Newton equations that we state here.  We recall that we consider here a floating object with vertical lateral walls located, in dimensionless coordinates, at $x=\pm \ell$ ($\ell>0$) and allowed to move only vertically (heave motion). At time $t$, the part of the bottom of the object in contact with the fluid is parametrized in dimensionless variables by a function $\eps\zeta_{\rm w}$ (the subscript w stands for the "wetted" part of the object), with
\begin{equation}\label{eqzw}
\zeta_{\rm w}(t,x)=\delta(t)+\frac{1}{\eps}\big(h_{\rm eq}(x)-1\big),
\end{equation}
where $\eps\delta(t)$ measures the vertical deviation of the object from its equilibrium position and $h_{\rm eq}$ the distance at equilibrium between the bottom of the object and the bottom of the fluid layer.\\
Denoting by $\Pint(t,x)$ the pressure exerted by the fluid on the object at the point $(x,\eps\zeta_{\rm w}(t,x))$ of the wetted surface, Newton's equation describing the vertical motion of the floating object under the action of its weight and of the hydrodynamic forces can be written as (see Appendix \ref{appND})
\begin{equation}\label{Newton-z_G-0}
 \tau_{\rm buoy}^2   \ddot{\delta} + \frac{1}{\eps}m= \displaystyle \frac{1}{\eps}\frac{1}{2\ell}\int_{-\ell}^\ell \underline{P}_{\rm i}(t,x){\rm d}x,
\end{equation} 
where $2\pi\tau_{\rm buoy}$ is the dimensionless buoyancy period (see Appendix \ref{appND}) and $m$ the dimensionless mass which, by virtue of Archimedes' principle, satisfies the relation
\begin{equation}\label{eqArchimede}
 m=\frac{1}{2\ell}\int_{-\ell}^\ell (1-h_{\rm eq}).
\end{equation}
A convenient equivalent formulation of Newton's equation is obtained by introducing the hydrodynamic pressure $\Pi_{\rm i}$ defined in  the interior region $(-\ell,\ell)$ as
\begin{equation}\label{defPi}
\Pi_{\rm i}(t,x)=\Pint(t,x)+\eps \zeta_{\rm w}(t,x);
\end{equation}
using \eqref{eqzw}, Newton's equation \eqref{Newton-z_G-0} is then equivalent to
\begin{equation}\label{Newton-z_G-ND}
\tau_{\rm buoy}^2   \ddot{\delta} + \delta= \displaystyle \frac{1}{\eps}\frac{1}{2\ell}\int_{-\ell}^\ell \Pi_{\rm i}(t,x){\rm d}x.
\end{equation} 

\begin{remark}
- One can of course add an external force to the right-hand side of \eqref{Newton-z_G-ND}.\\
- In the case of an object with prescribed motion, the function $\delta$ is known, and there is no need to use \eqref{Newton-z_G-0}.
\end{remark}

One naturally associates its mechanical energy ${\mathfrak E}_{\rm solid}$ with the object, which is the sum of its potential and kinetic energies. In dimensionless form, it is given by
$$
{\mathfrak E}_{\rm solid}=2\ell\big( \frac{1}{\eps}m\delta +\frac{1}{2}\tau_{\rm buoy}^2 \dot\delta^2 \big).
$$
The variations of the mechanical energy of the solid are due to the hydrodynamic forces; more precisely, we have
$$
\frac{d}{dt}{\mathfrak E}_{\rm solid}=2\ell\dot\delta \big( \frac{1}{\eps }m  +\tau_{\rm buoy}^2 \ddot\delta \big)
$$
and therefore, using \eqref{Newton-z_G-0}
\begin{equation}\label{varEsolid}
\frac{d}{dt}{\mathfrak E}_{\rm solid}=\Big(\frac{1}{\eps}\int_{-\ell}^\ell\Pint\Big) \dot\delta.
\end{equation}
\subsection{The wave-structure equations}\label{sectWS}

We recall that the Boussinesq-Abbott equations with a relative pressure $\underline{P}$ exerted at the surface are given by
\begin{equation}\label{Boussinesq}
\begin{cases}
\partial_t \zeta +\partial_x q =0 
\\
(1-\frac{\mu}{3} \partial_x^2) \partial_t q + \varepsilon \partial_x \left( \frac{1}{h} q^2 \right) + h \partial_x \zeta
= - \frac{1}{\eps} h \partial_x \underline{P};
\end{cases}
\end{equation}
among the three quantities involved in \eqref{Boussinesq}, namely, $\zeta$, $q$ and $\underline{P}$, only two are not constrained, but not always the same ones. To make a more precise statement, we must  distinguish between the {\it interior} domain ${\mathcal I}=(-\ell,\ell)$ which is the projection on the horizontal axis of the region where the surface of the water is in contact with the object, and the {\it exterior} domain ${\mathcal E}={{(-\infty,-\ell) \cup (\ell,\infty)}}$, where it is in contact with the air:
\begin{itemize}
\item In the {\it exterior domain}, the surface of the fluid is free, but the pressure is constrained. In absence of surface tension, the pressure at the surface should match the atmospheric pressure which we assume to be constant. Recalling that if $f$ is a function on $\RR$, we denote by $f_{\rm e}$ its restriction to the exterior domain $\cE$, we have therefore
\begin{equation}\label{contrainte1}
\underline{P}_{\rm e}=0
\end{equation}
while $\zeta_{\rm e}$ and $q_{\rm e}$ must solve  the standard Boussinesq-Abbott system
\begin{equation}\label{Boussinesq_ext}
\begin{cases}
\partial_t \zeta_{\rm e} +\partial_x q_{\rm e} =0 
\\
(1-\frac{\mu}{3} \partial_x^2) \partial_t q_{\rm e} + \varepsilon \partial_x \left( \frac{1}{h_{\rm e}} q_{\rm e}^2 \right) + h_{\rm e} \partial_x \zeta_{\rm e}
= 0,
\end{cases}\quad\mbox{ for }\quad t\geq 0,\quad x\in {\mathcal E}.
\end{equation}
\item In the {\it interior domain}, we have a symmetric situation in the sense that the pressure is free but the surface of the water is constrained: by definition of the interior domain, it should coincide with the bottom of the object which is parametrized by $\eps\zeta_{\rm w}(t,x)$. Recalling that $f_{\rm i}$ denotes the restriction of a function $f$ to the interior domain ${\mathcal I}$, we have therefore
\begin{equation}\label{contrainte2}
\zeta_{\rm i}=\zeta_{\rm w},
\end{equation}
with $\zeta_{\rm w}$ given by \eqref{eqzw} while $q_{\rm i}$ and $\Pint$ must solve
\begin{equation}\label{Boussinesq_int}
\begin{cases}
\partial_x q_{\rm i} =- \dot\delta
\\
\partial_t q_{\rm i} + \varepsilon \partial_x \left( \frac{1}{h_{\rm w}} q_{\rm i}^2 \right) + h_{\rm w} \partial_x \zeta_{\rm w}
= - \frac{1}{\eps}h_{\rm w} \partial_x {\underline{P}}_{\rm i},
\end{cases}
\quad\mbox{ for }\quad t\geq 0,\quad x\in {\mathcal I},
\end{equation}
with $h_{\rm w}=1+\eps\zeta_{\rm w}$ and where we used the fact $\dt\zeta_{\rm w}=\dot \delta$ and that $\dx^2 q_{\rm i}=0$.
\end{itemize}
The constraints \eqref{contrainte1} and \eqref{contrainte2} together with the systems of equations \eqref{Boussinesq_ext} and \eqref{Boussinesq_int} and Newton's equation \eqref{Newton-z_G-ND}
are not enough to fully determine $(\zeta,q,\underline{P})$ in both regions ${\mathcal E}$ and ${\mathcal I}$, and the position $\delta(t)$ of the object. Indeed coupling conditions between the exterior and interior regions are required:
\begin{itemize}
\item {\it Continuity of the discharge}, namely,
\begin{equation}\label{coupling-q-i-e}
q_{\rm i}(t,\pm \ell)=q_{\rm e}(t,\pm \ell),
\end{equation}
\item {\it Conservation of the total energy} at the order of precision of the model.
 As already noticed in \cite{BLM} in the case of a fixed object, \eqref{coupling-q-i-e} is not enough to obtain a closed set of equations. In \cite{BLM}, where the system \eqref{varBLM} was used, the exact equation \eqref{NRJBLM} for the local conservation of energy was used and a boundary condition for the pressure was derived by imposing the exact conservation of the energy of the fluid, or equivalently, since the solid was considered fixed, of the energy of the fluid+solid system. In the present case where the object is allowed to move, this condition becomes more complex and because the local conservation of the energy \eqref{eqNRJ} is only satisfied at order $O(\eps\mu)$ for the Boussinesq-Abbott system \eqref{Abbott}, the additional conditions must be stated as
\begin{equation}\label{addcoupling}
\mbox{\it The total energy of the fluid+solid system is conserved at order  }O(\eps\mu),
\end{equation}
 We show in \S \ref{sectderCB} below how to derive boundary conditions on $\Pint$ at $\pm \ell$ from this condition.
\end{itemize}

The remaining of this section and Section \ref{sectWSITP} are devoted to the proof of the fact that  the constraints \eqref{contrainte1} and \eqref{contrainte2},  the systems of equations \eqref{Boussinesq_ext} and \eqref{Boussinesq_int}, the coupling conditions \eqref{coupling-q-i-e} and \eqref{addcoupling}, together with Newton's equation \eqref{Newton-z_G-ND} form a well-posed system of equations (in a sense made precise below) that fully determines $(\zeta,q,\underline{P})$ in both regions $\cE$ and ${\mathcal I}$, as well as the position $\delta(t)$ of the floating object.

\subsection{The equations in the interior domain and the solid motion}\label{secteqint}

As said above, in the interior domain ${\mathcal I}=(-\ell,\ell)$, the surface elevation is constrained (one has $\zeta_{\rm i}=\zeta_{\rm w}$ with $\zeta_{\rm w}$ given by \eqref{eqzw}) but the surface pressure is an unknown quantity, denoted by $\underline{P}_{\rm i}$. As seen in \eqref{Boussinesq_int}, the mass conservation equation and the constraint on the free surface also imply that in the interior domain, one has $\dx q_{\rm i}=-\dot\delta$ and therefore 
\begin{equation}\label{exprqi}
q_{\rm i}(t,x)=-x\dot\delta +\av{q_{\rm i}}(t)
\end{equation}
where the {\it mean horizontal discharge} $\av{q_{\rm i}}$ is a function of time that needs to be determined. 

We show in \S \ref{sectderCB} how to derive equations for the interior pressure; these equations can be used to make more explicit the equation for the displacement $\delta$  of the floating object, exhibiting in particular the added mass phenomenon (see \S \ref{sectdelta}); the case of the mean discharge $\av{q_{\rm i}}$ is finally handled in \S \ref{sectavq}.

\subsubsection{The interior pressure}\label{sectderCB}

We first show how the condition \eqref{addcoupling} on the conservation of the total energy  can be used to find the boundary values of the interior pressure at $x=\pm \ell$.
The energy ${\mathfrak E}_{\rm fluid}$ of the fluid  can be decomposed into two parts corresponding to the exterior and interior regions, respectively denoted $\cE=(-\infty,-\ell)\cup (\ell,\infty)$ and ${\mathcal I}=(-\ell,\ell)$, 
\begin{align*}
{\mathfrak E}_{\rm fluid}=\int_{\cE} \mfe_{\rm e} + \int_{\mathcal I} \mfe_{\rm i},
\end{align*}  
with $\mfe$ as in \eqref{defeF},
while we recall that the mechanical energy of the solid is given by 
$$
{\mathfrak E}_{\rm solid}=2\ell\big( \frac{1}{\eps}m\delta +\frac{1}{2}\tau_{\rm buoy}^2 \dot\delta^2 \big);
$$
the total energy of the fluid+solid system is
$$
{\mathfrak E}_{\rm tot}={\mathfrak E}_{\rm fluid}+{\mathfrak E}_{\rm solid}.
$$
The following Proposition shows that if the energy flux ${\mathfrak F}$ introduced in \eqref{defeF}, namely,
$$
\mfF(\zeta,q)=q\big(\zeta+\frac{1}{\eps}\underline{P}+\eps \frac{1}{2}\frac{q^2}{h^2}- \mu \frac{1}{3h} \dx\dt q \big),
$$
is continuous across the contact points $x=\pm\ell$ then the total energy is conserved up to $O(\eps\mu)$ terms.
\begin{proposition}
Any regular enough solution of the  wave-structure equations  \eqref{Newton-z_G-ND} and \eqref{contrainte1}-\eqref{Boussinesq_int} satisfies
$$
\frac{d}{dt}{\mathfrak E}_{\rm tot}=\jump{{\mathfrak F}_{\rm e}-{\mathfrak F}_{\rm i}}+\eps\mu \big(\int_{{\mathcal I}} {\mathfrak R}_{\rm i}+\int_{\cE} {\mathfrak R}_{\rm e} \big),
$$
with ${\mathfrak F}$ as in \eqref{defeF} and ${\mathfrak R}$ as in \eqref{defRmu}.
\end{proposition}
\begin{proof}
For the sake of clarity, we simply denote by $O(\eps\mu)$ instead of $\eps\mu \big(\int_{{\mathcal I}} {\mathfrak R}_{\rm i}+\int_{\cE} {\mathfrak R}_{\rm e} \big)$ in the computations below.
One computes
\begin{align*}
\frac{d}{dt}{\mathfrak E}_{\rm fluid}&=\int_{\cE} \dt \mfe_{\rm e} + \int_{\mathcal I} \dt \mfe_{\rm i}\\
&=-\int_{\cE} \dx {\mathfrak F}_{\rm e} - \int_{\mathcal I} ( \dx{\mathfrak F}_{\rm i}   - \frac{1}{\eps}\Pint \dx q_{\rm i} \big)  + O(\eps\mu),
\end{align*}
where we used the approximate conservation of local energy \eqref{eqNRJ}. Recalling the definition \eqref{jumps} of the exterior and interior jumps, and since $\dx q_{\rm i}=-\dot\delta$, this yields
\begin{align*}
\frac{d}{dt}{\mathfrak E}_{\rm fluid}&=\jump{{\mathfrak F}_{\rm e}}-\jumpi{{\mathfrak F}_{\rm i}}- \Big(\frac{1}{\eps} \int_{-\ell}^\ell \Pint\Big)\dot\delta +O(\eps\mu).
\end{align*}
Together with \eqref{varEsolid}, this directly gives the result.
\end{proof}
The following corollary shows that if the coupling condition \eqref{coupling-q-i-e} is satisfied then the condition \eqref{addcoupling} on the conservation of the total energy reduces to imposing boundary condition at $x=\pm \ell$ on the interior pressure $\Pint$ or equivalently on the interior hydrodynamic pressure $\Pi_{\rm i}$ given by \eqref{defPi}, namely,
$$
\Pi_{\rm i}=\underline{P}_{\rm i}+\eps \zeta_{\rm w}.
$$
\begin{corollary}\label{coroBC}
Assume that in addition the condition \eqref{coupling-q-i-e} on the continuity of the discharge is satisfied, and  that the traces $\Pi_{\rm i}^\pm$ of the interior hydrodynamic pressure at  $x=\pm \ell$   are given by
\begin{equation}\label{CBpress}
\frac{1}{\eps}\Pi_{\rm i}^\pm=\zeta_{\rm e}^\pm+{\mathfrak G}_{\rm e}^\pm-{\mathfrak G}_{\rm i}^\pm
\end{equation}
where $\Pi_{\rm i}$ is as defined in \eqref{defPi} and
\begin{equation}\label{defG}
{\mathfrak G}=\eps \frac{1}{2}\frac{q^2}{h^2}-\frac{\mu}{3 h}\dx\dt q.
\end{equation}
Then one has 
$$
\frac{d}{dt} {\mathfrak E}_{\rm tot}=\eps\mu \big(\int_{{\mathcal I}} {\mathfrak R}+\int_{\cE} {\mathfrak R} \big).
$$
\end{corollary}
\begin{remark}\label{remGi}
Recalling that the first equation of \eqref{Boussinesq_int} implies
$$
q_{\rm i}=-x \dot\delta +\avi{q_{\rm i}},\qquad \dx\dt q_{\rm i}=-\ddot\delta,
$$
and that $\zeta_{\rm i}=\zeta_{\rm w}$ with $\zeta_{\rm w}$ given by \eqref{eqzw} (so that $h_{\rm w}=h_{\rm eq}+\eps\delta$), we have
$$
{\mathfrak G}_{\rm i}^\pm=\delta+\eps\frac{1}{2}\frac{(\av{q_{\rm i}}\mp \ell \dot\delta)^2}{h_{\rm w}^2}+\frac{\mu}{3 h_{\rm w}}\ddot\delta,
$$
which does not depend on any other unknown of the problem than the functions $\delta$ and $\av{q_{\rm i}}$.

\end{remark}
\begin{proof}
From the proposition, it is enough to show that under the assumptions of the corollary, one has
$$
\jump{{\mathfrak F}_{\rm e}-{\mathfrak F}_{\rm i}}=0.
$$
Using the hydrodynamic pressure $\Pi=\underline{P}+\eps \zeta$, one can write
$$
{\mathfrak F}=q\big( \frac{1}{\eps}\Pi+{\mathfrak G} \big)
$$
 with ${\mathfrak G}$ as in the statement of the corollary.
Using the identity
$$
\jump{f g}=\jump{f}\av{g}+\av{f}\jump{g},
$$
and remarking that the continuity of $q$ at $\pm\ell$ implies that $\av{q}=\avi{q_{\rm i}}$ and $\jump{q}=\jumpi{q_{\rm i}}$, this yields
$$
\jumpi{q_{\rm i}}\Big[\av{{\mathfrak G}_{\rm e}-{\mathfrak G}_{\rm i}+\frac{1}{\eps}(\Pi_{\rm e} -\Pi_{\rm i}}\Big]
+\avi{q_{\rm i}}\Big[\jump{{\mathfrak G}_{\rm e}-{\mathfrak G}_{\rm i}+\frac{1}{\eps}(\Pi_{\rm e}-\Pi_{\rm i})}\Big]=0.
$$
Since $\jumpi{q_{\rm i}}$ and $\avi{q_{\rm i}}$ are two uncorrelated functions of time, this leads us to impose $\jumpi{\Pi_{\rm i}}$ and $\av{\Pi_{\rm i}}$, and therefore $\Pi_{\rm i}^\pm=\pm\frac{1}{2}\big(\jump{\Pi_{\rm i}}\pm2\av{\Pi_{\rm i}}\big)$,
$$
\frac{1}{\eps}\Pi_{\rm i}^\pm=\zeta_{\rm e}^\pm+{\mathfrak G}_{\rm e}^\pm-{\mathfrak G}_{\rm i}^\pm,
$$
where we also used the fact that $\Pi_{\rm e}=\eps\zeta_{\rm e}$.
\end{proof}

We can rewrite the second equation of the Boussinesq equations \eqref{Boussinesq_int} in the interior domain using the hydrodynamic pressure $\Pi_{\rm i}$ introduced in \eqref{defPi} under the form
\begin{equation}\label{momentumbis}
\dt q_{\rm i}+\eps \dx \big( \frac{1}{h_{\rm w}} q_{\rm i}^2 \big)=-\frac{1}{\eps }h_{\rm w}\dx \Pi_{\rm i};
\end{equation}
differentiating this equation with respect to $x$ and substituting $\dt\dx q_{\rm i}=-\ddot\delta$, one obtains a second order elliptic equation for ${\Pi}_{\rm i}$, while Corollary \ref{coroBC} provides non homogeneous Dirichlet boundary conditions. The resolution of this elliptic boundary value problem is straightforward and fully determines $\Pi_{\rm i}$.
\begin{proposition}\label{propPint}
The hydrodynamic interior pressure $\Pi_{\rm i}$ is the unique solution of the elliptic problem
\begin{equation}
\begin{cases}
-\partial_x ( \frac{1}{\varepsilon} h_{\rm w} \partial_x {\Pi}_{\rm i} ) = - \ddot \delta 
+\varepsilon \partial_x^2 \big( \frac{1}{h_{\rm w}} q_{\rm i}^2 \big) ,
\\
\frac{1}{\eps}{\Pi_{\rm i}}_{\vert_{x=\pm \ell}}=\zeta_{\rm e}^\pm + {\mathfrak G}_{\rm e}^\pm-  {\mathfrak G}_{\rm i}^\pm
\end{cases}
\end{equation}
where $h_{\rm w}(t,x)=h_{\rm eq}(x)+\eps\delta(t)$  and ${\mathfrak G}$ is as in \eqref{defG}.
\end{proposition}

\subsubsection{An equation for $\avi{q_{\rm i}}$}\label{sectavq}
We have already seen that in the interior region, the equation for the conservation of mass in the fluid shows that the discharge is given by $q_{\rm i}=-x\dot\delta+\avi{q_{\rm i}}$. The following proposition shows that $\avi{q_{\rm i}}$ is determined by an ODE with a source term related to the wave field in the exterior domain. Note that ${\mathfrak G}_{\rm e}$ accounts for the contribution of the nonlinear and of the dispersive terms of this exterior wave field.
\begin{proposition}\label{propevolav}
Assume that $h_{\rm eq}$ is an even function. Then if $\Pint$ and $q_{\rm i}$ solve the interior fluid equations \eqref{Boussinesq_int} and if the interior pressure satisfies the boundary conditions given in Corollary \ref{coroBC}, then $\av{q_{\rm i}}$ satisfies the ODE
\begin{equation}\label{ODEavq}
 \alpha(\eps\delta) \frac{d}{dt}\avi{q_{\rm i}}+ \eps \alpha'(\eps\delta) \dot\delta \avi{q_{\rm i}}=
 - \frac{1}{2\ell} \jump{\zeta_{\rm e}+{\mathfrak G}_{\rm e}}
\end{equation}
with ${\mathfrak G}$ as in \eqref{defG} and
\begin{equation}\label{defalpha}
\alpha(\eps\delta)=\frac{1}{2\ell}\int_{-\ell}^\ell \frac{1}{h_{\rm eq}(x)+\eps\delta}{\rm d}x, \quad\mbox{ and }\quad
\alpha'(\eps\delta)=-\frac{1}{2\ell}\int_{-\ell}^\ell \frac{1}{(h_{\rm eq}(x)+\eps\delta)^2}{\rm d}x.
\end{equation}  
\end{proposition}
\begin{remark}
The assumption that the bottom parametrization is symmetric with respect to the vertical axis $\{x=0\}$ simplifies the computations but is not necessary. It could be handled as in \cite{Lannes_float} for the hyperbolic ($\mu=0$) case.
\end{remark}
\begin{proof}
Let us first state some relations that will be used throughout this proof and that  can easily be deduced from   the first equation of 
\eqref{Boussinesq_int},
$$
q_{\rm i}=-x \dot\delta +\avi{q_{\rm i}}, \qquad  \dx^2\dt q_{\rm i}=0,\qquad
 \jumpi{q_{\rm i}^2}=-4\ell \dot\delta\avi{q_{\rm i}},\qquad
 \frac{1}{2\ell}\int_{-\ell}^\ell \frac{q_{\rm i}}{h_{\rm w}}=\alpha \avi{q_{\rm i}},
$$
the last relation stemming from the assumption that the bottom of the object is symmetric with respect to the vertical axis $\{x=0\}$. \\
Rewriting the momentum equation as in \eqref{momentumbis}, namely,
$$
\dt q_{\rm i}+\eps \dx \big( \frac{1}{h_{\rm w}} q_{\rm i}^2 \big)=-\frac{1}{\eps }h_{\rm w}\dx \Pi_{\rm i};
$$
dividing by $h_{\rm w}=h_{\rm eq}+\eps\delta$ and integrating between $-\ell$ and $\ell$ one obtains
\begin{align*}
2\ell \alpha \frac{d}{dt} \!\avi{q_{\rm i}}+\eps \jumpi{ \frac{q_{\rm i}^2}{h_{\rm w}^2} }+\eps \int_{-\ell}^\ell \frac{\dx h_{\rm w}}{h_{\rm w}^3}  q_{\rm i}^2 =-\frac{1}{\eps}\jumpi{\Pi_{\rm i}}.
\end{align*}
Using the relations derived above, this gives
$$
2\ell \alpha \frac{d}{dt} \!\avi{q_{\rm i}}- 4\eps \ell \avi{\frac{1}{h_{\rm w}^2}}\dot\delta \avi{q_{\rm i}}-2 \eps \big( \int_{-\ell}^\ell \frac{x\dx h_{\rm w}}{h_{\rm w}^3}\big)\dot\delta \avi{q_{\rm i}}  =-\frac{1}{\eps}\jumpi{\Pi_{\rm i}}.
$$
The result then follows upon remarking that (see Corollary \ref{coroBC}, Remark \ref{remGi} and use the fact that $\zeta_{\rm eq}$ is even)
$$
\jumpi{\frac{1}{\eps}\Pi_{\rm i}}=\jump{\zeta_{\rm e}+{\mathfrak G}_{\rm e} }+ 2\eps\ell  \avi{\frac{1}{h_{\rm w}^2}}\dot\delta\av{q_{\rm i}}
$$
and (after integration by part)
$$
\alpha'(\eps\delta)= -\avi{\frac{1}{h_{\rm w}^2}}-   \frac{1}{\ell}\int_{-\ell}^\ell \frac{x\dx h_{\rm w}}{h_{\rm w}^3}.
$$
\end{proof}

\subsubsection{Reformulation of the equation for the solid motion}\label{sectdelta}

We recall that the solid motion is governed by Newton's equation that can be put under the form \eqref{Newton-z_G-ND}, namely
$$
\tau_{\rm buoy}^2   \ddot{\delta} + \delta= \displaystyle \frac{1}{\eps}\frac{1}{2\ell}\int_{-\ell}^\ell \Pi_{\rm i}(t,x){\rm d}x.
$$
Now that the interior hydrodynamical pressure $\Pi_{\rm i}$ is fully determined by Proposition \ref{propPint}, it is possible to rewrite this equation in a more explicit form, namely, a second order nonlinear ODE on $\delta$ with a source term coming from the exterior wave field.
\begin{proposition}\label{propdelta}
Assume that $h_{\rm eq}$ is an even function. For smooth enough solutions of the wave-structure equations \eqref{Newton-z_G-ND}, \eqref{contrainte1}-\eqref{coupling-q-i-e} and \eqref{CBpress}, the displacement   $\delta$ of the floating object solves the  ODE
\begin{equation}\label{ODEdelta}
{{{\tau}_\mu(\eps\delta)^2}}\ddot\delta+\delta- \eps \beta(\eps\delta)\dot\delta^2 - \frac{\eps}{2}\alpha'(\eps\delta)\avi{q_{\rm i}}^2=\av{\zeta_{\rm e}+{\mathfrak G}_{\rm e}},
\end{equation}
with ${\mathfrak G}$ as in \eqref{defG} and $\alpha'(\eps\delta)$ as in Proposition \ref{propevolav}, and where {${\tau}_\mu(\eps\delta)$ and  $\beta(\eps\delta)$ } are  given by
\begin{align}
\label{defmm}
{{{\tau}_\mu(\eps\delta)^2}}&= \tau_{\rm buoy}^2 +\frac{1}{2\ell}\int_{-\ell}^\ell \frac{x^2}{h_{\rm eq}(x)+\eps\delta}{\rm d}x +\frac{1}{3}\mu \av{\frac{1}{h_{\rm eq}+\eps\delta}}, \\
\label{defbeta}
\beta(\eps\delta)&= \frac{1}{2} \frac{1}{2\ell}\int_{-\ell}^\ell \frac{x^2}{(h_{\rm eq}(x)+\eps\delta)^2}{\rm d}x.
\end{align} 
\end{proposition}
\begin{remark}\label{remaddedmass}
Recalling that $2\pi\tau_{\rm buoy}$ is the dimensionless buoyancy period defined through
$$
\tau_{\rm buoy}^2=\frac{h_0^2}{L^2}m,
$$
where $m$ is the dimensionless mass (see Appendix \ref{appND}), one can write \eqref{defmm} under the form
$$
{\tau}_\mu(\eps\delta)^2=\frac{h_0^2}{L^2}\big(m+m_{\rm a}(\eps\delta) \big)
$$
where $m_{\rm a}(\eps\delta)$ acts as an added mass,
\begin{equation}\label{addedmass}
m_{\rm a}(\eps\delta)=\frac{L^2}{2\ell h_0^2}\int_{-\ell}^\ell \frac{x^2}{h_{\rm eq}(x)+\eps\delta}{\rm d}x+\frac{1}{3}\av{\frac{1}{h_{\rm eq}+\eps\delta}}.
\end{equation}
The buoyancy period is therefore affected by  the {\it added mass phenomenon}, that is, by the fact that when it moves in a fluid, a solid not only has to accelerate its own mass but also the mass of the fluid around it. One can check from \eqref{addedmass} that, in shallow water, the added mass can actually be larger than the proper mass of the solid, a fact that has been noticed in ocean engineering \cite{Zhou}. One deduces from \eqref{defmm} that the added mass effect increases the value of the buoyancy period.\\
Note also that the last term in \eqref{defmm} is due to the presence of the dispersive term in the equations. This is not the only contribution of dispersion to the added mass effect. As we shall see later (see Remark \ref{remark-coupling}), dispersion induces a qualitatively new added mass effect in the form of a coupling with the equation on $\av{q_{\rm i}}$.
\end{remark}

\begin{proof}
For the sake of conciseness, we use here the notation $\fint_{-\ell}^\ell=\frac{1}{2\ell}\int_{-\ell}^\ell f$.  Newton's equation \eqref{Newton-z_G-ND} can be written
\begin{align}
\nonumber
\tau_{\rm buoy}^2  \ddot{\delta} + \delta&= \displaystyle \frac{1}{\eps}\frac{1}{2\ell}\int_{-\ell}^\ell \Pi_{\rm i}(t,x){\rm d}x \\
\label{train}
&=\displaystyle -\frac{1}{\eps}\fint_{-\ell}^\ell x \dx \Pi_{\rm i}(t,x){\rm d}x+ \avi{\frac{1}{\eps}\Pi_{\rm i}},
\end{align}
where we used an integration by parts to derive the second equation. In order to compute the integral in the right-hand-side, let us remark that from Proposition \ref{propPint} we get
$$
-\frac{1}{\eps}\dx \big( h_{\rm w} \dx \Pi_{\rm i}\big)=-\ddot\delta +\eps \dx^2 \big( \frac{1}{h_{\rm w}}q_{\rm i}^2\big).
$$
Integrating this relation, there is a constant $c_0$ such that
$$
-\frac{1}{\eps}\dx \Pi_{\rm i}=- \frac{x}{h_{\rm w}}\ddot\delta+\eps \frac{1}{h_{\rm w}}\dx  \big( \frac{1}{h_{\rm w}}q_{\rm i}^2\big)+\frac{c_0}{h_{\rm w}}.
$$
Recalling that $h_{\rm eq}$ (and therefore $h_{\rm w}=h_{\rm eq}+\eps\delta$) is an even function, and using \eqref{exprqi}, we get
\begin{align*}
-\frac{1}{\eps}\fint_{-\ell}^\ell x \dx \Pi_{\rm i}&=- \Big(\fint_{-\ell}^\ell \frac{x^2}{h_{\rm w}}\Big) \ddot\delta+\eps \fint_{-\ell}^\ell \frac{x}{h_{\rm w}}\dx  \big( \frac{1}{h_{\rm w}}q_{\rm i}^2\big)\\
&=- \Big(\fint_{-\ell}^\ell \frac{x^2}{h_{\rm w}}\Big) \ddot\delta+\eps \Big(- \fint_{-\ell}^\ell \frac{x^2}{h_{\rm w}}\dx\big(\frac{x}{h_{\rm w}}\big)+ \avi{\frac{x^2}{h_{\rm w}^2}}\Big)\dot\delta^2\\
&+\eps \Big(- \fint_{-\ell}^\ell \frac{1}{h_{\rm w}}\dx\big(\frac{x}{h_{\rm w}}\big)+ \avi{\frac{1}{h_{\rm w}^2}}\Big)\avi{q_{\rm i}}^2.
\end{align*}
We now need the following lemma.
\begin{lemma}\label{lemma-mbeta}
The following identities hold (with $h_{\rm w}=h_{\rm eq}+\eps\delta$)
\begin{align*}
- \fint_{-\ell}^\ell \frac{x^2}{h_{\rm w}}\dx\big(\frac{x}{h_{\rm w}}\big){\rm d}x+ \frac{1}{2}\avi{\frac{x^2}{h_{\rm w}^2}}
&=\beta(\eps\delta)\\
- \fint_{-\ell}^\ell \frac{1}{h_{\rm w}}\dx\big(\frac{x}{h_{\rm w}}\big){\rm d}x+\frac{1}{2} \avi{\frac{1}{h_{\rm w}^2}}&=\frac{1}{2}\alpha'(\eps\delta).
\end{align*}
\end{lemma}
\begin{proof}[Proof of the lemma]
For the first identity, one just has to remark that 
\begin{align*}
- \fint_{-\ell}^\ell \frac{x^2}{h_{\rm w}}\dx\big(\frac{x}{h_{\rm w}}\big)&=-\fint_{-\ell}^\ell \frac{x^2}{h_{\rm w}^2}
-\frac{1}{2}\fint_{-\ell}^\ell x^3 \dx \big( \frac{1}{h_{\rm w}^2}\big)\\
&=\frac{1}{2}\fint_{-\ell}^\ell \frac{x^2}{h_{\rm w}^2}-\frac{1}{2}\av{\frac{x^2}{h_{\rm w}^2}},
\end{align*}
the last identity stemming from an integration by parts. \\
For the second identity, since $\alpha'(\eps\delta)=-\fint_{-\ell}^\ell \frac{1}{h_{\rm w}^2}$, we just have to remark that
\begin{align*}
- \fint_{-\ell}^\ell \frac{1}{h_{\rm w}}\dx\big(\frac{x}{h_{\rm w}}\big)&=-\fint_{-\ell}^\ell
\frac{1}{h_{\rm w}^2}-\frac{1}{2}\fint_{-\ell}^\ell x \dx \big( \frac{1}{h_{\rm w}^2}\big)\\
&=-\frac{1}{2}\fint_{-\ell}^\ell
\frac{1}{h_{\rm w}^2}-\frac{1}{2}\av{\frac{1}{h_{\rm w}^2}},
\end{align*}
the last line following from an integration by parts.
\end{proof}

Corollary \ref{coroBC} and Remark \ref{remGi} imply that
$$
\avi{\frac{1}{\eps}\Pi_{\rm i}}=\av{\zeta_{\rm e}+{\mathfrak G}_{\rm e}}-\eps \avi{ \frac{1}{2}\frac{\avi{q_{\rm i}}^2+\ell^2 \dot\delta^2}{h_{\rm w}^2}}-\frac{1}{3}\mu\av{\frac{1}{h_{\rm w}}} \ddot\delta,
$$
so that we can deduce from \eqref{train} and the lemma that
\begin{align*}
\big( \tau_{\rm buoy}^2  +\fint_{-\ell}^\ell \frac{x^2}{h_{\rm w}} +\frac{1}{3}\av{\frac{1}{h_{\rm w}}}\mu  \big) \ddot{\delta} +\delta 
=
\eps
\beta \dot\delta^2 +\eps \frac{1}{2} \alpha'\avi{q_{\rm i}}^2 +\av{\zeta_{\rm e}+{\mathfrak G}_{\rm e}},
\end{align*}
which is the result stated in the proposition.
\end{proof}

\section{Wave-structure interaction as a transmission problem}\label{sectWSITP}

Taking advantage of the analysis performed in the previous section, our aim here is to formulate the wave-structure interaction equations under the form of a transmission problem and to study this latter. The transmission problem, formed by the Boussinesq-Abbott equations in both components of the exterior domain coupled with transmission conditions involving forced ODEs on $\delta$ and $\av{q_{\rm i}}$, is made explicit in \S \ref{sectderivWS}. A toy model for this transmission problem (with more standard transmission conditions) if then proposed in \S \ref{sectStudytoy}; based on this analysis, a first reformulation of the wave-structure transmission problem is performed in \S \ref{sectReform}, exhibiting in particular a nontrivial contribution of the dispersive terms to the added mass phenomenon. In \S \ref{sectreducODE}, a second reformulation is proposed, in which we show that the whole system can be recast as an ODE; taking advantage of this structure, we show that the wave-structure equations are well-posed. The existence time thus obtained is however not uniform with respect to the dispersive parameter $\mu$; we therefore address in \S \ref{sectUE} the issue of proving uniform estimates and establish a conditional uniform estimate as well as uniform estimates for equations linearized around non trivial states. To this end, we exhibit a new hidden regularity phenomenon granted by the dispersive terms.

\medbreak

\noindent
We shall use the following notations throughout this section.
\noindent
\begin{notation}\label{notapart3}
\item[ -] For the sake of clarity we simply write $f$ instead of $f_{\rm e}$ when dealing with the restriction of a function $f$ to the exterior domain $\cE$. To avoid any confusion, we still keep the subscript and write $f_{\rm i}$ for the restriction to the interior domain ${\mathcal I}$.
\item[ -] Dispersive boundary layers play a central role in the analysis performed in this section. Since their decay rate is $\sqrt{\mu/3}$, it is convenient to introduce the parameter $\kappa$ as
$$
\kappa=\sqrt{\frac{\mu}{3}}.
$$
\item[ -] We shall denote by ${\mathfrak f}_{\rm sw}$ the momentum flux associated with the shallow water equations, namely,
\begin{equation}\label{defmfsw}
{\mathfrak f}_{\rm sw}=\frac{h^2-1}{2\eps}+\eps \frac{q^2}{h}=\zeta +\eps\big(\frac{1}{2}\zeta^2+\frac{q^2}{h}\big),
\end{equation}
so that the Boussinesq-Abbot equations \eqref{Abbott} in the exterior domain can be written in more compact form
$$
\begin{cases}
\partial_t \zeta +\partial_x q =0 
\\
(1-\kappa^2 \partial_x^2) \partial_t q + \dx \mfsw
= 0,
\end{cases}\quad\mbox{ for }\quad t\geq 0,\quad x\in {\mathcal E}.
$$
\end{notation}

\subsection{Derivation of a wave-structure transmission problem}\label{sectderivWS}

Recalling that the interior discharge is given by $q_{\rm i}(t,x)=-x \dot\delta(t)+\av{q_{\rm i}}(t)$, the continuity condition \eqref{coupling-q-i-e} on the discharge can be equivalently written under the form 
$$
\jump{q}=-2\ell \dot\delta \quad \mbox{ and }\quad \av{q}=\av{q_{\rm i}},
$$
where we recall that the jump $\jump{\cdot}$ and average $\av{\cdot}$ are defined in \eqref{jumps} and \eqref{averages}.
The analysis performed in Section \ref{sectderiv} shows that the wave-structure equations \eqref{Boussinesq}-\eqref{addcoupling} can be reduced to a transmission problem for the Boussinesq-Abbott system written on both components of the exterior domain $\cE$. This is summarized in the following theorem.  
\begin{theorem}\label{theotransm}
Assume that $\zeta_{\rm eq}$ is an even function and let $\mfsw$ be as in \eqref{defmfsw}. For smooth enough solutions, the resolution of the wave-structure equations \eqref{Newton-z_G-ND}, \eqref{contrainte1}-\eqref{coupling-q-i-e} and \eqref{CBpress}  is equivalent to the resolution of the standard Boussinesq-Abbott system
\begin{equation}\label{Boussinesq_ext2}
\begin{cases}
\partial_t \zeta +\partial_x q =0 
\\
(1-\kappa^2 \partial_x^2) \partial_t q +\dx \mfsw
= 0,
\end{cases}\quad\mbox{ for }\quad t\geq 0,\quad x\in {\mathcal E},
\end{equation}
on both components of the exterior domain ${\mathcal E}$ and
with transmission conditions
\begin{equation}
\label{trans1}
\av{q}=\av{q_{\rm i}} \quad \mbox{ and }\quad \jump{q}=-2\ell \dot\delta,
\end{equation}
and where $\av{q_{\rm i}}$ and $\delta$ solve
\begin{align}
\label{trans2}
 \alpha(\eps\delta) \frac{d}{dt}\avi{q_{\rm i}}+ \eps \alpha'(\eps\delta) \dot\delta \avi{q_{\rm i}} &=
 - \frac{1}{2\ell} \jump{\zeta+{\mathfrak G}}
,\\
 \label{trans3}
  {{\tau_\mu(\eps\delta)^2\ddot\delta+\delta-\eps \beta(\eps\delta)\dot\delta^2 -\eps \frac{1}{2}\alpha'(\eps\delta)\avi{q_i}^2 }}&=\av{\zeta+{\mathfrak G}},
\end{align}
where we recall that 
$$
{\mathfrak G}=\eps \frac{1}{2}\frac{q^2}{h^2}-\kappa^2\frac{1}{h}\dx\dt q,
$$
and that $\alpha(\eps\delta)$ is as in Proposition \ref{propevolav}, and ${{\tau_\mu(\eps\delta)}}$ and $\beta(\eps\delta)$ as in Proposition \ref{propdelta}.
\end{theorem}

 
 The energy of the fluid in the exterior domain, associated with \eqref{Boussinesq_ext2} is
 \begin{equation}\label{defEext}
 {\mathfrak E}_{\rm ext} = \frac{1}{2}\int_{\mathcal{E}} \Big( \zeta^2+\frac{1}{h}q^2 +\kappa^2 \frac{1}{h}(\dx q)^2\Big),
 \end{equation}
 and we also introduce an "interior energy" that depends only on ${\mathtt Z}=(\av{q_{\rm i}},\delta,\dot\delta)$,
 \begin{equation}\label{defEint}
 {\mathfrak E}_{\rm int}=\ell\big( \delta^2 + \tau_{\mu}(\eps \delta)^2 \dot\delta^2 + \alpha(\eps \delta) \av{q_{\rm i}}^2 \big).
 \end{equation}
They satisfy the following energy estimate in which we do not seek to close the estimate by providing a control of the residual term; this more delicate issue is addressed in \S \ref{sectUE} below.
 \begin{proposition}\label{prior-conservation-case}
Under the assumptions of Theorem \ref{theotransm}, the following energy estimate holds,
 $$
 \frac{d}{dt}\big[{\mathfrak E}_{\rm ext} +{\mathfrak E}_{\rm int}\big]+
 \eps\kappa^2 \ell \av{ \frac{1}{(h_{\rm eq}+\eps\delta)^2} } \dot\delta^3
 =3\eps\kappa^2 \int_{\cE} {\mathfrak R},
 $$
 where we recall that ${\mathfrak R}= \frac{1}{6h^2}(\dx q)^3 +\frac{1}{3h^2}q (\dt\dx q) \dx \zeta$.
 \end{proposition}
 
 \begin{proof}
 There are two ways to derive the energy estimate of the proposition. The first one consists in multiplying the two equations of  \eqref{Boussinesq_ext2} by $\zeta$ and $q$ respectively and integrating by parts, and multiplying \eqref{trans2} and \eqref{trans3} by $\av{q_{\rm i}}$ and $\dot\delta$ respectively, and adding the resulting identities. The second method is to deduce it from the approximate conservation of the total energy established in Corollary \ref{coroBC}, namely,
 \begin{equation}\label{NRJuti}
\frac{d}{dt}{\mathfrak E}_{\rm tot}=3\eps\kappa^2 \big(\int_{{\mathcal I}} {\mathfrak R}+\int_{\cE} {\mathfrak R} \big),
 \end{equation}
 and ${\mathfrak E}_{\rm tot}$ can be written as 
 $$
 {\mathfrak E}_{\rm tot}={\mathfrak E}_{\rm ext}+\int_{\mathcal I}{\mfe}_{\rm i} +{\mathfrak E}_{\rm solid}
 \quad\mbox{ with }\quad
 {\mathfrak E}_{\rm solid}=2\ell\big( \frac{1}{\eps} m \delta +\frac{1}{2}\tau_{\rm buoy}^2\dot\delta^2\big),
 $$
 where we recall that ${\mfe}=\frac{1}{2}\zeta^2 +\frac{1}{2}\frac{1}{h}q^2 +\kappa^2 \frac{1}{2h}(\dx q)^2$. Since in the interior region, one has $q_{\rm i}=-x\dot\delta+\av{q_{\rm i}}$ and $\zeta_{\rm i}=\zeta_{\rm w}$ with $\zeta_{\rm w}$ given by \eqref{eqzw}, namely, $\zeta_{\rm w}=\delta+\frac{1}{\eps}(h_{\rm eq}-1)$, one deduces
 \begin{align*}
 \int_{{\mathcal I}} {\mathfrak R}&=
 \frac{d}{dt}\big[\big( \int_{-\ell}^\ell \frac{1}{6 \eps h_{\rm w}}- \frac{\ell}{3\eps}\av{\frac{1}{h_{\rm w}}}\big) \dot\delta^2\big]-\frac{\ell}{3}\av{\frac{1}{h_{\rm w}^2}} \dot\delta^3,\\
  \int_{{\mathcal I}}{\mfe}_{\rm i}&= \frac{(h_{\rm eq}-1)^2}{2\eps^2} -2\ell \frac{1}{\eps}m \delta  +\ell \delta^2+\av{q_{\rm i}}^2\int_{-\ell}^\ell \frac{1}{2h_{\rm w}}+\kappa^2 \dot\delta^2 \int_{-\ell}^\ell \frac{1}{2h_{\rm w}},
 \end{align*}
 where we used Archimedes' principle \eqref{eqArchimede} for the second term in the right-hand side of the second identity.
 Plugging these identities into \eqref{NRJuti}  yields the result.
\end{proof}

\subsection{Study of a general transmission problem for the Boussinesq-Abbott system}\label{sectStudytoy}

Before addressing the transmission problem derived in the previous section, where the transmission conditions involve ODEs that are coupled with the solution of the transmission problem itself, it is instructive to study a simpler, yet quite general, transmission problem, where the transmission conditions are given in terms of known functions. More precisely,  we consider in this section the Boussinesq-Abbott equations
\begin{equation}\label{Boussinesq_ext2-presc}
\begin{cases}
\partial_t \zeta +\partial_x q =0 
\\
(1-\kappa^2 \partial_x^2) \partial_t q + \dx \mfsw
= 0,
\end{cases}\quad\mbox{ for }\quad t\geq 0,\quad x\in {\mathcal E}
\end{equation}
(with $\kappa^2=\mu/3$ and $\mfsw$ as in \eqref{defmfsw}) on both components of the exterior domain ${\mathcal E}$ and
with transmission conditions
\begin{equation}\label{trans-presc}
\av{q} =f
\quad \text{and} \quad
\jump{q} = 2g,
\end{equation}
where $f,g \in C^1( \mathbb{R}^+)$ are known functions. 
\begin{remark}\label{remCauchy}
One can see the boundary value problem on the half-line $(\ell,\infty)$ 
$$
\begin{cases}
\partial_t \zeta +\partial_x q =0 
\\
(1-\kappa^2 \partial_x^2) \partial_t q + \dx \mfsw
= 0,
\end{cases}\quad\mbox{ for }\quad t\geq 0,\quad x>\ell,
$$
with boundary condition $q_{\vert_{x=\ell}}=f$ as a particular case of the transmission problem \eqref{Boussinesq_ext2-presc}-\eqref{trans-presc}. Indeed, it suffices to extend $\zeta$ and $q$ as $\zeta(t,-x)=\zeta(t,x)$ and $q(t,-x)=-q(t,x)$ for all $x>\ell$ and to take $g=0$. The associated initial boundary value problem has been considered in \cite{Johnston} in the linear case ($\eps=0$, so that $\mfsw=\zeta$) using Fokas' unified transform method. This initial boundary value problem has also been considered both theoretically and numerically, but with a boundary condition on $\zeta$ rather than $q$, in \cite{Lannes-Weynans}.
\end{remark}

In the case where $\mu =0$ (the shallow-water equation), it is well known that the initial boundary value problem associated with \eqref{Boussinesq_ext2-presc}-\eqref{trans-presc} is locally well-posed in $H^n(\cE)\times H^n(\cE)$ ($n\geq 2$) provided that $n$ compatibility conditions are satisfied (see for instance \cite{IguchiLannes} or the lecture notes \cite{LannesBressanone}). The presence of the dispersive term $- \frac{\mu}{3} \partial_x^2 \partial_t$ makes things different; as observed in  \cite{BLM,Lannes-Weynans} in related situations, a single compatibility condition is enough to obtain a regular solution because dispersion smoothes the solution by  creating a dispersive boundary layer of order $O(\kappa)$.

Indeed, it is possible to reduce \eqref{Boussinesq_ext2-presc}-\eqref{trans-presc} to an ODE. To perform this, it is necessary to introduce the regularizing operators $R_0$ and $R_1$ defined as the inverses of $(1-\kappa^2 \dx^2)$ with homogeneous Dirichlet and Neumann data respectively at $x=\pm \ell$, that is,
\begin{equation}\label{R0R1definition}
 R_0 f= u \quad\mbox{ and }\quad  R_1 f=v,
\end{equation}
where
\begin{equation}\label{R0R1def}
\begin{cases} (1-\kappa^2 \dx^2) u =f,\\ u_{\vert_{x=\pm \ell}}=0, \end{cases}
\quad \mbox{ and }\quad
 \begin{cases} (1-\kappa^2 \dx^2) v =f,\\ (\dx v)_{\vert_{x=\pm \ell}}=0. \end{cases}
\end{equation}
In the statement below, we denote 
$$
\HH^n=H^{n+1}(\cE)\times H^{n+2} (\cE).
$$
\begin{proposition}\label{existence-presc}
Let $f,g\in C^1(\RR^+)$, $n \in {\mathbb N}$, and $U^{\rm in}=(\zeta^{\rm in}, q^{\rm in}) \in \HH^n $ be such that
 $$
  \inf  (1 + \varepsilon \zeta^{\rm in})>0,\qquad 
  \av{q^{\rm in}} =f(0)
\quad \text{and} \quad
\jump{q^{\rm in}} = 2g(0).
  $$
  Then for all $\kappa > 0$, there is $T > 0$ such that the system \eqref{Boussinesq_ext2-presc}-\eqref{trans-presc} has a unique solution $(\zeta,q)\in C^1([0,T[ ; \HH^n)$ with initial data $U^{\rm in}$.
\end{proposition}
\begin{proof}
The key ingredient of the proof is to reformulate the problem as an ODE. 
\begin{lemma}\label{propODE-presc}
Let $T>0$, $(f,g) \in C^1([0,T])$ and $U=(\zeta, q) \in C^1 ([0, T] ; \HH^0)$ be such that $\inf_{[0,T]\times \cE} 1+\eps\zeta>0$, and such that the transmission conditions are initially satisfied 
$$
\av{ q_{ | t = 0} } = f(0)
\quad \text{and} \quad
\jump{ q_{ | t = 0} } = 2 g(0).
$$
Then $U$ solves \eqref{Boussinesq_ext2-presc}-\eqref{trans-presc} if and only it solves
\begin{equation}\label{formODE-presc}
\begin{cases}
\dt \zeta&=-\dx q, \\
\dt q     &=  -  \dx R_1   \mfsw + \big( \dot f \pm \dot g \big) e ^{ -  \frac{1}{\kappa} \abs{x\mp \ell} }.
\end{cases}
\end{equation}
\end{lemma}
\begin{proof}[Proof of the lemma]
Recalling that  $R_0$ is the inverse of $(1 - \kappa^2 \dx^2)$ with Dirichlet boundary conditions on each side of $\mathcal E$, the second equation of \eqref{Boussinesq_ext2-presc} 
is equivalent to 
\begin{equation}\label{altq-presc}
\dt q     =  -  R_0  \dx\mfsw    + (\dot f \pm \dot g)  e ^{ -  \frac{1}{\kappa} \abs{x\mp \ell} } \quad \mbox{ on }\quad \cE^\pm.
\end{equation}
Using the fact that $R_0\dx =\dx R_1$ we obtain the expected equation \eqref{formODE-presc}. The only thing left to prove is therefore that if \eqref{altq-presc} is satisfied, and if the transmission condition \eqref{trans-presc} holds at $t=0$, then it holds for all time. This is obvious after remarking that one readily gets from \eqref{altq-presc} that
$$
\frac{d}{dt}\av{q}=\dot f\quad \mbox{ and } \quad \frac{d}{dt}\jump{q}=2\dot g.
$$
\end{proof}
Proving that \eqref{formODE-presc} is actually an ODE also requires the following lemma which can classically be established by multiplying both equations of \eqref{R0R1def} by $u$ and integrating by parts.
\begin{lemma}\label{lemmaRbounded}
The operators $R_\ell$, $\kappa \partial_x R_\ell$ and $ \frac{\kappa^2}{2} \partial_x^2 R_\ell$ for $\ell=0,1$ are bounded operators on $L^2(\mathcal{E})$, with operator norm smaller than one.
\end{lemma}  
%

Let $\mathcal O$ denote the open subset of $\HH^n$ of the $U = (\zeta, q)$ such that $\inf_\cE  (1+ \varepsilon \zeta) >0$. 
Let us also write \eqref{formODE-presc} in compact form as
 $$
 \frac{d}{dt} U = \Phi (U),
 $$
 where $\Phi=(\phi_1,\phi_2)$ and
 \begin{align*}
 \phi_1= -\dx q,\qquad 
 \phi_2= -  \dx R_1  \mfsw + \big( \dot f\pm \dot g \big) e ^{ -  \frac{1}{\kappa} \abs{x\mp \ell} }.
  \end{align*}
Since $\dx  R_1: H^n \to H^{n+1}$ is a bounded operator {{(as a consequence of Lemma \ref{lemmaRbounded})}}, we can deduce  from standard trace and product estimates in Sobolev spaces that $\Phi$ is a smooth mapping from ${\mathcal O}$ to $\HH^n$ and the local existence follows from Cauchy-Lipschitz theorem.
\end{proof}

\subsection{Reformulation of the wave-structure transmission problem}\label{sectReform}

The quantities $\jump{{\mathfrak G}}$ and $\av{{\mathfrak G}}$ that appear as source terms  in the differential equations \eqref{trans2} and \eqref{trans3} for $\av{q_{\rm i}}$ and $\delta$ depend themselves on these two terms; indeed, in order to compute  $\jump{{\mathfrak G}}$ and $\av{{\mathfrak G}}$, one must solve the transmission problem  \eqref{Boussinesq_ext2}-\eqref{trans1} in which the transmission conditions are given in terms of $\av{q_{\rm i}}$ and $\delta$. \\
 In the nondispersive case ($\mu=0$, shallow water equations), this dependence is of lower order and  $\jump{{\mathfrak G}}$ and $\av{{\mathfrak G}}$ can be treated as source terms in the ODEs for  $\av{q_{\rm i}}$ and $\delta$ (see Remark \ref{remark-coupling} below). A new phenomenon appears in the presence of dispersion: these quantities contain leading order terms in the differential equations for $\delta$ and $\av{q_{\rm i}}$. As for the added mass coefficient, they cannot therefore be treated as source terms, both theoretically and numerically. This issue is addressed in the following theorem where we essentially show that $\jump{{\mathfrak G}}$ and $\av{{\mathfrak G}}$ can be decomposed as the sum of explicit leading order terms and  lower order terms $\jump{{\mathfrak H}}$ and $\av{{\mathfrak H}}$ that can be treated as source terms. We recall that we denote respectively by $R_0$ and $R_1$ the inverses of $(1-\kappa^2 \dx^2)$ with homogeneous Dirichlet and Neumann data  at $x=\pm \ell$ (see (\ref{R0R1definition}, \ref{R0R1def})).

\begin{theorem}\label{theoIBVP}
Assume that $\zeta_{\rm eq}$ is an even function and let $\mfsw$ be as in \eqref{defmfsw}. For smooth enough solutions, the resolution of the wave-structure equations \eqref{Newton-z_G-ND}, \eqref{contrainte1}-\eqref{coupling-q-i-e} and \eqref{CBpress}  is equivalent to the resolution of the standard Boussinesq-Abbott system
\begin{equation}\label{Boussinesq_ext3}
\begin{cases}
\partial_t \zeta +\partial_x q =0 
\\
(1-\kappa^2\partial_x^2) \partial_t q + \dx \mfsw
= 0,
\end{cases}\quad\mbox{ for }\quad t\geq 0,\quad x\in {\mathcal E},
\end{equation}
on both components of the exterior domain ${\mathcal E}$ and
with transmission conditions
\begin{equation}
\label{trans1reform}
\av{q}=\av{q_{\rm i}} \quad \mbox{ and }\quad \jump{q}=-2\ell \dot\delta,
\end{equation}
and where $\av{q_{\rm i}}$ and $\delta$ solve the coupled system of ODEs
\begin{equation}\label{ODEIBVP}
{\mathfrak T}_\mu(\eps\delta,\eps\zeta_\pm)  \frac{d}{dt}\begin{pmatrix} \av{q_{\rm i} }\\ \dot \delta \end{pmatrix}
+
  \begin{pmatrix}
 \eps {{\alpha'(\eps\delta) \dot\delta \avi{q_{\rm i}}}} \\
\delta -\eps\big(  {{ \beta(\eps\delta)\dot\delta^2 +\frac{1}{2}\alpha'(\eps\delta)\avi{q_i}^2}} \big)
\end{pmatrix}
=
\begin{pmatrix}
-  \frac{1}{2\ell}\jump{{\mathfrak H}} \\
\av{{\mathfrak H}}
\end{pmatrix}
\end{equation} 
with $\beta(\eps\delta)$ as defined in \eqref{defbeta}, and with ${\mathfrak H}={\mathfrak H}(\zeta,q)$ given by
\begin{equation} \label{defC}
{\mathfrak H}(\zeta,q)=
\frac{1}{2}\eps \left( \frac{1}{h}\zeta^2- \frac{q^2}{h^2}\right) +\frac{1}{h} R_1 \mfsw,
\end{equation}
while ${\mathfrak T}_\mu(\eps\delta,\eps\zeta_\pm) $ is the invertible matrix given by 
\begin{equation}\label{defM}
 {\mathfrak T}_\mu(\eps\delta,\eps\zeta_\pm)  = \begin{pmatrix} \alpha(\eps\delta) +\frac{1}{\ell}\kappa \av{\frac{1}{h}}  & -\frac{1}{2} \kappa \jump{\frac{1}{h}} \\
- \frac{1}{2}\kappa \jump{\frac{1}{h}}  & {\tau_\mu(\eps\delta)^2}+\ell \kappa \av{\frac{1}{h}}
\end{pmatrix},
\end{equation}
where we recall that $\tau_\mu(\eps\delta)$ and $\alpha(\eps\delta)$ are defined in \eqref{defmm} and \eqref{defalpha} respectively.
\end{theorem}
\begin{remark}\label{remark-coupling}
The difference between \eqref{ODEIBVP} and  the evolution equations \eqref{ODEdelta} and \eqref{ODEavq} on $\delta$ and $\av{q_{\rm i}}$, is that a new contribution to the added mass effect (see Remark \ref{remaddedmass}) coming from the dispersive term has been exhibited. It is of interest to note that the dispersive terms not only provide a quantitative contribution to the added mass effect, but also a qualitative one since it induces a new coupling between the equations  on $\av{q_{\rm i}}$ and $\delta$. This is not the case in the non dispersive case where the matrix ${\mathfrak T}_\mu(\eps\delta,\eps\zeta_\pm) $ is diagonal and the equations are only coupled through the nonlinear and source terms, namely,
$$
\begin{cases}
\alpha(\eps\delta)\frac{d}{dt}\av{q_{\rm i}} + \eps \alpha'(\eps\delta)\dot \delta \av{q_{\rm i}}= -\frac{1}{2\ell}\jump{\zeta+\eps\frac{1}{2}\frac{q^2}{h^2} }, \\
\tau_0(\eps\delta)^2\ddot\delta +\delta -\eps\beta(\eps\delta)\dot\delta^2 - \frac{\eps}{2}\alpha'(\eps\delta)\av{q_{\rm i}}^2=\av{\zeta+\eps\frac{1}{2}\frac{q^2}{h^2} }
\end{cases}
$$
(this system can either be derived directly as in \cite{Lannes_float,Tucsnak,Bocchi}, or formally be derived from \eqref{ODEIBVP} by setting $\kappa=0$ and observing that $\lim_{\kappa \to 0} R_1f _{\vert_{x=\pm \ell}}= f(\pm \ell)$).
\end{remark} 

\begin{proof}
Taking into account the transmission conditions \eqref{trans1reform}, the formula \eqref{altq-presc} for $\dt q$ becomes 
\begin{equation}\label{altq}
\dt q     =  -  R_0 \dx \mfsw    + (\frac{d}{dt} \av{q_{\rm i}}\mp\ell\ddot\delta)  e ^{ -  \frac{1}{\kappa} \abs{x\mp \ell} } \quad \mbox{ on }\quad \cE^\pm.
\end{equation}
 We then obtain after differentiating in space, {multiplying by $1/h$} and taking the jump,
$$
\kappa^2\jump{\frac{1}{h}\dt \dx q} = - \kappa^2 \jump{ \frac{1}{h}\dx R_0 \dx\mfsw } - 2 \kappa \av{\frac{1}{h}}\frac{d}{dt}\av{q_{\rm i}} +\ell\kappa \jump{\frac{1}{h}}\ddot\delta. 
$$ 
Therefore the evolution equation \eqref{trans2} on $\av{q_{\rm i}}$ is equivalent to 
  $$
   - \kappa^2 \jump{ \frac{1}{h}\dx R_0 \dx\mfsw } - 2 \kappa \av{\frac{1}{h}}\frac{d}{dt} \av{q_{\rm i}} 
   +\ell\kappa \jump{\frac{1}{h}}\ddot\delta
   = \jump{\zeta+\eps \frac{1}{2} \frac{q^2}{h^2}} +  2\ell \Big( \alpha \frac{d}{dt} \av{q_{\rm i}}+ \eps \alpha' \dot\delta \av{q_{\rm i}}\Big)
  $$
and therefore
$$
 \big( \ell \alpha+\kappa \av{\frac{1}{h}}\big)\frac{d}{dt} \av{q_{\rm i}}
-\frac{1}{2} \ell \kappa \jump{\frac{1}{h}}  \ddot\delta
 +\eps \ell \alpha'(\eps\delta) \dot\delta \av{q_{\rm i}}= - \frac{1}{2}\Big( \kappa^2 \jump{\frac{1}{h} \dx R_0 \dx\mfsw  }  + \jump{ \zeta+\eps \frac{1}{2}\frac{q^2}{h^2} } \Big). 
$$
Remarking further that $R_0\dx=\dx R_1$, where 
$R_1$ is the inverse of $(1 - \kappa^2 \dx^2)$ with Neumann boundary conditions on each side of $\mathcal E$, one can write
\begin{align*}
\kappa^2 \dx R_0 \dx\mfsw&= \frac{\mu}{3}\dx^2 R_1 \mfsw \\
&= -\mfsw  +R_1 \mfsw,
\end{align*}
so that a first  ODE on $\av{q_{\rm i}}$ and $\dot\delta$ is given by
\begin{equation}
\label{dtqi}
\big( \ell \alpha(\eps\delta)+\kappa\av{\frac{1}{h}}\big)\frac{d}{dt} \av{q_{\rm i}} -\frac{1}{2} \ell \kappa \jump{\frac{1}{h}}  \ddot\delta=- \eps \ell \alpha'(\eps\delta)\dot\delta \av{q_{\rm i}} - \frac{1}{2} \jump{{\mathfrak H}},
\end{equation}
with ${\mathfrak H}$ given by 
$$
{\mathfrak H}(\zeta,q)=\zeta+\frac{1}{2}\eps \frac{q^2}{h^2}- \frac{1}{h}  \big( 1-R_1\big) \mfsw,
$$
or equivalently, by the formula given in \eqref{defC}.\\

Similarly, differentiating \eqref{altq} with respect to $x$, multiplying by $1/h$ and taking the average yields
$$
\kappa^2\av{\frac{1}{h}\dt\dx q}=-\kappa^2\av{\frac{1}{h}\dx R_0\dx\mfsw}-\frac{1}{2}\kappa\jump{\frac{1}{h}} \frac{d}{dt} \av{q_{\rm i}}+\ell\kappa \av{\frac{1}{h}}\ddot\delta,
$$
which we can plug into \eqref{ODEdelta} to obtain
\begin{equation}
\label{dt2delta}
\big( {{\tau_\mu(\eps\delta)^2}}+ {\ell \kappa\av{\frac{1}{h}}}\big)\ddot\delta {-\frac{1}{2}\kappa \jump{\frac{1}{h}}\frac{d}{dt}\av{q_{\rm i}}}=-\delta+
\eps \beta(\eps\delta)\dot\delta^2 +\frac{\eps}{2} \alpha'(\eps\delta)\avi{q_{\rm i}}^2
+\av{{\mathfrak H}}.
\end{equation}
The result therefore follows from \eqref{altq}, \eqref{dtqi} and \eqref{dt2delta}.

\end{proof}
\subsection{Reduction to an ODE}\label{sectreducODE}

It was remarked in \cite{BLM} in the case of a fixed structure (and for the simpler Boussinesq system \eqref{varBLM}) that the transmission problem could be reduced to an ODE. We show here that this remains true in the case of a freely floating structure and for the Boussinesq-Abbott system. In the statement below, we assume that $(\zeta,q)\in \HH$ with $\HH=H^1(\cE)\times H^2 (\cE)$; this regularity ensures that the traces of $\zeta$, $q$  and $\dx q$ are well defined at $\pm \ell$. Note also that the condition $\inf_{(t,x)\in [0,T]\times \RR} h(t,x)>0$ means that $\inf_{(t,x)\in [0,T]\times \cE} \{1+\eps\zeta(t,x) \}>0$ and $\inf_{(t,x)\in [0,T]\times {\mathcal I}} \{h_{\rm eq}(x)+\eps \delta(t)  \}>0$; this is therefore a condition on $\zeta$ and on $\delta$.
\begin{proposition}\label{propODE}
For  $U=(\zeta, q) \in C^1 ([0, T] ; \HH)$ and ${\mathtt Z}=(\av{q_{\rm i}},\delta,\dot\delta) \in C^1([0,T];\RR^3)$ such  that $\inf_{[0,T]\times \RR} h>0$,
 and 
$$
\jump{ q_{ \vert_{t = 0}} } = -2\ell \dot\delta(0),\qquad \av{q_{\vert_{t=0}}}=\av{q_{\rm i}}(0),
$$ 
the system \eqref{Boussinesq_ext3}--\eqref{ODEIBVP}  is equivalent to
\begin{equation}\label{formODE}
\begin{cases}
\dt \zeta&=-\dx q, \\
\dt q     &=  -  \dx R_1  \mfsw+ \Big( {\mathcal Q} ({\mathtt Z}, \jump{\mathfrak H})\mp \ell {\mathcal D}({\mathtt Z},\av{{\mathfrak H}}) \Big) e ^{ -  \frac{1}{\kappa} \abs{x\mp \ell} }, \\
\frac{d}{dt}{\mathtt Z}&={\mathcal Z}\big({\mathtt Z},\av{{\mathfrak H}}, \jump{\mathfrak H}\big),
\end{cases}
\end{equation}
where the first two equations are cast on $\cE^\pm$, ${\mathfrak H}$ is defined in \eqref{defC} and  where 
${\mathcal Z}({\mathtt Z},\av{{\mathfrak H}}, \jump{\mathfrak H}):=\big( {\mathcal Q}({\mathtt Z}, \jump{\mathfrak H}),\dot\delta,  {\mathcal D}({\mathtt Z},\av{{\mathfrak H}})\big)^{\rm T}$, with
$$
\left(\begin{array}{c}
 {\mathcal Q} ({\mathtt Z},\jump{\mathfrak H}) \\
 {\mathcal D}({\mathtt Z},\av{{\mathfrak H}})
 \end{array}\right)
 ={\mathfrak T}_\mu(\eps\delta,\eps\zeta_\pm) ^{-1}\begin{pmatrix} 
 - \eps \ell  {{\alpha'(\eps\delta) \dot\delta \avi{q_{\rm i}}}}  - \frac{1}{2} \jump{ {\mathfrak H}  }  \\
 -\delta+{{\eps \beta(\eps\delta)\dot\delta^2 +\eps \frac{1}{2}\alpha'(\eps\delta)\avi{q_i}^2}}
+\av{{\mathfrak H}}
\end{pmatrix}.
$$
\end{proposition}
\begin{proof}
Let us remark first that if the initial data satisfy $\jump{q_{\vert_{t=0}}}=-2\ell \dot\delta(0)$ and $\av{q_{\vert_{t=0}}}=\av{q_{\rm i}}(0)$ then the transmission condition \eqref{trans1} is equivalent to
\begin{equation}\label{dertrans}
\frac{d}{dt}\jump{q}=-2\ell \ddot\delta \quad\mbox{ and }\quad \frac{d}{dt}\av{q}=\frac{d}{dt}\av{q_{\rm i}}.
\end{equation}
As already noticed in \eqref{altq}, the second equation of \eqref{Boussinesq_ext3} together with the jump condition $\frac{d}{dt}\jump { q} =  -2\ell \ddot \delta$  is equivalent to 
$$
\dt q     =  -  R_0  \dx \mfsw + (\frac{d}{dt} \av{q_{\rm i}}\mp\ell\ddot\delta)  e ^{ -  \frac{1}{\kappa} \abs{x\mp \ell} } \quad \mbox{ on }\quad \cE^\pm.
$$
Using the fact that $R_0\dx =\dx R_1$ and replacing $\frac{d}{dt} \av{q_{\rm i}}$ and $\ddot\delta$ by the formula provided by \eqref{ODEIBVP}, one obtains the result.
\end{proof}

Using the fact that \eqref{formODE} is an ODE on the space $\HH^n\times \RR^3$, with 
 $\HH^n  
= H^{n+1} (\mathcal E)\times H^{n+2}(\mathcal E)$, we obtain the following well-posedness result for the wave-structure interaction problem.
\begin{theorem} \label{theoexist}
\label{prop1} For $n \ge 0$, consider  initial data $U^{\rm in}=(\zeta^{\rm in}, q^{\rm in}) \in \HH^n $ and ${\mathtt Z}^{\rm in}=(\av{q_{\rm i}}^{\rm in},\delta_0,\delta_1)\in \RR^3$ satisfying  
$ \inf  h^{\rm in}>0$.
 Then for all $\eps \in [0,1]$ and
  $\kappa > 0$, there is $T > 0$ such that 
the system \eqref{formODE} has a unique solution in $(U,{\mathtt Z})\in C^1([0,T[ ; \HH^n\times \RR^3)$ with initial data $(U^{\rm in},{\mathtt Z}^{\rm in})$, which in addition belongs to $C^\infty([0,T[ ; \HH^n\times \RR^ 3)$. Moreover, if $T^*$ denotes the maximal existence time and $T^*<\infty$, one has
$$
\limsup_{ t \to T^*}   \big[ \big\vert (\zeta , q,  \frac{1}{1+\eps\zeta})(t)\big\vert_{L^\infty (\mathcal E)} +\abs{\dot\delta(t)}+\abs{\av{q_{\rm i}}(t)}+ \big\vert \frac{1}{1+\eps \delta(t)+\eps\zeta_{\rm eq}}\big\vert_{L^\infty({\mathcal I})}\big] = + \infty . 
$$
\end{theorem} 
\begin{remark}
Since the relations \eqref{dertrans} obviously hold for the solution, the transmission condition 
$$
\av{q}=\av{q_{\rm i}} \quad \mbox{ and }\quad \jump{q}=-2\ell \dot\delta
$$
are satisfied for all time if the initial data satisfy
$$
\jump{q^{\rm in} } = -2\ell \delta_1,\qquad \av{q^{\rm in}}=\av{q_{\rm i}}^{\rm in}.
$$
\end{remark}
\begin{proof}
Let $\mathcal O$ denote the open subset of $\HH^n\times \RR^3$ of the $(U,{\mathtt Z}) = (\zeta, q,\av{q_{\rm i}}, \delta,\dot\delta)$ such that 
 $\inf_\RR  h>0$ (as already explained in the comments before Proposition \ref{propODE}, this latter is a condition on $\zeta$ in the exterior domain, and on $\delta$ in the interior domain).  Let us also write \eqref{formODE} in compact form as
 $$
 \frac{d}{dt} (U,{\mathtt Z})= \Phi (U,{\mathtt Z})
 $$
 where $\Phi=(\phi_1,\phi_2,\phi_3,\phi_4,\phi_5)$ and
 \begin{align*}
 \phi_1= -\dx q,\qquad \phi_2= -  \dx R_1 \mfsw + \big( {\mathcal Q} \mp \ell {\mathcal D} \big) e ^{ -  \frac{1}{\kappa} \abs{x\mp \ell} },\\
\phi_3={\mathcal Q}({\mathtt Z},\jump{\mathfrak H}),\qquad  \phi_4=\dot\delta, \qquad \phi_5={\mathcal D}({\mathtt Z},\av{\mathfrak H}).
 \end{align*}
 From standard trace and product estimates in Sobolev spaces, $\Phi$ is a smooth mapping from ${\mathcal O}$ to $\HH^n\times \RR^3$ and the local existence follows from Cauchy-Lipschitz theorem. From Moser type estimates, we also get
 \begin{equation}\label{eqCmu}
 \vert \Phi(U,{\mathtt Z})\vert_{\HH^n\times \RR^3}\leq C_\mu\big(\vert \zeta, q , \frac{1}{h} \vert_{L^\infty(\cE)},\abs{\av{q_{\rm i}},\delta,\dot\delta},\vert \frac{1}{h_{\rm i}}\vert_{L^\infty({\mathcal I})}\big)
 \vert (U,{\mathtt Z})\vert_{\HH^n\times \RR^3},
 \end{equation}
 with $C_\mu$ a smooth non decreasing function of its arguments. Classically, this means that if the maximal existence time is finite, one of the arguments of $C_\mu$ has to blow up. Remarking further that $\delta$ cannot blow up in finite time without $\dot\delta$ also blowing up, one gets the result.
\end{proof} 

\subsection{Uniform estimates}\label{sectUE}

Theorem \ref{theoexist} shows that the equations are locally well-posed, but the existence time is not uniform with respect to $\eps$ and $\mu$ (or equivalently $\kappa$) and may shrink to zero when these parameters become very small.  It is however possible to derive a uniform estimate on a time interval of size $O(\frac{1}{\eps})$ under the assumption that $(\zeta,q)$ remains uniformly bounded in $W^{1,\infty}(\cE)$. This estimate is a generalization of the estimate one can derive for the Boussinesq equations on the full line (see Step 0 of the proof), and implies in particular that for a time scale $O(1/\eps)$ the solid cannot touch the bottom if $\zeta$, $q$ and their first order spatial derivatives remain bounded.
\begin{theorem}\label{theostab}
 Assume that the assumptions of Theorem \ref{theoexist} are satisfied and let  $M_0>0$ be such that
$$
\inf_{\cE\cup{\mathcal I}} \big(\frac{1}{h^{\rm in}}\big)+\abs{(\zeta^{\rm in},q^{\rm in},\kappa\dx q^{\rm in})}_2 +\abs{(\av{q_{\rm i}}^{\rm in},\delta_0,\delta_1)}\leq M_0,
$$
and assume moreover that there are $\underline{T}>0$  and $M>0$ such that the solution provided by Theorem \ref{theoexist} exists on $[0,\underline{T}]$ and that  $ \abs{(\zeta,q)}_{L^\infty([0,T]\times W^{1,\infty}( \cE))}\leq M$. \\
Then there exists  $T_1=T_1(M_0,M)>0$  such that for all $0\leq t \leq \min\{\underline{T}, \frac{1}{\eps} T_1\}$, one has
\begin{align*}
\inf_{\cE\cup{\mathcal I}} \big(\frac{1}{h}\big)+\abs{(\zeta,q,\kappa\dx q )(t)}_2&+ \abs{(\av{q_{\rm i}},\delta,\dot\delta)(t)}
\leq C\big(M_0\big),
\end{align*}
with $C(\cdot)$ a nondecreasing function of its argument.
\end{theorem}
\begin{remark}
The time $T_1$ and the upper bound $C(M_0)$ only depend on $M_0$ and $M$; in particular they are uniform with respect to $(\eps,\kappa)\in (0,1)^2$.
\end{remark}
\begin{proof}
For the sake of clarity, we generically denote throughout this proof
 by $C(\cdot)$  a nondecreasing function of its arguments {\it that does not depend on $\eps$ nor $\kappa$}, but whose exact expression may differ form one line to another. We also recall that $\kappa^2=\mu/3$.\\
{\bf Step 0.} Energy estimates for the Boussinesq equations on the full line. For the sake of clarity we first explain here how to derive an energy estimate for the Boussinesq equations \eqref{Abbott} when they are cast on the full line $\RR$. More precisely, we show that if $(\zeta,q)$ is a smooth solution on a time interval $[0,T]$ on which $h\geq h_{\rm min}>0$, then
$$
\forall t \in [0,T], \qquad \mfE^{\rm Bouss}(t)\leq \mfE^{\rm Bouss}(0)\exp\Big(\eps  t C\big(\frac{1}{h_{\rm min}},\abs{(\zeta,q)}_{L^\infty([0,T];W^{1,\infty}( \RR))} \big) \Big),
$$
where $\mfE^{\rm Bouss}$ is the energy associated with the Boussinesq system, 
$$
\mfE^{\rm Bouss}=\frac{1}{2}\int_{\RR} \Big( \zeta^2+\frac{1}{h}q^2 +\kappa^2 \frac{1}{h}(\dx q)^2\Big).
$$
Using \eqref{eqNRJ}, one readily gets that
$$
\frac{d}{dt} \mfE^{\rm Bouss}=3\eps \kappa^2 \int_{\mathbb R} {\mathfrak R} \quad\mbox{ with }\quad {\mathfrak R}=  \frac{1}{6h^2}(\dx q)^3 +\frac{1}{3h^2}q (\dt\dx q) \dx \zeta,
$$
so that
\begin{align*}
\kappa^2 \int_{\RR}{\mathfrak R}\leq 
 C\big( \frac{1}{h_{\rm min}},\abs{(\dx \zeta, \dx q)}_{L^\infty(\RR)}\big) \big(\kappa^2 \abs{\frac{1}{\sqrt{h}}\dx q}_{L^2(\RR)}^2 + \abs{\frac{1}{\sqrt{h}}q}_2 \abs{\kappa^2 \dx\dt q}_{L^2(\RR)} \big).
\end{align*}
Now, using the second equation of the Boussinesq system, one has
 \begin{equation}\label{estfull}
 \kappa^2\dx\dt q=-\kappa^2 (1-\kappa^2\dx^2)^{-1}\dx^2 \mfsw,
 \end{equation}
and therefore, 
$$\abs{\kappa^2 \dx\dt q}_{L^2(\RR)} \leq C(\frac{1}{h_{\rm min}},\abs{(\zeta,q)}_{L^\infty(\RR)})\abs{(\zeta,\frac{1}{\sqrt{h}}q)}_{L^2(\RR)}.
$$
It is then straightforward to deduce that
$$
\frac{d}{dt} \mfE^{\rm Bouss} \leq \eps C(\frac{1}{h_{\rm min}},\abs{(\zeta,q)}_{W^{1,\infty}}) \mfE^{\rm Bouss},
$$
which yields the energy estimate stated above.\\
We shall follow the general scheme of this proof for our wave-structure system; the main difference are that some control are needed for the quantities $\delta$, $\dot\delta$ and $\av{q_{\rm i}}$ associated with the interior region, and one has to consider an initial boundary value problem instead of a simple initial value problem for the Boussinesq system; in particular, \eqref{estfull} is no longer valid and  boundary terms make the analysis more delicate.

\noindent
{\bf Step 1.} Adaptation for the wave-structure system, assuming that $h\geq h_{\rm min}$ on $\cE$ for some $h_{\rm min}>0$. We shall work here with the formulation  \eqref{Boussinesq_ext2}-\eqref{trans3} of the problem, as derived in Theorem \ref{theotransm}. The quantity $\mfE^{\rm Bouss}$ used in Step 0 is here replaced by $\mfE^{\rm ext}$, which is  also the integral of the local density of energy ${\mfe }$ but on the exterior region $\cE$ instead of the whole line ${\mathbb R}$, see \eqref{defEext}, and we also need the interior energy $\mfE_{\rm int}$ defined in \eqref{defEint}. 
As shown in Proposition \ref{prior-conservation-case}, we have 
$$
 \frac{d}{dt}\big[\mfE_{\rm ext} +\mfE_{\rm int}\big]+
 \eps\kappa^2 \ell\av{ \frac{1}{h_{\rm i}^2} } \dot\delta^3
 =3\eps\kappa^2 \int_{\cE} {\mathfrak R}.
$$
Controlling ${\mathfrak R}$ as in Step 0, we have
\begin{equation}\label{step1.1}
 \frac{d}{dt}\big[\mfE_{\rm ext} +\mfE_{\rm int}\big]+
 \eps\kappa^2 \ell \av{ \frac{1}{h_{\rm i}^2} } \dot\delta^3
 \leq \eps C\big( \frac{1}{h_{\rm min}},\abs{(\dx \zeta, \dx q)}_\infty \big) \mfE_{\rm ext}+\eps \abs{\kappa^2\dx\dt q}_2^2 
\end{equation}
(recall that the notation $\abs{\cdot}_2$ stands for $\abs{\cdot}_{L^2(\cE)}$), and, as in the previous step, the key point is to control $\abs{\kappa^2 \dx\dt q}_2$. Because of the boundaries, and as shown by Proposition \ref{propODE},  \eqref{estfull} must be replaced by
$$
\dt \dx q=-\dx^2 R_1 {\mathfrak f}_{\rm sw}\mp \frac{1}{\kappa} \dot q_\pm e^{-\frac{1}{\kappa}\abs{x\mp \ell}},
$$
where we recall that $q_\pm=q_{\vert_{x=\pm\ell}}$, so that
$$
\kappa^2 \abs{\dx\dt q}_2\lesssim \abs{\kappa^2 \dx^2 R_1 {\mathfrak f}_{\rm sw}}_2+{{\kappa}}^{3/2} \abs{(\dot q_-,\dot q_+)}.
$$
Since $\kappa^2 \dx^2 R_1: L^2 \to L^2$ is uniformly bounded (with respect to $\kappa$), the first term in the right-hand-side can be controlled exactly as in \eqref{estfull}, so that one gets from \eqref{step1.1} that
\begin{align}
\nonumber
 \frac{d}{dt}\big[\mfE_{\rm ext} +\mfE_{\rm int}\big]&+
 \eps\kappa^2\ell \av{ \frac{1}{h_{\rm i}^2} } \dot\delta^3 \\
 \label{step1.3}
& \leq \eps C\big( \frac{1}{h_{\rm min}},\abs{(\zeta, q)}_{W^{1,\infty}} \big) \mfE^{\rm ext} + \eps \kappa^{3}  \abs{(\dot q_-,\dot q_+)} ^2   .
 \end{align}
To close the estimate, we still need a control  on $  \abs{(\dot q_-,\dot q_+)}$.

\noindent
{\bf Step 2.} Control of $ \abs{(\dot q_-,\dot q_+)}$. According to Proposition \ref{propODE}, one has
 \begin{equation}\label{eqqpmDQ}
  \abs{(\dot q_-,\dot q_+)}\leq  \big(\abs{{\mathcal Q}}+\ell \abs{\mathcal D}\big),
  \end{equation}
  with ${\mathcal Q}$ and ${\mathcal D}$ defined in Proposition \ref{propODE}. We remark first that $\av{\mathfrak H} $ and $\jump{\mathfrak H}$, with ${\mathfrak H}$ as defined in \eqref{defC}, can be controlled as 
$$
\abs{\av{\mathfrak H}}+\abs{\jump{\mathfrak H}}\leq C\big( \frac{1}{h_{\rm min}},\abs{\zeta}_\infty \big) \big[\eps \big( 1+ {\dot\delta}^2+ {\av{q_{\rm i}}}^2\big)+\abs{(R_1 \mfsw)_\pm}\big].
$$
with $(R_1 \mfsw)_\pm=(R_1 \mfsw)_{\vert_{x=\pm \ell}}$. Let us now remark that for all $f\in L^2(\cE)$, one has
$$
(R_1 f)_\pm=\kappa^{-1}\int_{\cE^\pm} \exp(-\frac{1}{\kappa}\vert x\mp \ell \vert) f(x){\rm d}x
$$
so that
\begin{equation}\label{traceR1}
\abs{(R_1 f)_\pm }\leq (2\kappa)^{-1/2}\abs{f}_2.
\end{equation}
It follows from the above that
$$
\abs{\av{\mathfrak H}}+\abs{\jump{\mathfrak H}}\leq C\big( \frac{1}{h_{\rm min}},\abs{(\zeta,q)}_\infty \big) \big[\eps \big( 1+ {\dot\delta}^2+ {\av{q_{\rm i}}}^2\big)+\frac{1}{\kappa^{1/2}}\abs{(\zeta,q)}_2\big].
$$
We directly deduce from the definition of ${\mathcal Q}$ and ${\mathcal D}$ provided in Proposition \ref{propODE} and \eqref{eqqpmDQ} that
\begin{align}
\label{estQD}
 \abs{(\dot q_-,\dot q_+)}\leq& C\big( \frac{1}{h_{\rm min}} ,\abs{(\zeta,q)}_\infty\big)  \big[ \abs{\delta}+\eps \big(1+ {\dot\delta}^2+ {\av{q_{\rm i}}}^2\big)+\frac{1}{\kappa^{1/2}}\abs{(\zeta,q)}_2\big].
\end{align}

\noindent
{\bf Step 3.} We show here that one can choose $T_1$ such that the assumption $h\geq h_{\rm min}>0$ is satisfied on $\min\{\underline{T},\frac{1}{\eps}T_1\}$. Indeed, since by assumption $\inf_\cE h^{\rm in}>0$, there exists $h_{\rm min}>0$ such that $\inf_\cE h^{\rm in}\geq 2 h_{\rm min}$. Since $\dt\zeta=-\dx q$, one can write
$$
h(t,x)=h^{\rm in}(x)-\eps \int_0^t \dx q (s,x){\rm d}s
$$
and choosing $T_1>0$ such that $\eps T_1 M \leq h_{\rm min}$ yields the result.

\noindent
{\bf Step 4.}
 Conclusion. Using \eqref{estQD} in \eqref{step1.3}, and plugging the resulting estimate into \eqref{step1.1}, one obtains that
$$
 \frac{d}{dt}\big[{\mathfrak E}_{\rm ext} + {\mathfrak E}_{\rm int}\big]
 \leq \eps F\big({\mathfrak E}_{\rm ext}+  {\mathfrak E}_{\rm int}\big),
$$
for some smooth function $F$ that does not depend on $\kappa\in (0,1)$ and $\eps\in (0,1)$. From the theorem of comparison for ODEs, one
deduces that is possible to choose $T_1>0$ such that 
${\mathfrak E}_{\rm ext} + {\mathfrak E}_{\rm int}$ is uniformly bounded from above by a constant depending only on $M_0$ on the time interval
$\min\{\underline{T},\frac{1}{\eps}T_1\}$. We have already seen that $h\geq h_{\rm min}>0$ on $\cE$ over this time interval. Taking a smaller $T_1$ if necessary, one gets similarly that $h_{\rm i}\geq h_{\rm min}>0$ on ${\mathcal I}$. The results follows.
\end{proof}

Theorem \ref{theostab} is only a conditional result, since it assumes that the solution remains uniformly bounded in $W^{1,\infty}(\cE)$. This is the equivalent of the basic $L^2$-estimate for hyperbolic initial boundary value problems. In the hyperbolic framework, the next natural steps would be to obtain a similar control on the time derivatives of the solution by the initial value of these time derivatives, to express these latter quantities in terms of spatial derivatives of the initial data, and finally to use some ellipticity property to control space-derivatives in terms of time derivatives. By Sobolev embedding, one could then control the $W^{1,\infty}(\cE)$ by energy norms and obtain an unconditional result (see for instance \cite{IguchiLannes} or the lecture notes \cite{LannesBressanone} for the implementation of this strategy for the shallow water equations). \\
In the presence of dispersion, this strategy is much more delicate to implement; controlling the initial value of the time derivatives in terms of spatial derivatives of the initial data, and recovering information on the space derivatives from the control of the time derivatives is considerably more difficult than in the hyperbolic case. This program has been achieved in \cite{BLM}, where well-posedness is established for a time scale $O(1/\eps)$, uniformly with respect to $\mu$ (or, equivalently, $\kappa$), but  for the formally equivalent Boussinesq system \eqref{varBLM} instead of \eqref{Abbott}, and for a fixed object -- these two conditions made possible the reduction to a transmission problem with {\it linear} transmission conditions. The situation here is made more complicated for at least three reasons:
\begin{itemize}
\item The floating object is not fixed  and one needs to understand its coupling with the exterior wave field and in particular the dispersive contribution to the added mass effect;
\item  The contribution of the dispersive term in the transmission conditions \eqref{trans1}-\eqref{trans3}, namely $-\frac{\mu}{3h}\dx\dt q$, is nonlinear (in \cite{BLM}, it is given by the linear expression  $-\frac{\mu}{3}\dx\dt q$); as shown below, this is why we need the hidden regularity effect exhibited here;
\item  The energy conservation is not exact as in \cite{BLM}. Proposition \ref{prior-conservation-case} shows that the residual is formally small, namely, of order $O(\eps\kappa^2)$ but it is not obvious at all that it can be controled by the natural energy of the system.
\end{itemize}
A full proof of the uniform well-posedness for \eqref{Boussinesq_ext2}-\eqref{trans3} requires considerable work and would probably double the size of this paper; we therefore postpone it for future work. We want however to address here the issue of energy estimates for the linearized equations since this might be where the main difference with respect to \cite{BLM} lays, and because it exhibits a phenomenon of independent interest that can be interpreted as a dispersive equivalent of the trace estimates obtained in the hyperbolic case through Kreiss symmetrizers.

In order to understand where the difficulty comes from, let us remark that when one applies $\dt^j$ ($j\geq 1$) to the linear expression $-\frac{\mu}{3} (\dx\dt q)_\pm$, one finds $-\frac{\mu}{3} (\dx\dt (\dt^j q))_\pm$ which is the same term with $q$ replaced by $\dt^j q$. The transmission conditions one has to deal with  in \cite{BLM} for the time derivatives of the solution have therefore the same structure as the original one, and can be dealt with using the basic $L^2$-estimate (the equivalent of Theorem \ref{theostab}). Now, when applying $\dt^j$ to the {\it nonlinear} term $-\frac{\mu}{3h} (\dx\dt q)_\pm=\frac{\mu}{3h} (\dt^2 \zeta)_\pm$, one finds
\begin{align}
\nonumber
 \dt^j\big( \frac{\mu}{3h} (\dt^2 \zeta)_\pm \big)=&\frac{\mu}{3h} (\dt^2 (\dt^j \zeta))_\pm\\
 \label{expcj}
& - \mu\eps \frac{j}{3}\dt \big( \frac{\dt\zeta}{h^2}\dt^j \zeta\big)_\pm
 +\frac{\mu\eps}{3}\big[ j \dt (\frac{\dt\zeta}{h^2})-\frac{\dt^2\zeta}{h^2}\big]\dt^j \zeta_\pm.
\end{align}
The first term in the right-hand side of this expression is the same as the original one with $\zeta$ replaced by $\dt^j \zeta$, but the other two are new and they involve the trace of $\dt^j \zeta$ and $\dt^{j+1}\zeta$ at $x=\pm \ell$. These quantities cannot be controlled by the energy norms of $(\dt^j \zeta,\dt^j q)$ and require a specific treatment that we now describe and which is based on a hidden regularity effect of a completely different nature as the one, based on Kreiss symmetrizers, that is used in the hyperbolic case get control on the trace of the solution (see for instance \cite{Metivier2012,BenzoniSerre,IguchiLannes}).

We consider a system linearized around a couple of functions $(\underline{\zeta},\uq)$ (typically the exact solution), and with source terms $f$, $g_1$ and $g_2$ in the linearized momentum and transmission conditions equations respectively, namely,
\begin{equation}\label{Boussinesq_lin}
\begin{cases}
\dt \zeta+\dx q=0,\\
(1-\kappa^2\dx^2)\dt q+ (\underline{h}-\eps^2 \frac{\uq^2}{\uh^2})\dx\zeta+2\eps \frac{\uq}{\uh}\dx q= \eps f,
\end{cases}
\end{equation}
with the transmission conditions
\begin{equation}\label{transmlin}
\av{q}=\av{q_{\rm i}}
\quad\mbox{ and }\quad\jump{q}=-2\ell \dot \delta,
\end{equation}
where $\av{q_{\rm i}}$ and $\dot \delta$ are provided by the linear ODEs
\begin{align}
\label{transmlin1}
\alpha(\eps\underline{\delta})\frac{d}{dt}\av{q_{\rm i}}+\eps {\bf b}_1[\underline{\mathtt Z}]\cdot{\mathtt Z}&=-\frac{1}{2\ell}\jump{ \zeta+\underline{{\mathfrak G}}+\eps\kappa^2  \big( {\mathtt c}_2[\underline{\zeta}] \zeta +\dt ({\mathtt c}_1[\underline{\zeta}]\zeta) \big)} +\eps g_1,\\
\label{transmlin2}
\tau_\mu(\eps\udelta)^2 \ddot\delta+\delta +\eps {\bf b}_2[\underline{\mathtt Z}]\cdot{\mathtt Z}&=\av{\zeta+\underline{{\mathfrak G}} +\eps\kappa^2  \big( {\mathtt c}_2[\underline{\zeta}] \zeta +\dt ({\mathtt c}_1[\underline{\zeta}]\zeta) \big)}+\eps g_2,
\end{align}
where ${\mathtt c}_k[\underline{\zeta}]$ ($k=1,2$) is  a smooth function of $\underline{\zeta}, \dt\underline{\zeta},\dots,\dt^k\underline{\zeta}$ and, recalling that ${\mathtt Z}=(\av{q_{\rm i}},\delta,\dot\delta)^{\rm T}$,
\begin{align*}
{\bf b}_1[\underline{\mathtt Z}]&=\Big( \alpha'(\eps\udelta)\dot\udelta, \frac{d}{dt}\big( \alpha'(\eps\udelta)\av{\underline{q}_{\rm i}}\big), \alpha'(\eps\udelta)\av{\uq_{\rm i}}\Big)^{\rm T}\\
{\bf b}_2[\underline{\mathtt Z}]&=\Big(- \alpha'(\eps\udelta)\av{\underline{q}_{\rm i}}, 2\frac{d}{dt}\big( \tau_\mu(\eps\udelta){\tau_\mu}'(\eps\udelta)\dot\udelta\big)+\beta'(\eps\udelta)\dot\udelta^2- \frac{\eps}{2}\alpha''(\eps\udelta)\av{\underline{q}_{\rm i}}^2, -2 \beta(\eps\udelta)\dot\udelta\Big)^{\rm T},
\end{align*}
while $\underline{{\mathfrak G}}$ is given by
\begin{equation}\label{defGlin}
\underline{{\mathfrak G}}=-\eps^2 \frac{\uq^2}{\uh^3} \zeta  +\eps \frac{\uq}{\underline{h}^2}  q  - \kappa^2\frac{1}{\underline{h}}\dx\dt q.
\end{equation}
\begin{remark}\label{remTD}
If $(\zeta_{\rm ex},q_{\rm ex})$ denotes an exact solution to the  wave-structure equations \eqref{Boussinesq_ext2}-\eqref{trans3}, then for all $j\geq 1$, the  time derivatives $(\dt^j\zeta_{\rm ex},\dt^jq_{\rm ex})$ solve a system of the form  \eqref{Boussinesq_lin}-\eqref{transmlin2}, with $(\underline{\zeta},\udelta)=(\zeta_{\rm ex},q_{\rm ex})$, $(\zeta,q)=(\dt^j\zeta_{\rm ex},\dt^jq_{\rm ex})$ and, according to \eqref{expcj},
$$
{\mathtt c}_1[{\zeta}_{\rm ex}]=-  j\frac{\dt{\zeta}_{\rm ex}}{h_{\rm ex}^2}
\quad\mbox{ and } \quad
{\mathtt c}_2[{\zeta}_{\rm ex}]=\big[ j \dt (\frac{\dt{\zeta_{\rm ex}}}{h_{\rm ex}^2})-\frac{\dt^2{\zeta_{\rm ex}}}{h_{\rm ex}^2}\big],
 $$
while $f$, $g_1$ and $g_2$ lower order commutator terms; for instance for $j=1$,
$$
 f=\dt \big(\zeta_{\rm ex}-\eps \frac{q_{\rm ex}^2}{h_{\rm ex}^2}\big) \dx \zeta_{\rm ex}
 \quad\mbox{ and }\quad
 g_1=g_2=0.
$$
\end{remark}

The following theorem shows that the linearized problem \eqref{Boussinesq_lin}-\eqref{transmlin2} is well-posed, and provides a control on the augmented energy $\underline{\mfE}_{\rm augm}$ defined as
\begin{equation}\label{augmNRJ}
\underline{\mfE}_{\rm augm}=\abs{\mathtt Z}^2+\abs{(\zeta,q,\kappa\dx q)}_2^2 +\eps\kappa^3 \abs{(\zeta_-,\zeta_+)}^2+\eps\kappa^5 \abs{(\dot \zeta_-,\dot \zeta_+)}^2;
\end{equation}
this energy contains the energy $\underline{\mfE}_{\rm ext}+\underline{\mfE}_{\rm int}$ used in the proof of Theorem \ref{theostab} but provides in addition a control on the traces of $\zeta_\pm$ and their first time derivative. This is a hidden regularity property granted by the dispersive terms. For the sake of clarity, in the following statement, we simply write $\underline{\mathtt c}_j$ instead of ${\mathtt c}_j[\underline{\zeta}]$.
\begin{theorem}\label{theolin}
Let $(\underline{\zeta},\uq) \in C^1(\RR^+\times\cE)$ and assume that  $(\underline{\mathtt c}_1,\dt \underline{\mathtt c}_1,\underline{\mathtt c}_2)_{\vert_{x=\pm \ell}}$ and $ (\underline{\mathtt Z},\dot{\underline{\mathtt Z}})$ are continuous functions of time. Let also
 $M>0$ be such that
$$
\abs{(\underline{\zeta},\dt \underline{\zeta}, \dx \underline{\zeta}, \uq, \dx \uq)}_{L^\infty(\RR^+\times \cE)}\leq M
\quad\mbox{ and }\quad
\abs{(\underline{\mathtt c}_1,\underline{\mathtt c}_2,\underline{\mathtt Z},\dot{\underline{\mathtt Z}})}_{L^\infty(\RR^+)}\leq M
$$
and assume that there exists $h_{\rm min}>0$ and $c_{\rm min}>0$  such that
$$
 \inf_{\RR^+\times \cE} \uh \geq  h_{\rm min}  \quad \mbox{ and }\quad  \inf_{\RR^+\times \cE}\big(\uh -\eps^2 \frac{\uq^2}{\uh^2}\big)\geq c_{\rm min}.
$$
Then for all $(\zeta^{\rm in},q^{\rm in})\in L^2\times H^1(\cE)$ and all ${\mathtt Z}^{\rm in}\in \RR^3$, there exists a unique solution $(\zeta,q,{\mathtt Z})$ in $C^1(\RR^+;L^2\times H^1(\cE)\times \RR^3)$ to \eqref{Boussinesq_lin}-\eqref{transmlin2} with initial data $(\zeta^{\rm in},q^{\rm in},{\mathtt Z}^{\rm in})$. Moreover, $\zeta_{\vert_{x=\pm \ell}}$ exist in $W^{1,\infty}_{\rm loc}(\RR^+)$ and there are constants $C_0=C_0(\frac{1}{h_{\rm min}},\frac{1}{c_{\rm min}})$ and $\underline{C}=\underline{C}(C_0,M)$ such that if $\eps\kappa \underline{C} <1$, the following estimate holds for all $t>0$,
$$
\underline{\mfE}_{\rm augm}(t)\leq C_0\big[ \underline{\mfE}_{\rm augm}(0)+\eps \underline{C}\int_0^t \big( \abs{f}_2^2+\abs{(g_1,g_2)}^2\big) \big]
\exp(\sqrt{\eps}\underline{C}t).
$$
\end{theorem}
\begin{remark}
Without the extra control provided by the theorem on the hidden trace regularity of $\zeta$, one could not close the energy estimate. Hidden regularity at the boundary for hyperbolic systems was already noticed in \cite{Lions} and can be obtained in many cases by using Kreiss symmetrizers that make the boundary condition maximal dissipative. The hidden regularity is granted here by the dispersion (rather than a Kreiss symmetrizer), but it is of a different nature since it provides a control for each time $t$ of the traces, as opposed to an $L^2$-norm in time for maximally dissipative hyperbolic systems (see for instance \cite{BenzoniSerre,Metivier2012} and, more related to the present context, \cite{IguchiLannes,LannesBressanone}, as well as \cite{Audiard} for a generalization of Kreiss' approach to a class of linear dispersive equations that does not cover the linear version of the Boussinesq-Abbott system). Note also that even with this hidden regularity,  the equations \eqref{Boussinesq_lin}-\eqref{transmlin2} do not obviously make sense because \eqref{transmlin1} and \eqref{transmlin2} involve the traces $\dx\dt q_{\vert_{x=\pm}}$. This difficulty is removed if we rather work with the equivalent formulation \eqref{ODEIBVPbis} derived in the proof. 
\end{remark}
\begin{remark}
The constants $C_0$ and $\underline{C}$ involved in the statement of the theorem depend only on $h_{\rm min}$, $c_{\rm min}$ and $M$; in particular, they are uniform with respect to $\eps\in (0,1)$ and $\kappa\in(0,1)$ (equivalently, with respect to $\mu$). The theorem provides therefore uniform estimates over a large time scale, namely, $O(\eps^{-1/2})$, which is however shorter than the $O(\eps^{-1})$ time scale classically associated with the existence time of solutions to Boussinesq system on the full line. This is due to the necessity of controlling the traces of the solution at $x=\pm \ell$. Note that the $O(\eps^{-1/2})$ time scale is the same as the one obtained in \cite{LPS} for the existence of a Boussinesq system on the full line using dispersive methods. Using other methods, it was however later proved \cite{SautXu,Burtea} that the time scale $O(\eps^{-1})$ could be reached. The $O(\eps^{-1})$ time scale was also attained in \cite{BLM} for the Boussinesq system \eqref{varBLM} in the presence of a fixed object, but, as explained above, no control of the traces is needed there. It is therefore an open question to assess whether the shorter time scale $O(\eps^{-1/2})$ of Theorem \ref{theolin} is dictated by the dispersive control of the traces, or wether it is only a technical limitation.
\end{remark}
\begin{remark}
This theorem furnishes uniform bounds for the time derivatives to the solutions of  \eqref{Boussinesq_ext2}-\eqref{trans3} (see Remark \ref{remTD}); as explained above, this is the key step towards well-posedness on a uniform time-scale, and it differs strongly from the linear estimates of \cite{BLM} because of the necessary control of the trace of the solution. The other steps of the proof are expected to be more similar to \cite{BLM} and for the sake of conciseness, we prefer to treat them in a separate work.
\end{remark}
\begin{proof}
Throughout this proof, for the sake of clarity, we  use the same notations $C_0=C_0(\frac{1}{h_{\rm min}},\frac{1}{c_{\rm min}})$ and $\underline{C}=\underline{C}(C_0,M)$ 
for various constants that may differ from one line to another. In the first four steps of the proof, we establish the energy estimate stated in the theorem for smooth solutions of the problem. We then prove existence and uniqueness of regular solutions in Step 5, and extend this result to the regularity considered in the theorem using a density argument and the control of the trace of $\zeta$ at the boundaries furnished by the energy estimate.

\noindent
{\bf Step 1}. Defining the analogous  for the linearized equations of the energies ${\mathfrak E}_{\rm ext}$ and ${\mathfrak E}_{\rm int}$ defined in \eqref{defEext} and \eqref{defEint}, namely,
\begin{align*}
\underline{{\mathfrak E}}_{\rm ext}&=\int_\cE \frac{1}{2\uh} \big( \uh-\eps^2 \frac{\uq^2}{\uh^2}\big) \zeta^2 +\frac{1}{2\uh}q^2 +\frac{\kappa^2}{2\uh}(\dx q)^2\\
\underline{{\mathfrak E}}_{\rm int}&= \ell \alpha(\eps\udelta)\av{q_{\rm i}}^2  + \ell \tau_\mu(\eps\udelta)^2\dot\delta^2 +\ell\delta^2,
\end{align*}
the first step is to prove the following lemma. Note that the inequality stated in the lemma corresponds to \eqref{step1.1} in the proof of Theorem \ref{theostab}. As explained above, the nonlinear structure of the dispersive terms in the transmission conditions makes the analysis of the linearized equations more delicate. The last two terms in the estimate stated in the lemma come from the subprincipal terms involving ${\mathtt c}_1[\underline{\zeta}]$ and ${\mathtt c}_2[\underline{\zeta}]$ in \eqref{transmlin1} and \eqref{transmlin2} and that are not present in the original (nonlinear) equations. Note in particular the appearance of the traces $\zeta_\pm=\zeta_{\vert_{x=\pm \ell}}$ that cannot be controlled by the energy norm $\underline{\mfE}_{\rm ext}$. Another consequence of these subprincipal terms is that, in the left-hand-side of \eqref{ineqstep1}, the energy $  \underline{\mathfrak E}_{\rm ext} +\underline{\mathfrak E}_{\rm int}$ must be modified by adding non signed trace terms.
\begin{lemma}\label{lemmaineq}
The following inequality holds (denoting $\underline{\mathtt c}_j=\mathtt{c}_j[\underline{\zeta}]$),
\begin{align}
\nonumber
\frac{d}{dt}\big[  \underline{\mathfrak E}_{\rm ext} +\underline{\mathfrak E}_{\rm int}&
+\eps \kappa^2  \big( \av{q_{\rm i}}\jump{  \underline{\mathtt c}_1 \zeta  }-2  \ell \dot\delta\av{  \underline{\mathtt c}_1\zeta }\big)\big]
\label{ineqstep1}
\leq \eps \big(\abs{\kappa^2\dt\dx q}_2^2+ \abs{f}_2^2+\abs{(g_1,g_2)}^2\big)\\
&+
 \eps  \underline{C}\big(  \underline{\mathfrak E}_{\rm ext} +\big(1+\frac{\kappa}{\sqrt{\eps}}\big)\underline{\mathfrak E}_{\rm int}
+\kappa^3\eps^{1/2} \abs{(\zeta_-,\zeta_+)}^2\big)+\eps^{1/2}{\kappa}\abs{(\dot q_-,\dot q_+)}^2.
\end{align}
\end{lemma}
\begin{proof}[Proof of the lemma]
Multiplying the first equation of \eqref{Boussinesq_lin} by $\frac{1}{\underline{h}}(\underline{h}-\eps^2 \frac{\underline{q}^2}{\underline{h}^2})\zeta$ and the second one by $\frac{1}{\underline{h}}q$ and integrating by parts, one obtains after some computations
\begin{equation}\label{nrjloclin}
\frac{d}{dt}\underline{\mfe}_{\rm ext}+\dx \big( q (\zeta+\underline{{\mathfrak G}})\big)=\eps \underline{r}+3\eps \kappa^2 \underline{\mathfrak R}
\end{equation}
with $\underline{{\mathfrak G}}$ as in \eqref{defGlin} and
\begin{align*}
 \underline{{\mathfrak e}}_{\rm ext}&=\frac{1}{2\uh}  \big( \uh-\eps^2 \frac{\uq^2}{\uh^2}\big)  \zeta^2 +\frac{1}{2\uh}q^2 +\frac{\kappa^2}{2\uh}(\dx q)^2,\\
\underline{r}&=-\eps \dx \big( \frac{\uq^2}{\uh^3} \big)\zeta q+ \big[\dx\big( \frac{\uq}{\uh^2} \big) -\frac{1}{2\uh^2}(\dt\underline{\zeta})   \big]q^2+\frac{1}{\uh} f q,\\
\underline{\mathfrak R}&=-\frac{1}{6\uh^2}(\dt \underline{\zeta})(\dx q)^2+\frac{1}{3\uh^2}(\dx\underline{\zeta}) q (\dt\dx q).
\end{align*}
Integrating \eqref{nrjloclin} over $\cE$ and remarking that   $\jump{q \underline{{\mathfrak G}}} =\av{q}\jump{\underline{{\mathfrak G}}} + \jump{q}\av{\underline{{\mathfrak G}}}$, we get from the transmission conditions \eqref{transmlin} that 
$$
\frac{d}{dt} \underline{\mathfrak E}_{\rm ext}- \av{q_{\rm i}}\jump{\zeta+\underline{\mathfrak G}} +2\ell \dot{\delta}\av{\zeta+\underline{\mathfrak G}}=\eps \int_\cE \underline{r}+3\eps\kappa^2 \int_\cE \underline{R}_\mu.
$$
With \eqref{transmlin1} and \eqref{transmlin2}, this yields
\begin{align*}
\frac{d}{dt}\big(  \underline{\mathfrak E}_{\rm ext} +\underline{\mathfrak E}_{\rm int}\big)
=&-2 \eps \ell \big({\bf b}_1(\underline{\mathtt Z})\cdot {\mathtt Z} \av{q_{\rm i}} + {\bf b}_2(\underline{\mathtt Z})\cdot {\mathtt Z} \dot\delta  \big) \\
&-\eps \kappa^2 \av{q_{\rm i}}\jump{   \underline{\mathtt c}_2  \zeta +\dt (\underline{\mathtt c}_1 \zeta) }+2\eps\kappa^2 \ell \dot\delta\av{   \underline{\mathtt c}_2 \zeta +\dt (\underline{\mathtt c}_1\zeta) }\\
&+\eps \int_\cE \underline{r}+3\eps\kappa^2 \int_\cE \underline{\mathfrak R}+2\eps \ell\big( g_1\av{q_{\rm i}}+g_2 \dot\delta \big).
\end{align*}
Decomposing
\begin{align*}
- \av{q_{\rm i}}\jump{  \dt (\underline{\mathtt c}_1 \zeta) }+2 \ell \dot\delta\av{  \dt (\underline{\mathtt c}_1\zeta) }
=& \dt \big( -\av{q_{\rm i}}\jump{  \underline{\mathtt c}_1 \zeta  }+2  \ell \dot\delta\av{  \underline{\mathtt c}_1\zeta }\big)\\
&- \big( -\frac{d}{dt}\av{q_{\rm i}}\jump{  \underline{\mathtt c}_1 \zeta  }+2  \ell \ddot\delta\av{  \underline{\mathtt c}_1\zeta }\big),
\end{align*}
using that $\eps\kappa^2 ab\lesssim \eps^{1/2}\kappa a^2+\eps^{3/2}\kappa^3 b^3$, and remarking that $\abs{\frac{d}{dt}\av{q_{\rm i}}}+\abs{\ddot \delta} \lesssim \abs{(\dot q_-,\dot q_+)}$, one readily gets the result.
\end{proof}
The next step consists in controlling the term $\abs{\kappa^2\dx\dt q}_2$ that appears in \eqref{ineqstep1}. This step is an adaptation of Step 1 in the proof of Theorem \ref{theostab} that does not require any qualitative change.
Rewriting the second equation of \eqref{Boussinesq_lin} as
$$
(1-\kappa^2\dx^2)\dt q+ \dx \underline{{\mathfrak f}}_{\rm sw}= \eps \tilde f,
$$
with 
\begin{align*}
\underline{{\mathfrak f}}_{\rm sw}&=(\underline{h}-\eps^2 \frac{\uq^2}{\uh^2})\zeta+2\eps \frac{\uq}{\uh} q
\quad\mbox{ and }\quad
\tilde f=f+\dx \big((\underline{\zeta}-\eps \frac{\uq^2}{\uh^2})\big) \zeta+2 \dx\big( \frac{\uq}{\uh}\big) q,
\end{align*}
we get as in Step 1 of the proof of Theorem \ref{theostab} that, on $\cE^\pm$,
\begin{equation}\label{dtdxq0}
\dt\dx q=-\dx^2 R_1 \underline{{\mathfrak f}}_{\rm sw}+\eps \dx R_0 \widetilde{f} \mp \dot{q}_\pm\frac{1}{\kappa}\exp\big( - \frac{1}{\kappa}\abs{x\mp \ell} \big).
\end{equation}
Recalling that $\kappa^2\dx^2 R_1$ and $\kappa\dx R_0$ are uniformly bounded operator on $L^2(\cE)$, we deduce that
$$
\abs{\kappa^2\dx\dt q}_2\leq \underline{C} \,\underline{{\mathfrak E}}_{\rm ext}^{1/2}+\abs{f}_2+\kappa^{3/2}\abs{(\dot q_+,\dot q_-)};
$$
with \eqref{ineqstep1}, this yields, 
\begin{align}
\nonumber
\frac{d}{dt}\big[  \underline{\mathfrak E}_{\rm ext} +\underline{\mathfrak E}_{\rm int}
+\eps \kappa^2  \big( \av{q_{\rm i}}\jump{  \underline{\mathtt c}_1 \zeta  }-2  \ell \dot\delta\av{   \underline{\mathtt c}_1\zeta }\big) \big]&
\leq  \eps \big( \abs{f}_2^2+\abs{(g_1,g_2)}^2\big)\\
\label{eqnrjfirst}
 +\eps  \underline{C}\big(  \underline{\mathfrak E}_{\rm ext} +(1+\frac{\kappa}{\eps^{1/2}})\underline{\mathfrak E}_{\rm int}+&\kappa^3 \eps^{1/2}\abs{(\zeta_-,\zeta_+)}^2\big)+\eps^{1/2}{\kappa}\abs{(\dot q_-,\dot q_+)}^2;
\end{align}
this inequality should be compared with \eqref{step1.3} in the proof of Theorem \ref{theostab}. The coefficient   $\eps^{1/2}{\kappa}$ in front of $\abs{(\dot q_-,\dot q_+)}^2$, inherited from \eqref{ineqstep1},  is much larger than the coefficient $\eps\kappa^{3}$ in \eqref{step1.3}; moreover, a control on the traces of $\zeta$ at the boundary is also needed.

\noindent
{\bf Step 2.} Control on  $\abs{(\dot q_+,\dot q_-)}$. We show here that
\begin{equation}\label{estqpm}
\abs{\dot q_+}+\abs{\dot q_-}\leq \underline{C} 
\big( \abs{{\mathtt Z}} +\eps\kappa^2  \abs{(\zeta_+,\zeta_-,\dot\zeta_+,\dot\zeta_-)}+\frac{1}{\kappa^{1/2}}\abs{(\zeta,q)}_2 +\eps\abs{f}_2+\eps\abs{(g_1,g_2)}\big);
\end{equation}
the main difference with \eqref{estQD} in the proof of Theorem \ref{theostab} is the presence in the right-hand side of a term involving the traces $\zeta_\pm$ and their time derivatives, but the strategy of the proof is quite similar.
Recalling that $q_{\rm i}=-x\dot\delta + \av{q_{\rm i}}$, it suffices to prove that $\abs{\ddot\delta}$ and $\abs{\frac{d}{dt}\av{q_{\rm i}}}$ are bounded from above by the right-hand side of \eqref{estqpm}.  Following a procedure similar to the one used to derive  \eqref{estQD}, we get, using the fact that $-\kappa^2\dx^2 R_1= 1-R_1$ in \eqref{dtdxq0} and with the definition \eqref{defGlin} of $\underline{\mathfrak G}$ that, on $\cE^\pm$,
$$
 \zeta+\underline{\mathfrak G} +\eps\kappa^2 \big( \underline{\mathtt c}_2\zeta+\dt ( \underline{\mathtt c}_1 \zeta)\big)= \underline{\mathfrak H}-\eps\kappa^2  \frac{1}{\uh}\dx R_0 \widetilde{f} \pm \kappa \frac{1}{\uh}\dot{q}_\pm\exp\big( - \frac{1}{\kappa}\abs{x\mp \ell} \big),
$$
with
$$
\underline{\mathfrak H}=- \eps \frac{\uq}{\uh^2}  q + \eps\kappa^2 \big( \underline{\mathtt c}_2\zeta+\dt ( \underline{\mathtt c}_1 \zeta)\big)+
\frac{1}{\uh}R_1\underline{\mathfrak f}_{\rm sw}.
$$
Replacing $\zeta+ \underline{\mathfrak G} +\eps\kappa^2 \big( \underline{\mathtt c}_2\zeta+\dt ( \underline{\mathtt c}_1 \zeta)\big) $ by the above expression in \eqref{transmlin1} and \eqref{transmlin2} yields the following linearized version of \eqref{ODEIBVP},
\begin{align}
\nonumber
\underline{\mathfrak T}_\mu \frac{d}{dt}\left(\begin{array}{c} \av{q_{\rm i} }\\ \dot \delta \end{array}\right)
&+
\left(\begin{array}{c} 0\\ \delta \end{array}\right)
+
\eps \left( \begin{array}{c}
{\mathbf b}_1[\underline{\mathtt Z}]\cdot {\mathtt Z}\\
{\mathbf b}_2[\underline{\mathtt Z}]\cdot {\mathtt Z}
\end{array}\right)\\
\label{ODEIBVPbis}
&=
\left( \begin{array}{c}
-  \frac{1}{2\ell}\jump{\underline{\mathfrak H}   - \eps \kappa^2  \frac{1}{\uh}\dx R_0 \widetilde{f}      }+ \eps g_1 \\
\av{\underline{\mathfrak H}   - \eps \kappa^2  \frac{1}{\uh}\dx R_0 \widetilde{f}   } +\eps g_2
\end{array}\right),
\end{align}
where $\underline{\mathfrak T}_\mu ={\mathfrak T_\mu}(\eps\udelta,\eps \underline{\zeta}_\pm)$, see \eqref{defM}.
From the above definition of $\underline{\mathfrak H}$, one gets with the trace estimate \eqref{traceR1} that
$$
\abs{\underline{\mathfrak H}_\pm}\leq \underline{C} \times \big( \eps \abs{{\mathtt Z}}+\eps \kappa^2\abs{(\zeta_\pm,\dot\zeta_\pm)}+\frac{1}{\kappa^{1/2}}\abs{(\zeta,q)}_2\big).
$$
Inverting the matrix $\underline{\mathfrak T}_\mu$, we therefore get
$$
\vert  \frac{d}{dt} \av{q_{\rm i}}
 \vert+\abs{\ddot \delta}
\leq  \underline{C} 
\big( \abs{{\mathtt Z}} +\eps\kappa^2  \abs{(\zeta_+,\zeta_-,\dot\zeta_+,\dot\zeta_-)}+\frac{1}{\kappa^{1/2}}\abs{(\zeta,q)}_2 +\eps \abs{f}_2+\eps\abs{(g_1,g_2)}\big),
$$
and we thus obtain \eqref{estqpm}.\\
From \eqref{eqnrjfirst}, we therefore get
\begin{align*}
\frac{d}{dt}\big[  \underline{\mathfrak E}_{\rm ext} +\underline{\mathfrak E}_{\rm int}
+&\eps \kappa^2 \big( \av{q_{\rm i}}\jump{  c_1 \zeta  }-2  \ell \dot\delta\av{  c_1\zeta }\big) \big]
\leq  \eps  \underline{C}\Big(  \frac{1}{\eps^{1/2}}\underline{\mathfrak E}_{\rm ext} +(1+\frac{\kappa}{\eps^{1/2}})\underline{\mathfrak E}_{\rm int}\\
&+\abs{f}_2^2+\abs{(g_1,g_2)}^2
+\kappa^{3}\eps^{1/2}\abs{(\zeta_-,\zeta_+)}^2+\eps^{3/2}\kappa^{5}\abs{(\dot \zeta_-,\dot\zeta_+)}^2\Big).
\end{align*}
In order to control the singular $\eps^{-1/2}$ term in front of $\underline{\mathfrak E}_{\rm ext}$, which is due to the subprincipal terms in the linearized transmission conditions, one has to change the $\varepsilon$ in front of the right-hand side into a $\eps^{1/2}$ (this is the reason why the estimate of the theorem is only valid over a $O(\eps^{-1/2})$ time scale), leading to
\begin{align}
\nonumber
\frac{d}{dt}\big[  \underline{\mathfrak E}_{\rm ext} +\underline{\mathfrak E}_{\rm int}
+&\eps \kappa^2  \big( \av{q_{\rm i}}\jump{  c_1 \zeta  }-2  \ell \dot\delta\av{  c_1\zeta }\big) \big]
\leq  \eps^{1/2}  \underline{C}\Big(  \underline{\mathfrak E}_{\rm ext} +\underline{\mathfrak E}_{\rm int}\\
\label{eqnrjsecond}
&+\eps^{1/2}\abs{f}_2^2+\eps^{1/2}\abs{(g_1,g_2)}^2
+\eps\kappa^{3}\abs{(\zeta_-,\zeta_+)}^2+\eps^{2}\kappa^{5}\abs{(\dot \zeta_-,\dot\zeta_+)}^2\Big);
\end{align}
contrary to Step 4 in the proof of Theorem \ref{theostab}, this inequality is not enough to derive an energy estimate; we still need to  find a control on  the trace terms $\eps^{1/2}\kappa^{3/2} \abs{\zeta_\pm}$ and $\eps\kappa^{5/2}\abs{\dot \zeta_\pm} \lesssim \eps^{1/2}\kappa^{5/2}\abs{\dot \zeta_\pm }$ that appear in the right-hand side of \eqref{eqnrjsecond}; such a control is also necessary to absorb the non signed perturbation of the energy that appears in the left-hand side. 

\noindent
{\bf Step 3}. Control on $\eps^{1/2}\kappa^{3/2} \abs{\zeta_\pm}$ and $\eps^{1/2}\kappa^{5/2}\abs{\dot \zeta_\pm }$. Introducing a trace energy as
$$
\underline{{\mathfrak E}}_{\rm trace}:=\frac{1}{2}\big[  \kappa^2(\eps^{1/2}\kappa^{3}\dt \zeta_\pm)^2+(\uh-\eps^2 \frac{\uq^2}{\uh^2})_\pm(\eps^{1/2}\kappa^{3}\zeta_\pm)^2        \big],
$$
we show here that
\begin{equation}\label{estnrjtrace}
\frac{d}{dt}\underline{{\mathfrak E}}_{\rm trace}\leq \eps^{1/2} \underline{C}
\big( \underline{\mathfrak E}_{\rm ext}+ \underline{\mathfrak E}_{\rm int}+ \underline{\mathfrak E}_{\rm trace}+\eps^{1/2}\abs{f}^2_2+\eps^{1/2}\abs{(g_1,g_2)}^2_2 \big).
\end{equation}
Recalling that $\dt^2\zeta=-\dt\dx q$, one gets, evaluating \eqref{dtdxq0} at $x=\pm \ell$, that
$$
\dt^2 \zeta_\pm+\frac{1}{\kappa^2} \big( \uh-\eps^2 \frac{\uq^2}{\uh^2}\big)_\pm\zeta_\pm=-\frac{\eps}{\kappa^2}(2 \frac{\uq}{\uh})_\pm q_\pm+\frac{1}{\kappa^2} (R_1\underline{\mathfrak f}_{\rm sw})_\pm-\eps (\dx R_0 \tilde f)_\pm \pm \frac{1}{\kappa} \dot q_\pm
$$
Since we want a control on $\eps^{1/2}\kappa^{3/2} \zeta_\pm$, we multiply  both sides of the equation by $\eps \kappa^5 \dt \zeta_\pm$, using the trace estimate \eqref{traceR1} and observing that $\abs{q_\pm}\lesssim \abs{\mathtt Z}$, one readily deduces \eqref{estnrjtrace}.

\noindent
{\bf Step 4.} Conclusion. Summing up \eqref{eqnrjsecond} and \eqref{estnrjtrace}, one obtains
\begin{align}
\nonumber
\frac{d}{dt}\big[  \underline{\mathfrak E}_{\rm ext} +\underline{\mathfrak E}_{\rm int}+ \underline{\mathfrak E}_{\rm trace}
&+\eps \kappa^2  \big( \av{q_{\rm i}}\jump{  \underline{\mathtt c}_1 \zeta  }-2  \ell \dot\delta\av{   \underline{\mathtt c}_1\zeta }\big) \big]\\
\label{eqnrjthird}
&\leq  \eps^{1/2}  \underline{C}\big(  \underline{\mathfrak E}_{\rm ext} +\underline{\mathfrak E}_{\rm int}
 +\underline{\mathfrak E}_{\rm trace}
 +\eps^{1/2}\abs{f}_2^2+\eps^{1/2}\abs{(g_1,g_2)}^2\big).
\end{align}
We can now notice that $\eps \kappa^2  \big( \av{q_{\rm i}}\jump{   \underline{\mathtt c}_1 \zeta  }-2  \ell \dot\delta\av{   \underline{\mathtt c}_1\zeta }\big) $ is a lower order term in the sense that 
$$
\eps \kappa^2  \abs{ \av{q_{\rm i}}\jump{  c_1 \zeta  }-2  \ell \dot\delta\av{  c_1\zeta } } 
\leq
\eps^{1/2}\kappa^{1/2} \underline{C}\big( \underline{{\mathfrak E}}_{\rm int}+\underline{{\mathfrak E}}_{\rm trace} \big),
$$
so that it can be absorbed by the sum of the three energies when $\eps\kappa$ is small enough to have $\eps^{1/2}\kappa^{1/2} \underline{C} <1$. For instance, if $\eps^{1/2}\kappa^{1/2} \underline{C} <1/2$, and denoting
$$
\widetilde{\underline{\mfE}}:=\underline{\mathfrak E}_{\rm ext} +\underline{\mathfrak E}_{\rm int}+ \underline{\mathfrak E}_{\rm trace},
$$
one obtains after a Gronwall estimate
$$
\widetilde{\underline{\mfE}}(t)\leq 3 \big[ \widetilde{\underline{\mfE}}(0)+\eps \underline{C}\int_0^t \big( \abs{f}_2^2+\abs{(g_1,g_2)}^2\big) \big]
\exp(\sqrt{\eps}\underline{C}t).
$$
Since moreover there exists a constant $C_0=C_0(\frac{1}{h_{\rm min}},\frac{1}{c_{\rm min}})$ such that
$$
\widetilde{\underline{\mathfrak E}} \leq C_0 \big( \abs{{\mathtt Z}}^2+\abs{(\zeta,q,\kappa\dx q)}_2^2+\eps\kappa^3\abs{(\zeta_-,\zeta_+)  }^2+
\eps\kappa^5\abs{(\dot \zeta_-,\dot\zeta_+)  }^2
$$
and
$$
 \abs{{\mathtt Z}}^2+\abs{(\zeta,q,\kappa\dx q)}_2^2+\eps\kappa^3\abs{(\zeta_-,\zeta_+)  }^2+
\eps\kappa^5\abs{(\dot \zeta_-,\dot\zeta_+)  }^2 \leq C_0 \widetilde{\underline{\mathfrak E}},
$$
one deduces the estimate stated in the theorem.

\noindent
{\bf Step 5.} Well-posedness. By a straightforward adaptation of the proof of Theorem \ref{theoexist}, one can observe that \eqref{Boussinesq_lin}-\eqref{transmlin2}  can be reformulated as an ODE for $(\zeta,q,{\mathtt Z})\in H^1\times H^2(\cE)\times \RR^3$ and prove existence and uniqueness of a solution in this space by Cauchy-Lipschitz's theorem. For data in $(\zeta,q,{\mathtt Z})\in L^2\times H^1(\cE)\times \RR^3$, this strategy does not work directly because the traces $\zeta_\pm$ that appear in the component of the ODE  \eqref{ODEIBVPbis} for $\av{q_{\rm i}}$ and $\dot{\delta}$ cannot be controlled by the $L^2$ norm of $\zeta$. However, the energy estimate just proved provides such a control and one can obtain the result by a classical density argument (as used for instance in the proof of Theorem 3.1.1 in \cite{Metivier2001}, for hyperbolic initial boundary value problems where the control on the trace is furnished by using a Kreiss symmetrizer).
\end{proof}

\section{Return to equilibrium} \label{sectreturn}

We now deal with a specific kind of wave-structure interaction that was called the  return to equilibrium problem in \cite{Lannes_float} and is commonly referred to as "free decay test" in engineering. This a situation where  the solid is released at zero speed from an out of equilibrium position ($\delta (t=0)\neq 0$), in a fluid that is at  rest. The solid then oscillates vertically and its motion sends waves outwards; by this process, the solid loses energy and its oscillations are damped so that the solid asymptotically stabilizes to its equilibrium position. Engineers use this free decay test because by measuring the oscillations of the object, they deduce some buoyancy properties of the object. More precisely, assuming that the motion of the object satisfies the phenomenological Cummins equation \cite{Cummins,Wamit}
\begin{equation}\label{cummins-ing}
M \ddot \delta  + k \ast \dot\delta + a \delta= 0
\end{equation}
with $ M,a \in \mathbb{R}^+$ and $k \in L^1_{\text{loc}}(\mathbb{R}^+)$, they calibrate these coefficients with experimental measurements. These measurements are also used to propose nonlinear extensions to \eqref{cummins-ing} (by fitting coefficients with ad hoc nonlinear terms) \cite{NLCummins}.

Our goal in this section is to study this problem from a mathematical viewpoint, by proposing a qualitative analysis of the solutions to the transmission problem \eqref{Boussinesq_ext2}-\eqref{trans3} in the particular configuration corresponding  to the return to equilibrium problem. This approach is expected to lead in some cases to an equation of the form  \eqref{cummins-ing}, which would provide an analytic description of the coefficients involved, and also to nonlinear extensions  that could be of interest to engineers.

This program was initiated and achieved in \cite{Lannes_float} for the (non dispersive) nonlinear shallow water equations, where it was found that $\delta$ solves a nonlinear second order ODE without integro-differential term. Still working with the shallow water equations but in horizontal dimension $d=2$, assuming radial symmetry and neglecting the nonlinear effects in the exterior region, it was shown in \cite{Bocchibis} that the equation on $\delta$ should contain an integro-differential term. Such a term is also necessary for the nonlinear shallow water equations in dimension $d=1$ if viscosity is taken into account \cite{Tucsnak}. The goal of this section is to investigate the contribution of the dispersive terms of the Boussinesq system to the equation satisfied by $\delta$ in this specific configuration of the return to equilibrium problem.

From now, we  assume that the initial data correspond to the configuration of the return to equilibrium problem, namely,
\begin{equation}\label{CIRE}
q(t=0)=\zeta(t=0) =0
\quad \text{and} \quad
\delta(t=0)= \delta_0, \quad \dot\delta(t=0)=0.
\end{equation}

\medbreak

\begin{notation}\label{notapart4}
We use throughout this section the same notations as in Section \ref{sectWSITP}, namely, we write $\kappa=\sqrt{\mu/3}$ and denote by $\mfsw$ the momentum flux of the nonlinear shallow water equations,
$$
{\mathfrak f}_{\rm sw}=\frac{h^2-1}{2\eps}+\eps \frac{q^2}{h}=\zeta+\eps\big(\frac{1}{2}\zeta^2+\frac{q^2}{h}\big).
$$
We also recall that the buoyancy frequency $\tau_{\rm buoy}$ is defined in Appendix \ref{appND}.
\end{notation}

We introduce in \S \ref{sectgCummins} two Cummins operators that allow us to derive an abstract evolution equation for the solid. We then investigate two specific cases where it is possible to derive an explicit expression of these operators. The non dispersive case ($\eps\neq 0$, $\mu=0$) is considered in \S \ref{sectnondisp} where it is shown that the motion of the object can be found by solving a simple nonlinear second order scalar ODE. Waves can then be described by solving an initial boundary value problem for a scalar Burgers equation. The opposite case, namely, the linear dispersive case ($\eps=0$, $\mu\neq 0$) is addressed in \S \ref{sectlinear}; here again, it is possible to derive an explicit expression for the Cummins operators leading us to an integro-differential Cummins-type equation for the motion of the solid; qualitative properties of the solutions, such as their decay rate are then investigated. Finally, it is shown that the motion of the waves can be found by solving a nonlocal (in space) perturbation of the transport equation.

\subsection{The general Cummins equation}\label{sectgCummins}

Quite obviously, any smooth solution of the transmission problem  \eqref{Boussinesq_ext2}-\eqref{trans3} with initial condition \eqref{CIRE} is such that $\zeta$ is an even function while $q$ is odd -- {\it such solutions will be called ``symmetric"}. This implies that $\av{q_{\rm i}}=0$ and that the transmission problem can be reduced into a simpler  boundary value problem stated in the following direct corollary of Theorem \ref{theotransm}. 
\begin{corollary}\label{corotransm-ret}
Any smooth symmetric solution to the transmission problem \eqref{Boussinesq_ext2}-\eqref{trans3} solves the following boundary value problem on the half-line $(\ell,\infty)$,    
\begin{equation}\label{IBVP_ext2}
\begin{cases}
\partial_t \zeta +\partial_x q =0 
\\
(1-\kappa^2 \partial_x^2) \partial_t q + \dx \mfsw
= 0,
\end{cases}\quad\mbox{ for }\quad t\geq 0,\quad x\in {\mathcal E}^+,
\end{equation}
with boundary condition
\begin{align}
\label{coro1R}
q_{\vert_{x=\ell}}&=-\ell \dot\delta,
\end{align}
where $\delta$ solves the  ODE
\begin{equation}\label{coroODEdeltaR}
{(\tau_\mu(\eps\delta)^2 + \ell \kappa \frac{1}{h_+})}\ddot\delta+\delta =\eps \beta(\eps\delta)\dot\delta^2 
 + {\mathfrak H}_{+},
\end{equation}
\noindent
where $h_+=h_{\vert_{x=\ell}}$, $ {\mathfrak H}_{+}= {\mathfrak H}_{\vert_{x=\ell}}$ and we recall that ${\mathfrak H}={\mathfrak H}(\zeta,q)$ with
$$
{{\mathfrak H}(\zeta,q)=
\frac{1}{2}\eps \big( \frac{1}{h}\zeta^2- \frac{q^2}{h^2}\big) +\frac{1}{h} R_1\mfsw},
$$
and that $\tau_\mu(\eps\delta)$ and $\beta(\eps\delta)$ are defined in Proposition \ref{propdelta}, namely,
\begin{align*}
{{{\tau}_\mu(\eps\delta)^2}}&= \tau_{\rm buoy}^2 +\frac{1}{\ell}\int_{0}^\ell \frac{x^2}{h_{\rm eq}(x)+\eps\delta}{\rm d}x + {\frac{\kappa^2}{h_{\rm eq}(\ell)+\eps\delta}}, \\
\beta(\eps\delta)&= \frac{1}{2} \frac{1}{\ell}\int_{0}^\ell \frac{x^2}{(h_{\rm eq}(x)+\eps\delta)^2}{\rm d}x.
\end{align*} 
\end{corollary}
We know by Proposition \ref{existence-presc} that if $f$ is a given $C^1$ function of time then there is a unique solution $(\zeta,q)$ to \eqref{IBVP_ext2} with boundary condition $q_{\vert_{x=\ell}}=-\ell f$ with initial condition corresponding to the return to equilibrium problem, namely, $(\zeta,q)(t=0)=(0,0)$. It is in particular possible to compute the trace of $\zeta$ at $x=\ell$, so that the following definition makes sense.
\begin{definition}[Cummins operators]\label{def-Cummins-operator}
Let $\eps\in \mathbb{R}^+$, $\mu=\kappa^2/3 >0$. Let also $f\in C^1(\RR^+)$ and $T>0$, and  $(\zeta,q )\in C^1\big([0,T;H^1(\cE^+)\times H^2(\cE^+) \big)$ be a solution to \eqref{IBVP_ext2} with boundary condition $q_{\vert_{x=\ell}}=-\ell f$ and initial condition $(\zeta,q)(t=0)=(0,0)$.
We define the {\it Cummins operators} ${\mathfrak c}_{\eps,\mu}$  and ${\mathfrak C}_{\eps,\mu}$ as
$$
{\mathfrak c}_{\eps,\mu}[f]:= \zeta_{\vert_{x=\ell}}
\quad\mbox{ and }\quad
 {\mathfrak C}_{\eps,\mu}[f]:= {{-}} {{\mathfrak H}(\zeta,q)_{\vert_{x=\ell}}}.
$$
\end{definition}
\begin{remark}
The Cummins operators can be defined for more general cases, for instance, the solution $(\zeta,q)$ to the initial boundary value problem needs only to be regular near the boundary $x=\ell$ (regular enough for the trace to make sense). This allows one to extend the definition of the Cummins operators in the case $\mu=0$, as done in \S \ref{sectnondisp} below.
\end{remark}

\begin{corollary}
The ODE \eqref{coroODEdeltaR} can be reformulated in a compact form as what we shall refer to as the Cummins equation
\begin{equation}
\label{eqCummins}
{{\big(\tau_\mu(\eps\delta)^2+ \ell \kappa \frac{1}{1+\eps {\mathfrak c}_{\eps,\mu}[\dot\delta]}\big)}}\ddot\delta+\delta {{{+ \mathfrak C}_{\eps,\mu}[\dot\delta]}}=\eps \beta(\eps\delta)\dot\delta^2 ,
\end{equation}
with initial conditions $\delta(0)=\delta_0$ and $\dot\delta(0)=0$. 
\end{corollary}

The equation \eqref{eqCummins} is compact but not simple since the Cummins operators are { nonlinear nonlocal operators which require the resolution of the equations for the fluid in the exterior domain. In order to get some qualitative insight on the Cummins equation, we describe it in two limiting cases: in the nonlinear non dispersive case ($\eps> 0$, not necessarily small, and $\mu=0$), and in the linear, dispersive case ($\eps=0$ and $\mu>0$, not necessarily small). Note that in both cases, it is not necessary to compute the first Cummins operator ${\mathfrak c}_{\eps,\mu}[\dot\delta]$ and that it is possible to provide an explicit expression of the second one ${\mathfrak C}_{\eps,\mu}[\dot\delta]$.


\subsection{The nonlinear non dispersive case}\label{sectnondisp}

Neglecting the dispersive effects is equivalent to setting $\mu=\kappa^2/3=0$ in the equations \eqref{IBVP_ext2}-\eqref{eqCummins} ; in particular, the model considered for the propagation of the waves is now the shallow water equations
\begin{equation}\label{SW}
\begin{cases}
\partial_t \zeta +\partial_x q =0 
\\
 \partial_t q + \varepsilon \partial_x \left( \frac{1}{h} q^2 \right) + h \partial_x \zeta
= 0,
\end{cases}\quad\mbox{ for }\quad t\geq 0,\quad x\in {\mathcal E}^+,
\end{equation}
the boundary condition is unchanged
\begin{align}
\label{CBSW}
q_{\vert_{x=\ell}}&=-\ell \dot\delta,
\end{align}
and the  ODE solved by $\delta$ is simplified into
\begin{equation}\label{ODESW}
{{\tau_0(\eps\delta)^2}}\ddot\delta+\delta{{{+ \mathfrak C}_{\eps,0}[\dot\delta]}} =-\eps  \tau_0(\eps\delta) \tau_0'(\eps\delta)\dot\delta^2,
\end{equation}
where we used the fact that $\beta(\eps\delta)=-2  \tau_0(\eps\delta)\tau_0'(\eps\delta)$ when $\mu=0$ (see \eqref{defmm} and \eqref{defbeta}), and where the definition of the second Cummins operator has been extended to the case $\mu=0$ as
$$
 {\mathfrak C}_{\eps,0}[\dot\delta]:={{-}}\Big(\zeta+\eps\frac{1}{2}\frac{q^2}{h^2}\Big)_{\vert_{x=\ell}};
 $$
 the fact that this definition makes sense follows from the decomposition of the shallow water invariants into Riemann invariants, as shown in the proof of the following theorem where an explicit expression of the Cummins operator is provided. This theorem is a reformulation of Corollary 1 in \cite{Lannes_float}, but with a slight difference in the function $\gamma$, so that we reproduce a sketch of the proof\footnote{The difference comes from the fact that in \cite{Lannes_float}, the choice of the boundary condition for the interior pressure was made by assuming that the jump of pressure at the contact point was purely hydrostatic; as in \cite{Tucsnak,Bocchi}, we rather use here a choice of the boundary condition on the pressure which is consistent with the approach used throughout this paper and motivated by the conservation of total energy, as explained in Corollary \ref{coroBC}. With the choice of \cite{Lannes_float}, one would have ${\mathfrak C}_{\eps,0}[\dot \delta]=-\zeta_{\vert_{x=\ell}}$ and consequently $-\ell \dot\delta - \eps\dot\delta^2 \gamma(\eps\dot \delta)=\frac{1}{\eps}(\sigma_0(\eps\frac{\ell}{2}\dot\delta)^2-1)    $.}.
 \begin{theorem}\label{theoCumminsSW}
Let $T>0$, $\delta\in C^2([0,T])$ and $(\zeta,q)$ be a continuous, piecewise $C^1$ solution of \eqref{SW}-\eqref{ODESW} on $[0,T]\times(\ell,\infty)$ satisfying the non vanishing depth condition
$$
\inf_{[0,T]\times \cE} h>0\quad\mbox{ and }\quad \inf_{[0,T]\times {\mathcal I}} h_{\rm eq}+\eps \delta>0.
$$
If moreover $\ell \, \eps \dot \delta < 2 r_0$, with $r_0:=\frac{4}{27}$, we have $\sqrt{h}= \sigma_0(\eps\frac{\ell}{2}\dot\delta)$ with the real function
$$
\sigma_0(r)=\frac{1}{3}\Big( 1+C_-(r)+C_+(r) \Big), \quad
C_\pm(r)=\frac{3}{2}\Big( -4 r + 2r_0 \pm 4 \sqrt{r(r-r_0)}\Big)^{1/3},
$$
and the Cummins operator ${\mathfrak C}_{\eps,0}$ is given explicitly by
\begin{equation}\label{CumminsopSW}
 {\mathfrak C}_{\eps,0}[\dot\delta]= - \eps^{-1} \Big(\sigma_0(\eps \frac{\ell}{2} \dot\delta) -1\Big) \Big(3 \sigma_0(\eps \frac{\ell}{2} \dot\delta) -1\Big) 
 =:\ell \dot\delta {{+}} \eps\dot\delta^2 \gamma(\eps\dot \delta),
\end{equation}
where $\gamma: (-\infty, 2 r_0)\to {\mathbb R}$ is a smooth function such that $\gamma(0)=\frac{1}{4}\ell^2$ and whose exact expression is given in \eqref{defnu} below.
\end{theorem}
\begin{proof}
The proof of Corollary 1 of \cite{Lannes_float} is based on the fact that the  shallow water equations can be put in diagonal form,
$$
\dt R+ (\sqrt{h}+\eps \frac{q}{h})\dx R=0 \quad \mbox{ and }\quad \dt L- (\sqrt{h}-\eps \frac{q}{h})\dx L=0 ,
$$
 where $R$ and $L$ are respectively the right and left Riemann invariants
$$
R= \frac{q}{h}+\frac{2}{\eps}(\sqrt{h}-1)\quad\mbox{ and }\quad L= \frac{q}{h}-\frac{2}{\eps}(\sqrt{h}-1).
$$
One then notices that with the initial and boundary conditions considered here, $L$ vanishes identically on $(\ell, \infty)$, which allows one to find $\sqrt{h}$ in terms of $q$ as a root of the third order polynomial equation in $\sigma$,
$$
\sigma^3-\sigma^2 - \eps \frac{1}{2} q=0.
$$
If $-\eps \frac{1}{2} q<r_0$, then, as discussed in \cite{Lannes_float}, the relevant root is $\sigma_0(-\eps \frac{1}{2} q)$. \\
Recalling that $q_{\vert_{x=\ell}}=-\ell \dot \delta$, we have $\sqrt{h}=\sigma_0(\eps \frac{\ell}{2} \dot\delta)$. Moreover, $L=0$ implies
$$
\eps\frac{1}{2}\frac{q^2}{h^2}_{\vert_{x=\ell}}= 2 \, \eps^{-1} \Big( \sigma_0(\eps \frac{\ell}{2} \dot\delta) -1 \Big)^2.
$$
Remarking that $\zeta_{\vert_{x=\ell}}=\frac{1}{\eps}(\sigma_0(\eps \frac{\ell}{2}\dot\delta)^2-1)$, one gets
\begin{align}
\nonumber
 \Big(\zeta+\eps\frac{1}{2}\frac{q^2}{h^2}\Big)_{\vert_{x=\ell}} &=  \eps^{-1} (\sigma_0 -1) (3 \sigma_0 -1)\\
 \label{defnu}
 &=: -\ell \dot\delta {{-}} \eps\dot\delta^2 \gamma(\eps\dot \delta),
\end{align}
where we used the fact that $\sigma_0(0)=-\sigma_0'(0)=1$. The fact that $\gamma(0)=\frac{1}{4}\ell^2$ follows from the observation that $\sigma_0''(0)=-4$.
\end{proof}

A first corollary is that the motion of the solid can be reduced to a simple nonlinear ODE, provided that the initial displacement satisfies an upper bound ensuring that the velocity of the object does not become too big.
\begin{corollary}\label{corocoro}
Under the assumptions of the theorem, and with the same notations, let us assume moreover that
$$
 \eps^2 \delta_0^2 <  {\tau_0(\eps \abs{\delta_0})^2}  \big(\frac{2r_0}{\ell}\big)^2.
$$
Then, using the notations of the theorem, the motion of the solid is found by solving the nonlinear second order ODE 
\begin{equation}\label{ODESW2bis}
{{\tau_0(\eps\delta)^2}}\ddot\delta+\ell \dot\delta+\delta +\eps {{\left( \tau_0(\eps\delta)\tau'_0(\eps\delta)
 + \gamma (\eps\dot\delta) \right)}} \dot\delta^2 =0,
\end{equation}
with initial condition $\delta(0)=\delta_0$, $\dot\delta(0)=0$.
\end{corollary}
\begin{remark}\label{remJohn}
In the linear case ($\eps=0$), this equation is almost the same as (3.2.12)  in \cite{John1}, the only difference being that the author neglected the buoyancy frequency $\tau_{\rm buoy}$ in the expression for $\tau_0(0)$.
\end{remark}
\begin{proof}
One just needs to check that the condition $\ell\eps\dot\delta<2r_0$, which ensures by Theorem \ref{theoCumminsSW} that the Cummins operator takes the form \eqref{CumminsopSW}, is satisfied for all times. Since at $t=0$, one has $\dot\delta=0$, we now that this condition is satisfied for small times. Since moreover one can deduce from Proposition \ref{prior-conservation-case} (by setting $\mu=0$) that
\begin{equation}\label{decaynrj}
\tau_0(\eps\delta)^2\dot\delta^2+\delta^2 \leq \delta_0^2,
\end{equation}
one deduces that $\abs{\delta}\leq\abs{\delta_0}$ and therefore that $\dot\delta^2\leq \tau_0(\eps\abs{\delta_0})^{-2}\delta_0^2$. The assumption made in the statement of the corollary therefore grants the result.
\end{proof}

The interest of reducing the motion of the solid to an ODE on  the surface displacement is that it is possible to solve it even in situations when singularity arise in the exterior domain (typically, when shock happen). It is in particular possible to obtain a global existence result for the ODE \eqref{ODESW2bis}, while such a result cannot be expected for strong solutions to the full transmission problem \eqref{SW}-\eqref{ODESW} due to shock formation. Note that the first condition on $\delta_0$ means that at $t=0$, the solid neither touches the bottom nor is lifted from a height greater than the height of the water column under the object when it is at equilibrium.
\begin{proposition}\label{propexistODE}
Let $\delta_0\in \RR$ be such that
$$
\inf_{{\mathcal I}}h_{\rm eq}-\eps \abs{\delta_0}>0
\quad\mbox{ and }\quad  \eps \abs{\delta_0} < \tau_0(\eps \abs{\delta_0})  \frac{2r_0}{\ell}.
$$
Then there exists a unique global solution  $\delta\in C^\infty(\RR^+)$ to the ODE \eqref{ODESW2bis} with initial condition $(\delta,\dot\delta)_{\vert_{t=0}}=(\delta_0,0)$. 
\end{proposition}
\begin{remark}
A byproduct of the proof is that ${\mathfrak C}_{\eps,0}[\dot\delta] \, \dot\delta \geq 0$ and that this quantity corresponds to the energy transferred at each instant to the exterior fluid domain, that is, with the notations of Proposition \ref{prior-conservation-case}, one has
$$
\frac{d}{dt}{\mathfrak E}_{\rm ext}={\mathfrak C}_{\eps,0}[\dot\delta]\dot\delta\geq 0.
$$
\end{remark}
\begin{remark}
The second condition of the proposition is a smallness condition on $\delta_0$, but this condition is not restrictive as it allows $\delta_0$ to be of size $O(\eps^{-1})$. As communicated to us by the author it is possible, under stricter smallness conditions, to prove exponential decay of the solution of ODEs related to \eqref{ODESW2bis} using techniques developed in \cite{Koike}.
\end{remark}
\begin{proof}
There exists a positive time $T>0$ such that on $[0,T)$, there is a solution $\delta$ such that $\inf_{\mathcal I} h_{\rm eq}+\eps \delta >0$ and $\ell\eps\dot\delta<2 r_0$. We want to show that one can take $T=+\infty$. As in the proof of Corollary \ref{corocoro}, this follows from \eqref{decaynrj}. We therefore need to prove that \eqref{decaynrj} holds, without appealing to Proposition \ref{prior-conservation-case} as in the proof of Corollary \ref{corocoro}, but by direct manipulations on the solution to the ODE \eqref{ODESW2bis}. We need the following two lemmas.
\begin{lemma}
The function $\sigma_0$ is  decreasing on $(-\infty,r_0)$.
\end{lemma}
\begin{proof}[Proof of the lemma]
By construction, one has for all $r<r_0$,
$$
\sigma_0(r)^3-\sigma_0(r)^2+r=0;
$$
differentiating this identity yields
$$
\sigma'_0(r)(3\sigma_0(r)^2-2\sigma_0(r) )=-1,
$$
so that $\sigma_0(r)$ and $3\sigma_0(r)^2-2\sigma_0(r)$ have opposite sign. It is therefore enough to prove that $3\sigma_0(r)^2-2\sigma_0(r)>0$ for all $r<r_0$. Since $\sigma_0(0)=1$, this quantity is positive at $r=0$ and must therefore vanish if it changes sign. This means that for some $r_1<r_0$, one must have $\sigma_0(r_1)=0$ or $\sigma_0(r_1)=\frac{2}{3}$. Using the cubic equation solved by $\sigma_0(r_1)$, this implies that $r_1=0$ or $r_1=r_0$. Both cases have to be excluded because $\sigma_0(0)=1\neq 0$ and $r_1<r_0$ by assumption. The result follows.
\end{proof}
\begin{lemma}\label{lemma-pos-cummins-hyperblolique}
If $\dot\delta\neq 0$ and  $\eps\ell\dot\delta <  2 r_0$, the Cummins operator satisfies
$$
 {\mathfrak C}_{\eps,0}[\dot\delta] \, \dot\delta >0.
$$
\end{lemma}

\begin{proof}[Proof of the lemma]
Recalling that from \eqref{CumminsopSW}, one has
$$
 {\mathfrak C}_{\eps,0}[\dot\delta] \, \dot\delta = - \eps^{-1} (\sigma_0(\eps \frac{\ell}{2} \dot\delta) -1) (3 \sigma_0(\eps \frac{\ell}{2} \dot\delta) -1)  \dot \delta,
$$
the conclusion follows from the previous lemma and the observation that
 $ \sigma_0(0)=1 $ and $\sigma_0(r_0)=2/3$.
%
\end{proof}
We can now use the second lemma to conclude: multiplying \eqref{ODESW} by $\dot\delta$ and integrating in time yields
\begin{align*}
\tau_0(\eps\delta)^2\dot\delta^2+\delta^2&=\delta_0^2-\int_0^t {\mathfrak C}_{\eps,0}[\dot\delta]\dot\delta <\delta_0^2,
\end{align*}
which implies \eqref{decaynrj}; the proposition is therefore proved.    \end{proof}

The following Corollary then shows that, once the ODE \eqref{ODESW2bis} has been solved, the solution in the exterior domain reduces to a simple initial boundary value problem for a scalar Burgers-type equation. This is a simple byproduct of the proof of Theorem \ref{theoCumminsSW} where it was shown that the nonlinear shallow water equations were reduced to the scalar equation on the right-going Riemann invariant.
\begin{corollary}\label{coroCumminsSW}
Under the assumptions of Corollary \ref{corocoro},  $q$ is found in the exterior domain by solving the initial boundary value problem
$$
\begin{cases}
\dt q+ \big(- \sigma_0'(-\frac{\eps}{2} q)\sigma_0(-\frac{\eps}{2}q)\big)^{-1}\dx q&=0 \qquad (t>0,\quad x>\ell),\\
q_{\vert_{t=0}}&=0,\\
 q_{\vert_{x=\ell}}&=-\ell \dot\delta,
\end{cases}
$$
with $\delta$ furnished by Proposition \ref{propexistODE},
while $\zeta$ is given in terms of $q$ by the algebraic expression
\begin{equation}\label{zetaqBurg}
\zeta=\frac{1}{\eps}(\sigma_0(-\eps q /2)^2 -1).
\end{equation}
\end{corollary}
\begin{remark}\label{remforced}
More generally, if one wants to compute the waves created by an object in forced motion, one must solve the same equations as in the corollary, but with $\delta$ corresponding to this forced motion rather than given by Proposition \ref{propexistODE}.
\end{remark}
\subsection{The linear dispersive case}\label{sectlinear}

We have studied in the previous section the situation where dispersive effects could be neglected ($\mu=\kappa^2/3=0$) in front of the nonlinear effects. We consider here the opposite situation where nonlinear effects are negligible ($\eps=0$) but the dispersive effects are taken into account. That is, we consider the linear approximation to  \eqref{IBVP_ext2}-\eqref{coroODEdeltaR}. The model considered for the propagation of the waves is therefore 
\begin{equation}\label{linear}
\begin{cases}
\partial_t \zeta +\partial_x q =0 
\\
(1-\kappa^2 \dx^2) \partial_t q +  \partial_x \zeta
= 0,
\end{cases}\quad\mbox{ for }\quad t\geq 0,\quad x\in {\mathcal E}^+,
\end{equation}
the boundary condition is unchanged
\begin{align}
\label{CBlinear}
q_{\vert_{x=\ell}}&=-\ell \dot\delta,
\end{align}
and the  ODE solved by $\delta$ is simplified into
\begin{equation}\label{ODElinear}
 {(\tau^2_\mu+ \ell \kappa)} \ddot\delta+\delta +
  {\mathfrak C}_{0,\mu}[\dot\delta]=0,
\end{equation}
where we recall that according to the definition of the Cummins operator (see Definition \ref{def-Cummins-operator} and \eqref{defC}), 
$$
 {\mathfrak C}_{0,\mu}[\dot\delta]:= -(R_1 \zeta)_{\vert_{x=\ell}},
 $$
 and where, for the sake of clarity, we simply write throughout this section 
 $$
 \tau_\mu^2 =\tau_\mu(0)^2=\tau_{\rm buoy}^2+\frac{1}{\ell}\int_0^\ell \frac{x^2}{h_{\rm eq}}  +\kappa^2 \frac{1}{h_{\rm eq}}.
 $$
 
We know by Theorem \ref{theoexist} that for all $n\in {\mathbb N}$ and $T>0$, there exists a unique solution $(\zeta,q,\delta) \in C^\infty([0,T];{\mathbb H}^n\times \RR)$ of \eqref{linear}-\eqref{ODElinear} with initial conditions \eqref{CIRE}; we want here to analyze the behavior of  this solution.  As for the nonlinear non dispersive case in the previous section, we first provide an explicit expression for the Cummins operator, from which we are able to derive an uncoupled scalar equation for the evolution of $\delta$, whose solution can be used to find $\zeta$ and $q$ in the exterior domain through the resolution of a simpler scalar initial boundary value problem. All the equations involved in this section are linear, the difficulty coming from their nonlocal nature.

\subsubsection{Preliminary material} 

In order to give an explicit representation of the Cummins operator ${\mathfrak C}_{0,\mu}$, we first need to recall the definition of the Bessel functions $J_n$ (\S 8.41 in \cite{Bessel})
$$
J_n(t)=\frac{1}{\pi}\int_0^\pi \cos\big(n\theta- t \sin\theta\big){\rm d}\theta;
$$
we also define the causal} convolution kernels ${\mathcal K}^0_\mu $ and ${\mathcal K}^1_\mu  $ as
\begin{equation}\label{defK0K1}
{\mathcal K}^0_\mu(t)=\frac{1}{\kappa}J_0(\frac{t}{\kappa})
\quad\mbox{ and }\quad
{\mathcal K}^1_\mu(t)= \frac{1}{t}J_1(\frac{t}{\kappa})
, \quad \mbox{ for all } t  \geq 0,
\end{equation}
and use the following standard notation for the convolution of time causal functions,
$$
\forall t\geq 0, \qquad f*g(t)=\int_0^t f(t-s)g(s) {\rm d}s.
$$

We also need to use  the Laplace transform with respect to the time variable, which we define as 
$$\mathcal{L} : q \mapsto \hat q,$$ where
$$
 \mathcal{L} [q] (s) = \displaystyle \int_0^{\infty} q(t) e^{-st} {\rm d}t
 \quad \text{ with }\quad 
 s \in \mathbb{C}_0 := \{ s \in \mathbb{C} \, | \, \frak{Re} (s) >0 \}.
$$
We shall in particular use the following properties on Bessel functions \cite{Bessel}
 \begin{equation}\label{propLB}
 {\mathcal L}^{-1} \big( \frac{1}{\sqrt{1+\kappa^2s^2}}\big)={\mathcal K}^0_\mu(t)
 \quad \mbox{ and }\quad
 {\mathcal L}^{-1}\big( \frac{1}{\sqrt{1+\kappa^2 s^2} + \kappa s} \big) = {\mathcal K}^1_\mu(t),
 \end{equation}
with ${\mathcal K}^0_\mu$ and ${\mathcal K}^1_\mu$ as defined in \eqref{defK0K1}.

\subsubsection{Analysis of the equations}

Using the linear structure of the equations, one can obtain an explicit expression for the Cummins operator ${\mathfrak C}_{0,\mu}$.
\begin{theorem}\label{theoCumminslinear}
The Cummins operator ${\mathfrak C}_{0,\mu}$ is given explicitly by
$$
 {\mathfrak C}_{0,\mu}[\dot\delta]:=\ell \, {\mathcal K}^1_\mu*\dot\delta,
$$
where ${\mathcal K}^1_\mu$ is defined in \eqref{defK0K1}. 
\end{theorem}
\begin{proof}
Applying the Laplace transform to the equations \eqref{linear} and \eqref{CBlinear}, which is possible since all the functions are continuous and bounded in time (as a consequence of Proposition \ref{prior-conservation-case}), and taking into account that $\zeta_{\vert_{t=0}}=q_{\vert_{t=0}}=0$, this yields
\begin{equation}\label{Boussi-linear}
\begin{cases}
s \hat \zeta +\partial_x \hat q =0 
\\
(1-\kappa^2 \partial_x^2) s \hat q + \partial_x \hat \zeta
= 0,
\end{cases}
\quad\mbox{ and }\quad \widehat{q}_{\vert_{x=\ell}}=-\ell \widehat{\dot \delta}.
\end{equation}
This is an ODE for $(\widehat{\zeta},\widehat{q})$ on the half-line $(\ell,\infty)$ that can be explicitly solved in terms $\widehat{\dot\delta}$ (note that a representation of the solution in terms of the Laplace transform in space is also possible \cite{Johnston} but not adapted to our purpose here; see also \cite{Audiard} for other types of linear dispersive equations); the formula of the lemma below provides  "right-going" solutions to the linear Boussinesq equations and it is therefore no surprise that the relationship between $\widehat{\zeta}$ and $\widehat{q}$ is the same as the one that arises when imposing transparent boundary conditions as in \cite{KazakovaNoble}.
\begin{lemma}\label{lemma-jump-linear}
There is one and only one solution $(\widehat{\zeta},\widehat{q})$ to \eqref{Boussi-linear} that does not grow exponentially at infinity; it is given by 
$$
\begin{cases}
\hat{q}(s,x) = \displaystyle - \ell  \,  \widehat{\dot\delta}(s) e^{- \frac{s}{\sqrt{1+ \kappa^2s^2}}(x-\ell)},
\\
\displaystyle \widehat\zeta(s,x)=  \frac{1 }{\sqrt{1+ \kappa^2s^2}} \widehat{q}(s,x),
  \end{cases}
$$
where the square root is taken in order to have positive real part. 
\end{lemma}

\begin{proof} [Proof of the lemma]
From  \eqref{Boussi-linear}, one deduces
\begin{equation}\label{Boussi-lin-Laplace}
\partial_x^2 \hat{q}(s)-\frac{s^2}{1+ \kappa^2s^2} \hat{q}(s)=0,
\end{equation}
and there are therefore two constants $A(s)$ and $B(s)$ such that
\begin{equation*}
\displaystyle
\hat{q}(s,x) = A(s) e^{-\frac{s}{\sqrt{1+ \kappa^2s^2}}x} + B(s) e^{\frac{s}{\sqrt{1+ \kappa^2s^2}}x}.
\end{equation*}
Since exponentially increasing functions are not allowed, we have $B(s)=0$ and thus
\begin{equation*}
\hat{q}(s,x) = A(s) e^{-\frac{s}{\sqrt{1+ \kappa^2s^2}}x}.
\end{equation*}
Then using the boundary condition on $\widehat{q}$ at $x=\ell$, we find the expected formula for $\hat q$. By using the first equation of  \eqref{Boussi-linear}, we get
the formula for $\widehat{\zeta}$. 
\end{proof}

Let us now remark that for all $f\in L^2(\cE^+)$, one has
$$
(R_1 f)_{\vert_{x= \ell}}=\kappa^{-1} \int_{\cE^+} e^{-\kappa^{-1} ( x - \ell)} f(x){\rm d}x
$$
so that, using the lemma,
\begin{align*}
\widehat{{\mathfrak C}_{0,\mu}[\dot\delta]}(s)&= -(R_1 \widehat{\zeta})_{\vert_{x=\ell}}\\
&= \displaystyle \frac{\ell \widehat{\dot\delta}(s) }{\kappa \sqrt{1+ \kappa^2s^2}}  \,  \int_{\cE^+} e^{- \big( \frac{1}{\kappa} + \frac{s}{\sqrt{1+ \kappa^2s^2}} \big) (x-\ell)}{\rm d}x.
\end{align*}
It follows that
$$
\widehat{{\mathfrak C}_{0,\mu}[\dot\delta]}(s)= \frac{\ell}{\sqrt{1+\kappa^2 s^2} + \kappa s} \, \widehat{\dot\delta}(s).
$$
Using \eqref{propLB}, this yields
$$
 {\mathfrak C}_{0,\mu}[\dot\delta]= \ell {\mathcal K}^1_\mu* \dot\delta;
$$
note also for future use that we also get from the lemma that
\begin{align*}
 \zeta(t,x)&={\mathcal K}^0_\mu*q.
 \end{align*}
 \end{proof}

As in Corollary \ref{corocoro} in the non dispersive case, it is possible to determine the motion of the solid by the resolution of a single scalar equation on $\delta$; due to the presence of the dispersive terms however, this equation is no longer an ordinary differential equation but an integro-differential equation.
 \begin{corollary}\label{corocorolin}
The motion of the floating object for the problem \eqref{linear}-\eqref{ODElinear} can be found directly by solving the linear second order integro-differential equation,
\begin{equation}\label{ODESW2}
\big(\tau^2_\mu+\ell\kappa \big)\ddot\delta+\ell {\mathcal K}^1_\mu*\dot\delta+\delta =0,
\end{equation}
with initial conditions $\delta(0)=\delta_0$ and $\dot\delta(0)=0$.
 \end{corollary}
\begin{remark}
   In \cite{Tucsnak}, the authors consider the linearized shallow-water equations with some viscosity $\nu$. More precisely, they consider \eqref{linear} with $-\nu \partial_x^2 q$ instead of
$-\kappa^2 \partial_x^2 \partial_t q$ in the second equation and find the following Cummins equation
\begin{equation}\label{ODEvisc}
\tau_\mu^2 \ddot\delta + \sqrt{\nu} \ell \delta^{\left(\frac{3}{2}\right)} + \nu \dot\delta+\ell {\mathcal F}_\nu \ast \dot\delta + \delta=0
\quad \text{with} \quad
{\mathcal F}_\nu:= \mathcal{L}^{-1} \left[ \frac{1}{\sqrt{1+ \nu s}+ \sqrt{\nu s}} \right],
\end{equation}
and where $\delta^{(\frac{3}{2})}$ stands for the fractional derivative of order $3/2$ of $\delta$. This equation shares some similarities with \eqref{ODESW2}, in particular the convolution term, although with a different kernel (note that one gets $\widehat{\mathcal K}^1_\mu(s)$ by replacing $\nu s$ by $\kappa^2 s^2$ in $\widehat{{\mathcal F}}_\mu(s)$). One the contrary, there is  in \eqref{ODEvisc} a viscous damping term $\nu\dot\delta$ that has no equivalent in \eqref{ODESW2}. Note finally that the fractional derivative term $\sqrt{\nu} \ell \delta^{\left(\frac{3}{2}\right)} $ in \eqref{ODEvisc} can be related to the added mass term ${\ell\kappa \ddot\delta}$ in \eqref{ODESW2}. Indeed, in the analysis of \cite{Tucsnak}, this fractional derivative is the leading order term of a convolution term $\ell F*\dot\delta$ with $\widehat{F}(s)=\sqrt{1+\nu s}$. In the dispersive case, the same analysis would give a symbol $\sqrt{1+\kappa^2 s^2}$ and the leading order term of the same convolution would be the dispersive added mass term $\ell \kappa\ddot\delta$.
\end{remark}  

In the linear non dispersive case ($\eps=\mu=0$), Corollary \ref{corocoro} shows that the motion of the object is governed by the same equation as a damped harmonic oscillator; the return to equilibrium occurs therefore at an exponential rate. In the presence of dispersion, Corollary \ref{corocorolin} states that the motion of the solid is now governed by the integro-differential equation \eqref{ODESW2} and numerical simulations (see Figure \ref{fig2}) suggest that the decay gets slower as the dispersion parameter $\kappa=\sqrt{\mu/3}$ increases.
   \begin{figure}[h]
\begin{center}.
\includegraphics[width=0.9\linewidth]{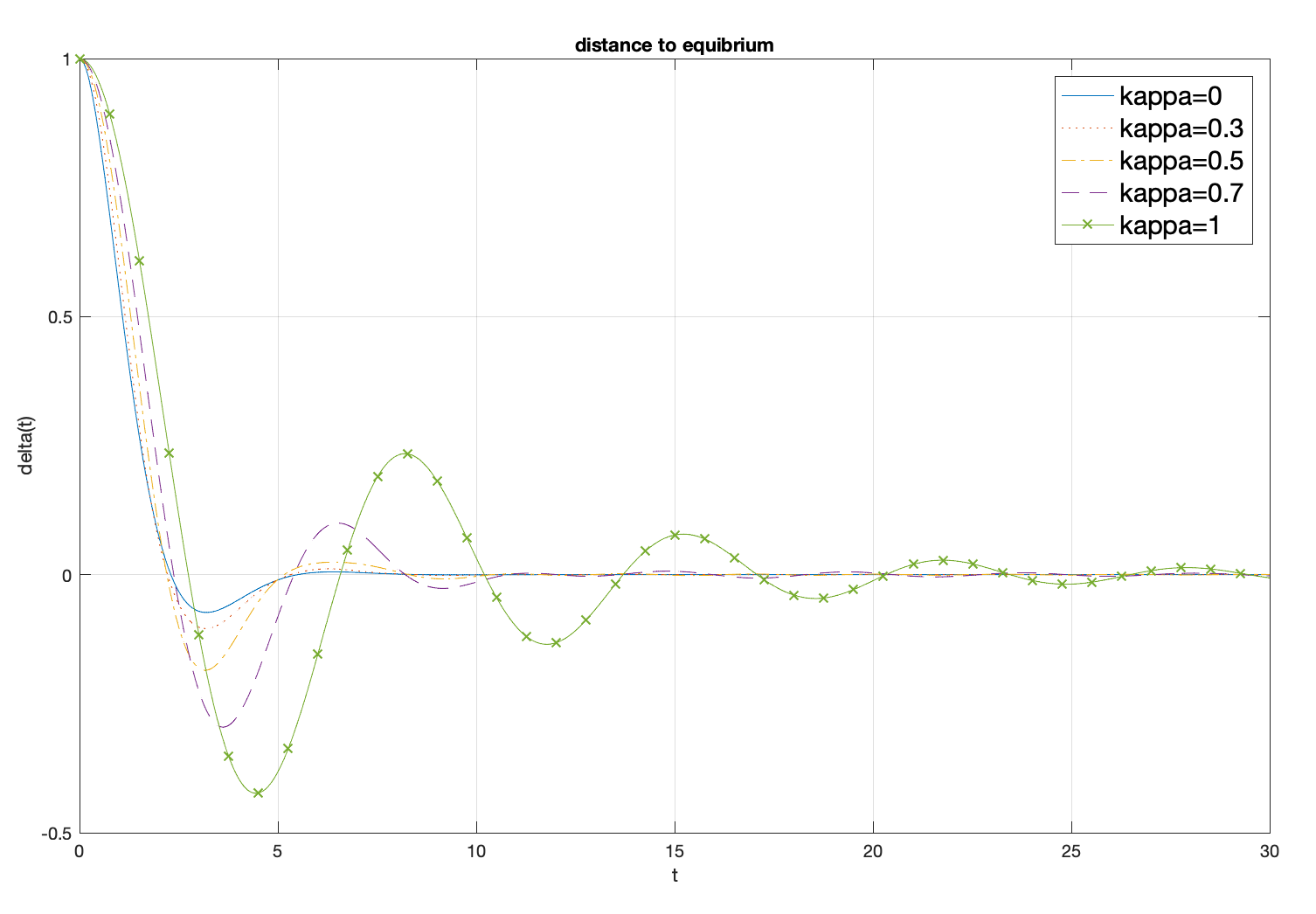}
\end{center}
\caption{Return to equilibrium: evolution of $\delta(t)$ with increasing value of $\kappa$ and $\delta_0=1,\tau_{\rm buoy}=1/6,h_{\rm eq}=1$.}\label{fig2}
\end{figure}
This issue is addressed in the following proposition. In particular, the fact that $\delta$ belongs to $H^2(\RR^+)$ implies that $\delta$ and $\dot\delta$ tend to zero at infinity, but the third point of the proposition shows that the decay cannot be stronger than $O(t^{-3/2})$ (as opposed to the exponential convergence rate in the linear non dispersive case), bringing a theoretical confirmation to the above numerical observations.
   \begin{proposition}\label{decay-polynom}
{\bf i.} There is a unique solution $\delta\in C^2(\RR^+)\cap W^{1,\infty}(\RR^+)$ to \eqref{ODESW2} with initial data $\delta(0)=\delta_0$ and $\dot\delta(0)=0$. \\
{\bf ii.} Moreover, $\delta\in H^2(\RR^+)$, but for $k\in\{0,1,2\}$, $t\delta^{(k)}\not\in L^2(\RR^+)$.\\
{\bf iii.}       For all $\alpha >0$ and $k\in\{0,1,2\}$ and for all $c>0$ and $T_0>0$, there exists $t>T_0$ such that
     $$
  |\delta^{(k)}(t)| > c t^{-\frac{3}{2} - \alpha}.
     $$
   \end{proposition}
   \begin{remark}
   The dispersive delay (convolution)  term in \eqref{ODESW2} is responsible for the slow decay of the solution. Indeed, in the non dispersive limit case $\kappa=0$, the branching points $s=\pm i \kappa^{-1}$ disappear from the transfer function $\hat{H}_\mu$ derived in \eqref{transfert-delta} below, which then becomes
   $$
   {\hat{H}}_{0}(s):=
\frac{\tau^2_{0} s  + \ell } {{\tau^2_{0} \, s^2 + s \ell + 1}},
   $$
whose poles have a strictly negative real part, hence an exponential decay for $\delta$.
   \end{remark}

   \begin{proof}
Since the kernel ${\mathcal K}_1$ belongs to $L^1(\RR^+)$ (recall that the Bessel function $J_1(t)$ decays like $O(t^{-1/2})$), the proof of the first part of the proposition does not raise any particular problem. To prove the second part of the proposition, we need a careful analysis of the transfer function $\widehat{H}_\mu$ defined by the relation $\widehat{\delta}=\widehat{H}_\mu \delta_0$. 
Since $(\delta,\dot \delta) \in C\cap L^\infty(\mathbb{R}^+)$, the Laplace transform of $\delta$ and $\dot\delta$ are well defined on $\mathbb{C}_0$ an after remarking that
$$
 \kappa s + \frac{1}{ \sqrt{ 1+ \kappa^2s^2 } + \kappa s} = \sqrt{ 1+ \kappa^2s^2 },
$$
the Laplace transform of $\delta$ is, owing to  \eqref{ODESW2},
\begin{equation}\label{transfert-delta}
\hat{\delta} = {\hat{H}}_{\mu}(s) \delta_0
\quad \text{ where } \quad
{\hat{H}}_{\mu}(s):=
\frac{\tau^2_{\mu} s  + \ell \sqrt{ 1+ \kappa^2s^2 }}{\tau^2_{\mu} \, s^2 + s \ell \sqrt{ 1+ \kappa^2s^2 }+1 }.
\end{equation}

\begin{lemma}\label{poles-negative}
The transfer function ${\hat{H}}_{\mu}$ defined in  \eqref{transfert-delta}  is holomorphic on $\mathbb{C}_0$ and admits only two branching points at $\pm i \kappa^{-1}$. 
Moreover all the zeros of the denominator in \eqref{transfert-delta} have strictly negative part.   
 \end{lemma}
 
  \begin{proof}[Proof of the lemma]
 We denote by $P$ the holomorphic function on $\mathbb{C}_0$
 $$
P(s):= \tau^2_{\mu} \, s^2 + s \ell \sqrt{ 1+ \kappa^2s^2 }+1, 
 $$
 where the square root stands for the square root with positive real part.  
 Since $P$ has only two singularities which are the branching points $\pm i \kappa^{-1}$, we can extend it analytically  on all the complex plane except on the cuts $i (-\infty, - \kappa^{-1})$ and $i (\kappa^{-1}, \infty)$, and extend it continuously by a function $P^*$ on the imaginary axis by
  $$
 P^*(i \omega):=
 \begin{cases}
 1 - \tau^2_{\mu} \, \omega^2 + i \omega \ell \sqrt{ 1 - \kappa^2 \omega^2} & | \omega | \leq \kappa^{-1}, \\
 1 - \tau^2_{\mu} \, \omega^2 - \abs{\omega} \ell \sqrt{ \kappa^2 \omega^2 -1} & | \omega | > \kappa^{-1}.
 \end{cases}
 $$
We first show that $s$ cannot be purely imaginary, then  that if a zero is real,  it must be strictly negative, and finally that if a zero is not a real number,    it must satisfy $\Re(s) <0$.

 \noindent
 {\bf Step 1.} The zeros of $P$ cannot be purely imaginary. 
Indeed, if $\omega$ were we solution of $P^*(i\omega)=0$, then, from the expression of $P(i\omega)$ given above,  $\omega^2$ would be a real root of the second order polynomial
$$
(\tau_\mu^4-\ell^2\kappa^2)X^2+(\ell^2-2\tau^2_\mu)X+1.
$$
But the discriminant of this polynomial is $\Delta=\ell^4+4\ell^2(\kappa^2-\tau^2_\mu)<-\ell^2/3$ (since $\tau^2_\mu>\kappa^2+\ell^2/3$), which is negative, implying that the polynomial cannot have any real root.

  \noindent
 {\bf Step 2.} The zeros of $P$ cannot belong to $\RR^+$ from the simple observation that $P(\eta)>0$ for all $\eta\in \RR^+$.
 
  \noindent
 {\bf Step 3.} The zeros of $P$ cannot have positive real part. In order to prove this, we show here that $\Im(P(s))\neq 0$ for all $s=\eta + i \omega$ with $s \notin i \mathbb{R} \cup \mathbb{R}$ (the case $s\in \RR^+$ having been dealt with in Step 2).  The imaginary part of $P (s)$ is given by
  $$
 \Im [ P (s) ] = 2 \tau^2_{\mu} \, \eta \,  \omega 
  + \eta \, \ell \Im [ \sqrt{1+ \kappa^2 s^2} ]
  + \omega \, \ell \Re [ \sqrt{1+ \kappa^2 s^2} ].
  $$ 
  By definition of the square root, $\Re [ \sqrt{1+ \kappa^2 s^2} ] \geq 0$ and the sign of $\Im [ \sqrt{1+ \kappa^2 s^2} ]$ is the same as the sign of the product $\eta \, \omega$. The following table summarizes the sign of some quantities in different cases (where $+$ stands for strictly positive, $-$ stands for strictly negative and ind. signifies that the sign is indeterminate).
\begin{center}
\begin{tabular}{|c|c|c|c|c|c|}
  \hline
$ \eta$ & $\omega$ & $2 \tau^2 \, \eta \,  \omega$ & $\eta \, \Im [ \sqrt{1+ \kappa^2 s^2} ]$ & $\omega \Re [ \sqrt{1+ \kappa^2 s^2} $]  & $\Im [ P (s) ]$
\\
\hline
 - & - & + & - & -  & ind.
 \\
  - & + & - & + & +  & ind.
 \\
  + & - & - & - & -  & -
   \\
  + & + & + & + & +  & +
  \\
    \hline
\end{tabular}  
\end{center}
Therefore, if $\eta>0$ then  $\Im [ P (s) ]$ is either strictly positive or strictly negative, so that it does not vanish.

   \end{proof}
   
Using Lemma \ref{poles-negative}, we can extend continuously the transfer function ${\hat{H}}_{\mu}$ on the imaginary axis by
$$
{\hat{H}}^*_{\mu}(i \omega)= 
\begin{cases}
\dsp \frac{\tau_\mu^2 i \omega  + \ell \sqrt{ 1- \kappa^2 \omega^2 }}{- \tau^2_\mu \, \omega^2 + i \omega \ell \sqrt{ 1- \kappa^2 \omega^2 }+1 } \delta_0,
& | \omega | \leq \kappa^{-1},
\\
\dsp \frac{\tau^2_\mu i \omega  + i \ell {\rm sign}(\omega)\sqrt{  \kappa^2 \omega^2-1 }}{- \tau^2_\mu \, \omega^2 - \abs{\omega} \ell \sqrt{ \kappa^2 \omega^2 -1 }+1 } \delta_0,
& | \omega | > \kappa^{-1}.
\end{cases}
$$
Integrating $| {\hat{H}}^*_{\mu}(i \omega) |^2$ over $\mathbb{R}$ we get
$$
\displaystyle \int_{\mathbb{R}} | {\hat{H}}^*_{\mu}(i \omega) |^2 {\rm d} \omega
= \displaystyle \int_{-\kappa^{-1}}^{\kappa^{-1}} | {\hat{H}}^*_{\mu}(i \omega) |^2 {\rm d} \omega + \displaystyle \int_{ | \omega | > \kappa^{-1}} |{\hat{H}}^*_{\mu}(i \omega) |^2 {\rm d} \omega.
$$
The first integral is obviously finite (the denominator in \eqref{transfert-delta} does not vanish on $i\RR$). The second integral is also finite since $| {\hat{H}}^*_{\mu}(i \omega) |^2 { \underset{ |\omega | \to \infty}{\sim}} \omega^{-2}$.
Then  ${\hat{H}}_{\mu}$ belongs to the standard Hardy space $ \mathcal{H}^2( \mathbb{C}_0 )$ and by the Paley-Wiener theorem (see Theorem  \ref{Paley-Wiener} below), one has $\delta\in L^2(\RR^+)$.\\
The same reasoning can be applied to
$$
\widehat{\dot\delta}= \left( \frac{-1}{\tau^2_\mu \, s^2 + s \ell \sqrt{ 1+ \kappa^2s^2 }+1 } \right) \delta_0
$$
and
$$
\widehat{\ddot\delta}=\left( \frac{-s}{\tau^2_\mu \, s^2 + s \ell \sqrt{ 1+ \kappa^2s^2 }+1 } \right) \delta_0
$$
so that $\dot\delta$ and $\ddot\delta$ also belong to $L^2(\RR^+)$.\\
Let us now prove that $u(t):=t\delta(t)$ does {\it not} belong to $L^2(\RR^+)$. 
    Denoting by $U$ the Laplace transform of $u$, one has 
   $$
   U(s):=(-1) \left( \frac{d}{ds} \right) H_{\mu}(s) \delta_0 \quad \mbox{ on }\quad {\mathbb C}_0,
   $$
  and the following extension to  the imaginary axis holds,
   $$
   U^*(i \omega)= 
\begin{cases}
\dsp (-1)\left( \frac{d}{d\omega} \right) 
\frac{\tau^2_\mu i \omega  + \ell \sqrt{ 1- \kappa^2 \omega^2 }}{- \tau^2_\mu \, \omega^2 + i \omega \ell \sqrt{ 1- \kappa^2 \omega^2 }+1 } \delta_0,
& | \omega | \leq \kappa^{-1},
\\
\dsp (-1) \left( \frac{d}{d\omega} \right) 
\frac{\tau^2_\mu i \omega  + i \ell {\rm sign}(\omega)\sqrt{  \kappa^2 \omega^2-1 }}{- \tau^2_\mu \, \omega^2 - \abs{\omega} \ell \sqrt{ \kappa^2 \omega^2 -1 }+1 } \delta_0,
& | \omega | > \kappa^{-1}.
\end{cases}
$$
But $U$ is no longer bounded on the imaginary axis as it contains two non isolated singularities (of order $-1/2$ is the Puiseux series expansion) at $\pm i \kappa^{-1}$; the integral 
$$
\displaystyle \int_{-\kappa^{-1}}^{\kappa^{-1}} | U^*(i \omega) |^2 d \omega
$$
is not finite and thus $U \notin \mathcal{H}^2(\mathbb{C}_0)$.  By the Paley-Wiener theorem, this implies that $u \notin L^2(\mathbb{R}^+)$. 
The same reasoning can be applied to
$$
V(s)= (-1) \left( \frac{d}{ds} \right) \left( \frac{-1}{\tau^2_\mu \, s^2 + s \ell \sqrt{ 1+ \kappa^2s^2 }+1 } \right) \delta_0
$$
and
$$
W(s) =(-1)\left( \frac{d}{ds} \right) \left( \frac{-s}{\tau^2_\mu \, s^2 + s \ell \sqrt{ 1+ \kappa^2s^2 }+1 } \right) \delta_0,
$$
which are respectively the Laplace transforms of $t \dot\delta$ and $t \ddot\delta$. This completes the proof of the second point of the proposition.\\
%
Let us now prove the third point by contradiction. Assuming that there exists $C >0$ such that for $t$ large enough
$$
| \delta^{(k)}(t) | \leq C t^{-\frac{3}{2} - \alpha},
$$ 
one gets
$$
|t \delta^{(k)}(t) |^2 \leq C t^{-1 - 2 \alpha}
$$
which implies $t \delta^{(k)} \in L^2(\mathbb{R}^+)$, which contradicts the second point. 
   \end{proof}
   
In the non dispersive case, we showed in Corollary \ref{coroCumminsSW} that once the motion of the object is known, it is possible to find $q$ in the exterior domain by solving an initial boundary value problem for a Burgers-type scalar equation. This remains true in the present dispersive linear case, but the initial boundary value problem one has to solve is now nonlocal in time. Note that as in Remark \ref{remforced}, the corollary can easily be generalized to describe the waves created by an object in forced motion.
 \begin{corollary}\label{coroCumminslinear}
The return to equilibrium problem for the linear Boussinesq equations \eqref{linear}-\eqref{ODElinear} with initial condition \eqref{CIRE} can be equivalently formulated as a scalar nonlocal initial boundary value problem on $q$
\begin{equation}\label{linWave}
\begin{cases}
\dx q+ {\mathcal K}^0_\mu \ast \partial_t q  &=0 \qquad (t>0,\quad x>\ell),\\
 q_{\vert_{t=0}}&=0,\\
  q_{\vert_{x=\ell}}&=-\ell \dot\delta,
\end{cases}
\end{equation}
where $ {\mathcal K}^0_\mu$ is defined in \eqref{defK0K1} while $\zeta$ is given in terms of $q$ by a convolution in time
\begin{equation}\label{zetaqlin}
\zeta={\mathcal K}^0_\mu * q,
\end{equation}
with $\delta$ furnished by Proposition \ref{decay-polynom}.
\end{corollary}
\begin{remark}
The nonlocal initial boundary value problem \eqref{linWave} is not standard. The most convenient way to handle it is to see it as an evolution equation with respect to $x$ rather than $t$; it then becomes  a particular case of the nonlocal initial boundary value problems considered in Section \ref{sectNLPB}. It is in particular a consequence of Theorem \ref{theo-pde-laplace} below that \eqref{linWave} admits a unique solution $q\in  C(\RR^+_x;H^{1}(\RR^+_x))\cap C^1(\RR^+_x;L^{2}(\RR^+_x))$. Moreover, Proposition \ref{propRL} and Corollary \ref{cororeg} imply that the solution if actually of class $C^2(\RR^+\times \RR^+)$ and infinitely regular with respect to time, showing that the dispersive terms induce a smoothing effect. Indeed, when $\kappa=0$,  the first equation in \eqref{linWave} becomes 
\begin{equation}\label{linWavetransport}
  \partial_t q + \partial_x q =0
\end{equation}
and the solution to the initial boundary value problem, explicitly given by
$$
q(x,t)= \begin{cases} 
- \ell \dot\delta(t-(x-\ell)) & \text{ for } t-(x-\ell) \geq 0
\\
0 & \text{ for } t-(x-\ell) < 0,
\end{cases}
$$
does not belong to $ C^1(\mathcal{E}_+ \times \mathbb{R}^+ )$  because $\ddot{\delta}(0)= -\frac{1}{ \tau^2_\mu}\neq 0$. 
\end{remark}

 \section{The initial boundary value  problem for a class of nonlocal transport equations}\label{sectNLPB}

As shown in the previous section, the analysis of the return to equilibrium problem in the linear dispersive case leads to a nonlocal generalization of the transport equation. The analysis of the initial value problem for such equations is not standard and we address it in this section. Since this subject is of interest in its own, we work here with more standard notations. More precisely, we consider an evolution with respect to the time variable and a nonlocal term with respect to the space variable (this is the reverse in \S \ref{sectlinear}). The domain of consideration is the quadrant $\{x\geq 0,t\geq 0\}$. The typical initial boundary value we shall consider is therefore of the form,
$$
\begin{cases}
\dt u +\cK*_x\dx u &=f,\\
u_{\vert_{x=0}}&=\uu,\\
u_{\vert_{t=0}}&=u^{\rm in}.
\end{cases}
$$
for some convolution kernel $\cK$ to be made precise later.

After presenting some technical material in \S \ref{sectFS} for the functional setting and the Laplace transform, we remind in \S \ref{sectbehav} some very classical facts on the initial and/or boundary value problems for the standard transport equation, making a distinction between the case of a positive and a negative velocity. The nonlocal generalizations of these transport problems, in which $\dx u$ is replaced by a nonlocal term $\cK*_x\dx u$, are addressed in \S \ref{sectNLT}; in particular similarities and differences (such as the presence of an additional compatibility condition and a smoothing effect) with their local counterparts are commented.

\medbreak

\noindent
{\bf NB.} To avoid confusions with the computations performed in \S \ref{sectlinear} where the Laplace transform $\widehat{u}$ was taken with respect to time (with dual variable $s$) we denote throughout this section by $\widetilde{u}$ the Laplace transform with respect to $x$ (with dual variable $p=\alpha+i\xi$). 

\subsection{Functional setting and a brief reminder on the Laplace transform}

We gather here some definition of functional spaces that play an important role in the analysis of initial boundary value problems, as well as some classical facts on the Laplace transform.

\subsubsection{Functional setting}\label{sectFS}

In the study of initial boundary value problems for hyperbolic systems of equations, the space ${\mathbb X}^n$ plays a central role; it is defined for all $n\in {\mathbb N}$ as
$$
{\mathbb X}^n=\bigcap_{j=0}^n C^j(\RR^+_t;H^{n-j}(\RR^+_x));
$$
in particular, for all $u\in {\mathbb X}^n$, one can define for all $t\geq 0$ the quantity
$$
||| u(t,\cdot)|||_n:= \sup_{j+k\leq n} \abs{\dt^j\dx^k u(t,\cdot)}_{L^2(\RR^+)}.
$$
Let us also define ${\mathbb Y}^n$ as
$$
{\mathbb Y}^n=\bigcap_{j=0}^n W^{j,1}_{\rm loc}(\RR^+_t;H^{n-j}(\RR^+_x)).
$$

When working with nonlocal transport equations, it is convenient to introduce weighted versions of these spaces. For any $a\in \RR$, and $k\in {\mathbb N}$, we introduce therefore
\begin{align*}
L^2_a(\RR^+)&:=\{u\in L^2_{\rm loc}(\RR^+), \, \vert u\vert_{L^2_a}:=\Big(\int_{\RR^+}e^{-2ax }\abs{u(x)}^2 {\rm d}x\Big)^{1/2}<\infty \},\\
H^k_a(\RR^+)&:=\{u\in L^2_a(\RR^+), \, \vert u\vert_{H^k_a}:= \sum_{l=0}^k \vert \dx^l u\vert_{L^2_a}<\infty\},
\end{align*}
and denote by ${\mathbb X}^n_a$ and ${\mathbb Y}^n_a$ the weighted version of the spaces  ${\mathbb X}^n$ and ${\mathbb Y}^n$ obtained by replacing all $L^2(\RR^+_x)$ based spaces by their $L^2_a(\RR^+_x)$ analogue: we also write
$$
||| u(t,\cdot)|||_{a,n}:= \sup_{j+k\leq n} \abs{\dt^j\dx^k u(t,\cdot)}_{L_a^2(\RR^+)}.
$$

\subsubsection{Some results on the Laplace transform}

For all $u\in L^1_{\rm loc}(\RR^+)$, the Laplace transform is defined by
$$
\widetilde{u}(p)=\int_0^\infty e^{-px}u(x){\rm d}x,
$$
for all $p=\alpha+i\xi \in {\mathbb C}$ such that this integral converges absolutely. Using for all $a\in \RR$ the notation
$$
{\mathbb C}_a=\{ p\in {\mathbb C}, \Re p>a\},
$$
we can define the Hardy space
  $$
        \mathcal{H}^2(\mathbb{C}_a) := \Big\{ U \mbox{\rm{ holomorphic on }}\mathbb{C}_a \, ; \, || U ||^2_{\mathcal{H}^2( \mathbb{C}_{a})}:= \sup_{\alpha > a}   \int_{\mathbb{R}} | U (\alpha + i \xi)|^{2} d \xi  < \infty \Big\}.
        $$
Every function $U\in  \mathcal{H}^2(\mathbb{C}_a) $ admits a boundary trace denoted $U^*$ on $a+i\RR$, that belongs to $L^2(a+i\RR)$, and $\mathcal{H}^2(\mathbb{C}_a) $ is a Hilbert space for the scalar product
$$
\langle F,G\rangle_{ \mathcal{H}^2(\mathbb{C}_a)}=\frac{1}{2\pi}\int_{-\infty}^\infty F^*(a+i\xi)\overline{G^*(a+i\xi)} {\rm d}\xi.
$$
Recalling that the weighted space $L^2_a(\RR^+)$ is defined in the previous section, we can state the well-known Paley-Wiener theorem.
\begin{theorem}\label{Paley-Wiener} 
        Let $a\in \RR$. The Laplace-transform 
        $$
        	\mathcal{L}: \begin{array}{lcl}
        L_a^2(\RR^+) \to \mathcal{H}^2(\mathbb{C}_a)\\
        u\mapsto \widetilde{u}
        \end{array}
        $$ 
        is an isometry between Hilbert spaces. 
        \end{theorem}

Recalling that
$$
\widetilde{\frac{du}{dx} }(p)=p\widetilde{u}(p)-u(0)
$$
(whenever these quantities make sense), we also have the following characterization of the weighted Sobolev spaces $H^k_a(\RR^+)$.
\begin{proposition}\label{lemcarSob}
Let $k\in {\mathbb N}$ and $(\uu_0,\dots,\uu_{k-1})\in \RR^k$. The following assertions are equivalent,
\begin{itemize}
\item[{\bf i.}] One has $u\in H^k_a(\RR^+)$ and for all $0\leq j \leq k-1$, $\lim_{x\to 0^+} \dx^j u(x)=\uu_j$.
\item[{\bf ii.}] For all $0\leq j \leq k$, the mapping $p\mapsto p^j\widetilde{u}(p)-\sum_{i=0}^{j-1}p^{j-1-i}\uu_i$ belongs to ${\mathcal H}^2({\mathbb C}_a)$ (with the sum taken to be zero if $j=0$).
\end{itemize}
Moreover, for all  $0\leq j \leq k$, one has $\widetilde{\dx^j u}=p^j\widetilde{u}(p)-\sum_{i=0}^{j-1}p^{j-1-i}\uu_i$.
\end{proposition}

\subsection{Reminder on the standard transport equation}\label{sectbehav}

Let us start with some considerations on the standard initial boundary value problem for the transport equations
$\dt u +\dx u=f$ (referred to as right-going case) and $\dt u -\dx u=f$  (left-going case).
\subsubsection{The right-going case}
We consider here the following initial boundary value problem
\begin{equation}\label{IBVPtransp}
\begin{cases}
\dt u +\dx u &=f,\\
u_{\vert_{x=0}}&=\uu,\\
u_{\vert_{t=0}}&=u^{\rm in},
\end{cases}
\end{equation}
with $f\in {\mathbb Y}^1$, $u^{\rm in}\in H^1(\RR^+)$ and $\uu\in H^1_{\rm loc}(\RR^+)$. In order for \eqref{IBVPtransp} to admit a solution $u\in {\mathbb X}^1=C(\RR^+_t;H^1(\RR^+_x)) \cap C^1(\RR^+_t;L^2(\RR^+_x))$, and therefore continuous on $[0,\infty)\times [0,\infty)$, it is necessary that 
$$
\uu(t=0)=u^{\rm in}(x=0).
$$
This {\it compatibility condition} is actually sufficient to ensure the existence and uniqueness of such a solution. Even if the data are more regular, i.e. if $f\in {\mathbb Y}^n$, $u^{\rm in}\in H^n(\RR^+)$ and $\uu\in H^n_{\rm loc}(\RR^+)$ for some $n>1$, one cannot expect in general the solution to be in ${\mathbb X}^n$. It is a general feature of first order hyperbolic systems that such a regularity is achieved if and only if $n$ algebraic compatibility conditions are satisfied  (see for instance \cite{BenzoniSerre,Metivier2001,Metivier2012,IguchiLannes}). 
 Of course, the situation is the same if we choose to work in the weighted space ${\mathbb X}^n_a$ since the presence of the weight changes the integrability properties at infinity, but not local regularity.
 In the present case, this can easily be checked on the following explicit representation of the solution
\begin{equation}\label{soltransp}
u(t,x)=u^{\rm in}(x-t)+\uu(t-x)+\int_0^t f(t',x-t+t'){\rm d}t',
\end{equation}
where $u^{\rm in}$, $\uu$ and $f(t,\cdot)$ are extended by zero in order to be considered as functions defined on the full line $\RR$ instead of $\RR^+$.

\subsubsection{The left-going case}

It is well-known that an initial boundary value problem similar to \eqref{IBVPtransp} is ill-posed for the left-going transport equation $\dt u-\dx u=f$. Indeed, the initial value problem (without boundary condition)
\begin{equation}\label{IVPtransp}
\begin{cases}
\dt u -\dx u &=f,\\
u_{\vert_{t=0}}&=u^{\rm in},
\end{cases}
\end{equation}
is well posed, and the solution can be explicitly written as
\begin{equation}\label{left}
u(t,x)=u^{\rm in}(x+t)+\int_0^t f(t',x+t-t'){\rm d}t';
\end{equation}
in particular, the boundary value $\uu$ is given in terms of $u^{\rm in}$ and $f$ through the relation
$$
\uu(t)=u^{\rm in}(t)+\int_0^t f(t',t-t'){\rm d}t'
$$
and therefore cannot be freely prescribed. Note that using this relation in \eqref{left}, one can express the solution in terms of the boundary data instead of the initial data, namely,
\begin{equation}\label{left_bis}
u(t,x)=\uu(x+t)-\int_0^{x} f(x+t-x',x'){\rm d}x'.
\end{equation}
This proves in particular that the following boundary value problem (without initial condition) 
\begin{equation}\label{BVPtransp}
\begin{cases}
\dt u -\dx u &=f,\\
u_{\vert_{x=0}}&=\uu,
\end{cases}
\end{equation}
is also well-posed for the left-going transport equation. \\
We note finally that  for the initial value problem \eqref{IVPtransp} as well as for the boundary value problem \eqref{BVPtransp} (which are essentially the same by switching the variables $t$ and $x$) the solution $u$ belongs to $ {\mathbb X}^1$ if the data are smooth enough without having to impose any compatibility condition, contrary to what we saw for the right-going case.

\subsection{The nonlocal transport equation}\label{sectNLT}

The aim of this section is to investigate the behavior of nonlocal perturbations of the right-going and left-going transport equations respectively given by
\begin{equation}\label{NLT}
\dt u +\cK^0_\mu*_x\dx u =f\quad \mbox{ and }\quad \dt u -\cK^0_\mu*_x\dx u =f,
\end{equation}
where $*_x$ stands for the causal convolution with respect to the space variable, 
$$
\forall x\in \RR^+, \qquad f*_xg(x)=\int_0^x f(x-x')g(x'){\rm d} x'
$$
and with the Bessel kernel $\cK^0_\mu$ as in \eqref{defK0K1}; in particular, we recall that
$$
\widetilde{\cK^0_\mu}(p)=\frac{1}{\sqrt{1+\kappa^2p^2}} \qquad (\kappa^2=\mu/3).
$$
\begin{remark}\label{rempertt}
 Though we consider here the Bessel kernel $\cK^0_\mu$, the results of this section can easily be adapted to other kernels. 
 \end{remark}

 An important feature of the family $(\cK^0_\mu)_{\mu>0}$ is that it formally converges to the Dirac mass at $x=0$ as $\mu\to 0$, so that the nonlocal transport equations \eqref{NLT} formally converge to the standard right-going and left-going transport equations respectively, namely,
 $$
 \dt u +\dx u =f\quad \mbox{ and }\quad \dt u -\dx u =f.
$$
A natural question is therefore to ask whether the nonlocal initial and/or boundary value problems have a similar behavior to the behavior of their local counterpart described in \S \ref{sectbehav}.

\subsubsection{The right-going case}

We want to address in this section the same kind of initial boundary value problem as \eqref{IBVPtransp}, but where the space derivative is now replaced by a nonlocal term, namely, we consider
\begin{equation}\label{IBVPtranspNL}
\begin{cases}
\dt u +\cK^0_\mu*_x\dx u &=f,\\
u_{\vert_{x=0}}&=\uu,\\
u_{\vert_{t=0}}&=u^{\rm in}.
\end{cases}
\end{equation}

As for \eqref{IBVPtransp}, if there exists a solution  $u\in {\mathbb X}^1$ (or more generally in the  weighted version ${\mathbb X}^1_a$ with $a\geq 0$) to \eqref{IBVPtranspNL}, then it is continuous at $x=t=0$ and the data must therefore satisfy the same compatibility condition 
\begin{equation}\label{CCT1}
\uu(t=0)=u^{\rm in}(x=0)
\end{equation}
 as for the standard transport equation.\\
 There is however a new compatibility condition that arises here. Indeed, since $\cK^0_\mu\in L^1_{\rm loc}(\RR^+)$, the trace of $\cK^0_\mu*_x \dx u$ at $x=0$ is well defined if $\dx u\in C(\RR^+_t;L^2_{\rm loc}(\RR^+_x))$, and it must be equal to zero by definition of the convolution. Taking the trace of the first equation in \eqref{IBVPtranspNL}, one therefore finds the following additional compatibility condition for the existence of solutions with the aforementioned regularity,
\begin{equation}\label{CCT2}
\forall t\in \RR^+,\qquad \dt \uu(t)= f(t,0).
\end{equation}
\begin{remark}
The similar procedure applied to the standard transport problem \eqref{IBVPtranspNL} yields the relation
$$
\dx u (t,0)=-\dt \uu(t)+f(t,0),
$$
which is not a compatibility condition but an information on the behavior of the trace of $\dx u$ at the boundary.
\end{remark}
If these two compatibility conditions are satisfied, the theorem below shows the well-posedness of the nonlocal initial boundary value problem \eqref{IBVPtranspNL}. 
We recall that the functional spaces have been defined in \S \ref{sectFS}; note also that we have to work in weighted spaces here  in order to compensate the slow decay of $\cK^0_\mu$ at infinity (which is of order $O(\abs{x}^{-1/2}$) and that more information on the regularity of the solution is given in Corollary \ref{cororeg} below.
 \begin{theorem}\label{theo-pde-laplace}
      Let $a>0$ and $f\in {\mathbb Y}^1_a$, $u^{\rm in}\in H_a^1(\RR_x^+)$ and $\uu \in W^{1,1}_{\rm loc}(\RR^+_t)$. Assume moreover that the compatibility conditions \eqref{CCT1} and \eqref{CCT2} hold.
      Then there exists a unique solution $u\in {\mathbb X}^1_a$ to the nonlocal initial boundary value problem \eqref{IBVPtranspNL}, 
 and there exists $c_a>0$ such that, for all $t\in \RR^+$,
$$
\vert u(t,\cdot ) \vert_{H_a^1(\RR^+_x)}\leq  e^{-c_a t}\vert u^{\rm in} \vert_{H_a^1(\RR^+_x)}+\int_0^t  e^{-c_a(t-t')}\big[ \abs{f(t',\cdot)}_{H_a^1(\RR^+_x)}
+\abs{{\mathcal K^0_\mu}}_{L_a^2} \abs{\uu(t')}\big]
{\rm d}t'.
 $$     
 If moreover $\uu=0$ then the result still holds with $a=c_a=0$.
        \end{theorem}
\begin{proof} {For the sake of clarity, we simply write $\cK$ instead of $\cK^0_\mu$.}
Taking the Laplace transform of \eqref{IBVPtranspNL} with respect to space, one gets that 
$$
\dt \widetilde{u}+\widetilde{{\mathcal K}}(p)(p\widetilde{u}-\uu)=\widetilde{f} \quad \mbox{ on }\quad \RR^+.
$$
Solving this ODE with initial condition $\widetilde{u}_{\vert_{t=0}}=\widetilde{u^{\rm in}}$, one gets the following expression for $\widetilde{u}$,
for all $p\in {\mathbb C}_a$ and $t\in \RR^+$,
\begin{align}
\nonumber
\widetilde{u}(t,p)&=e^{-p\widetilde{\cK}(p) t} \widetilde{u^{\rm in}}(p)+\int_0^t e^{-p\widetilde{\cK}(p) (t-t')}\widetilde{f}{\rm d}t'
+\int_0^t e^{-p\widetilde{\cK}(p) (t-t')} \widetilde{\cK}(p)\uu(t'){\rm d}t' \\
\label{exptu}
&=:  \widetilde{u_1}+ \widetilde{u_2}+ \widetilde{u_3}.
\end{align}
Since the  Paley-Wiener Theorem \ref{Paley-Wiener} states that the Laplace transform is an  isometry between $L_a^2(\RR^+)$ and ${\mathcal H}^2({\mathbb C}_a)$, the following lemma shows that both $u_1$ and $u_2$ belong to $C(\RR^+_t;L^2_a(\RR^+_x))$ if $u^{\rm in}\in L^2_a(\RR^x)$ and $f\in L^1_{\rm loc}(\RR^+_t;L^2_a(\RR^+_x))$.
\begin{lemma}\label{lemHardy}
Let $a\geq 0$. For all $U\in {\mathcal H}^2({\mathbb C}_a)$, the mapping 
$$
\begin{array}{lcl} 
\RR^+ &\to&  {\mathcal H}^2({\mathbb C}_a)\\
t & \mapsto &   \big(p \mapsto e^{-p\widetilde{\cK}(p) t} U(p) \big)
\end{array}
$$
is well defined and continuous, and for all $t\in \RR^+$, $\Vert  e^{-p\widetilde{\cK}(p) t} U \Vert_{{\mathcal H}^2({\mathbb C}_a)}\leq  \Vert  U \Vert_{{\mathcal H}^2({\mathbb C}_a)} $.\\
If moreover $a>0$ then there exists $c_a>0$ such that for all $t\in \RR^+$,
$$
\Vert  e^{-p\widetilde{\cK}(p) t} U \Vert_{{\mathcal H}^2({\mathbb C}_a)}\leq e^{-c_a t} \Vert  U \Vert_{{\mathcal H}^2({\mathbb C}_a)} .
$$
\end{lemma}
\begin{proof}[Proof of the lemma]
Except for the last assertion, we consider only the case $a=0$ since the case $a>0$ can easily be deduced from it. From the definition of $ {\mathcal H}^2({\mathbb C}_0)$ and Lebesgue's dominated convergence theorem, it is sufficient to prove that $e^{-p\widetilde{\cK}(p) t} $ is holomorphic and bounded on ${\mathbb C}_0$. The fact that it is holomorphic directly stems from  the explicit expression $\widetilde{\cK}(p)=(1+\kappa^2 p^2)^{-1/2}$. For the boundedness, this is a consequence of the fact that $\Re (p \widetilde{\cK}(p) ) \geq 0$ on  ${\mathbb C}_0$, as we now prove. For all $p= \alpha + i \xi \in \mathbb{C}_0$, one computes
\begin{equation}\label{exprsK}
\Re( p \hat{{\mathcal K}}(p)) =  \frac{ \alpha \, \Re ( {\sqrt{1+ \kappa^2p^2}}) + \xi \,  \Im ( {\sqrt{1+ \kappa^2p^2}})  }{| \sqrt{1+ \kappa^2p^2} |^2}.
\end{equation}
Since $\Re (\sqrt{1+ \kappa^2p^2})$ is positive (by definition of the square root) and the sign of $\Im ( {\sqrt{1+ \kappa^2p^2}})$ is the same as the sign  of the product $ \alpha \, \xi$, one gets the result. \\
Since we have proved that $\Re (p \widetilde{\cK}(p) ) \geq 0$ on  ${\mathbb C}_0$, the last assertion follows if we can prove that $\Re (p \widetilde{\cK}(p) )$ does not vanish on ${\mathbb C}_a$ if $a>0$. Since both terms in the numerator in \eqref{exprsK} are positive, both must vanish if $\Re( p \widetilde{{\mathcal K}}(s)) $ vanishes. Since $\alpha>0$ on ${\mathbb C}_0$, this implies that there should be $p=\alpha+i\xi \in {\mathbb C}_a$ such that $\Re(\sqrt{1+\kappa^2 p^2})=0$ and $ \xi \,  \Im ( {\sqrt{1+ \kappa^2p^2}})=0$, which is obviously not possible.
\end{proof}
Remarking that for any $a>0$, one has $\widehat{\mathcal K}\in {\mathcal H}^2({\mathbb C}_a)$, it is also a direct consequence of the lemma that there is $c_a>0$ such that
\begin{align*}
\Vert{\widetilde{u}_3}\Vert_{\Hardya}&\leq
 \Vert{\widetilde{\cK}}\Vert_{\Hardya} \int_0^t e^{-c_a (t-t')} \vert \uu(t')\vert {\rm d}t'.
 \end{align*}
Together with the results already proved on $\widetilde{u}_1$ and $\widetilde{u_2}$, we deduce (see the Paley-Wiener Theorem \ref{Paley-Wiener} below) that
$$
\vert u(t,\cdot ) \vert_{L_a^2(\RR^+_x)}\leq  e^{-c_a t}\vert u^{\rm in} \vert_{L_a^2(\RR^+_x)}+\int_0^t  e^{-c_a(t-t')}\big[ \abs{f(t',\cdot)}_{L_a^2(\RR^+_x)}
+\abs{{\mathcal K}}_{L_a^2} \abs{\uu(t')}\big]
{\rm d}t'.
 $$     
In order to conclude the proof of the theorem, we still need to control $\dx u$ and $\dt u$.
\begin{itemize}
\item Control of $\dx u$. We want to show that $\dx u\in C(\RR^+_t;L^2_a(\RR^+_x))$, or equivalently that $\widetilde{\dx u} \in C(\RR^+_t;{\mathcal H}^2({\mathbb C}_a))$. Since $\widetilde{\dx u}=p\widetilde{u}-\uu$, we consider
$$
p\widetilde{u}(t,p)=e^{-p\widetilde{\cK}(p) t} p\widetilde{u^{\rm in}}(p)+\int_0^t e^{-p\widetilde{\cK}(p) (t-t')}p\widetilde{f}{\rm d}t'
+\int_0^t e^{-p\widetilde{\cK}(p) (t-t')} p\widetilde{\cK}(p)\uu(t'){\rm d}t'.
$$
Writing $p\widetilde{u^{\rm in}}=\widetilde{\dx u^{\rm in}}+u^{\rm in}(0)$, $p\widetilde{f}(t,p)=\widetilde{\dx f}(t,p)+f(t,0)$, we can remark that
\begin{align*}
\int_0^t e^{-p\widetilde{\cK}(p) (t-t')} p\widetilde{\cK}(p)\uu(t'){\rm d}t'&=\int_0^t \partial_{t'}\big( e^{-p\widetilde{\cK}(p) (t-t')} \big)\uu(t'){\rm d}t' \\
&=\uu(t)-e^{-p\widetilde{\cK}(p)t}\uu(0)-\int_0^t  e^{-p\widetilde{\cK}(p) (t-t')} \dt\uu(t'){\rm d}t' ,
\end{align*}
from which we deduce that
\begin{align*}
\widetilde{\dx u}(t,p)=&
e^{-p\widetilde{\cK}(p) t} \widetilde{\dx u^{\rm in}}(p)+\int_0^t e^{-p\widetilde{\cK}(p) (t-t')}\widetilde{\dx f}{\rm d}t'\\
&+e^{-p\widetilde{\cK}(p) t}\big( u^{\rm in}(0)-\uu(0) \big)
+\int_0^t e^{-p\widetilde{\cK}(p) (t-t')} \big( f(t',0)- \dt \uu(t') \big) {\rm d}t'.
\end{align*}
While the first two components of the right-hand side belong to $C(\RR^+_t;{\mathcal H}^2({\mathbb C}_a))$ by Lemma \ref{lemHardy}, the last two ones do not, unless the compatibilty conditions given in the statement of the Theorem are satisfied, in which case these two components cancel and the result follows together with the upper bound
$$
\vert \dx u (t,\cdot)\vert_{L^2_a} \leq    \vert \dx u^{\rm in}\vert_{L^2_a}+\int_0^t e^{-c_a (t-t')}\vert \dx f(t',\cdot)\vert_{L^2_a}{\rm d}t'.
$$
\item Control of $\dt u$. Using the equations, one has
\begin{align*}
\abs{\dt u}_{L^2_a}&\leq \abs{\cK*_x \dx u}_{L^2_a}+\abs{f}_{L^2_a}\\
&\leq  \abs{\cK}_{L^1_a} \abs{\dx u}_{L^2_a}+\abs{f}_{L^2_a},
\end{align*}
with $L^1_a=L^1(\RR^+,e^{-ax}{\rm d}x)$, showing as needed that $\dt u \in C(\RR^+_t;L^2_a(\RR^+_x))$.
\end{itemize}
The theorem follows easily.
\end{proof}
\begin{remark}\label{remexplCC}
As explained above in Remark \ref{rempertt}, the initial boundary value problem \eqref{IBVPtranspNL} can be seen as a nonlocal perturbation of the standard transport problem \eqref{IBVPtransp} toward which it formally converges when $\mu\to 0$. There seems however to be some discrepancy because {\it two} compatibility conditions, namely, \eqref{CCT1} and \eqref{CCT2}, are needed to ensure the existence of solutions $u\in {\mathbb X}^1_a$ to \eqref{IBVPtranspNL}, while the sole compatibility condition \eqref{CCT1} is sufficient to get a similar result for the standard transport problem \eqref{IBVPtransp}. One should explain why the second compatibility condition \eqref{CCT2} disappears in the formal limit $\mu=0$.\\
The reason is that \eqref{CCT2} is here to ensure continuity of the solution at the boundary $x=0$. Indeed, by the initial value theorem, we know that $\lim_{x\to 0^+} u(t,x)=\lim_{p\in {\mathbb C}_a, \abs{p}\to \infty} p \widetilde{u}(t,p)$, and we therefore get from the Laplace representation formula \eqref{exptu} that
$$
\lim_{x\to 0^+}u(t,x)=e^{-\frac{t}{\kappa}}u^{\rm in}(0)+\int_0^t e^{-\frac{t-t'}{\kappa}}f(t',0){\rm d}t'+\int_0^t e^{-\frac{t-t'}{\kappa}}\frac{1}{\kappa}\uu(t'){\rm d}t'
$$
where we  used the fact that $\lim_{p\in {\mathbb C}_a, \abs{p}\to \infty} p\widetilde{\cK}(p)=\kappa^{-1}$; after an integration by parts, the right-hand side can be written
$$
 \uu(t)+e^{-\frac{t}{\kappa}}\big(u^{\rm in}(0)-\uu(0)\big)+\int_0^t e^{-\frac{t-t'}{\kappa}}\big(f(t',0)-\dt\uu(t')\big) {\rm d}t',
$$
so that, if the first compatibility condition \eqref{CCT1} is satisfied, one has
$$
\lim_{x\to 0^+}u(t,x)-\uu(t)=\int_0^t e^{-\frac{t-t'}{\kappa}}\big(f(t',0)-\dt\uu(t')\big) {\rm d}t',
$$
which is nonzero if the second compatibility condition is not satisfied, hence a lack of continuity at $x=0$ (there would therefore be a Dirac mass at $x=0$ is the expression for $\dx u(t,\cdot)$ that would therefore not be in $L^2_a(\RR^+_x)$ as seen in the proof). However, one readily observes that
$$
\lim_{\mu \to 0}\int_0^t e^{-\frac{t-t'}{\kappa}}\big(f(t',0)-\dt\uu(t')\big) {\rm d}t'=0\qquad (\kappa^2=\mu/3),
$$
so that this discontinuity shrinks to zero in the limit $\mu\to 0$, explaining why the second compatibility condition is no longer necessary in the endpoint case $\mu=0$.
\end{remark}

Before going further, we recall that there are two possibilities to define fractional derivatives of order $\alpha\in (0,1)$ on $\RR^+$ using the convolution kernel ${\mathfrak K}_\alpha(x)=x^{-\alpha}/{\Gamma(1-\alpha)}$ with $\alpha \in (0,1)$ and $\Gamma$ the Euler Gamma function, namely, the Riemann-Liouville and Caputo derivatives, defined respectively as
$$
D^\alpha_{\rm RL} u= \dx \big( {\mathfrak K}_\alpha *_x u)
\quad\mbox{ and }\quad
D^\alpha_{\rm C} u=  {\mathfrak K}_\alpha *_x \dx u.
$$
In the nonlocal initial boundary value problem \eqref{IBVPtranspNL}, the space derivative $\dx u$ in the standard transport equation has been replaced by the nonlocal term ${\mathcal K}^0_\mu*\dx u$ which can be considered as a generalized derivative {\it of Caputo type}, with  the kernel ${\mathfrak K}_\alpha$ replaced by the Bessel kernel $\cK^0_\mu$. It is noteworthy that working with the Riemann-Liouville version of this operator, namely $\dx\big( \cK^0_\mu*_x u)$, the situation is drastically different. Indeed, as shown in the following proposition, it is not possible to impose a boundary data anymore since the knowledge of the initial data suffices to fully determine the solution; in other words, the initial value problem
\begin{equation}\label{IBVPtranspNLRL}
\begin{cases}
\dt u +\dx\big( \cK^0_\mu*_x  u) &=f,\\
u_{\vert_{t=0}}&=u^{\rm in},
\end{cases}
\end{equation}
is well posed on $\RR^+_t\times \RR^+_x$. In particular, the trace of the solution at the boundary $x=0$ is determined by $f$ and $u^{\rm in}$ and therefore cannot be imposed. We also show that if the data $u^{\rm in}$ and $f$ are smoother, then the solution is in ${\mathbb X}_a^2$, but generally not in ${\mathbb X}_a^3$ or higher in the absence of additional compatibility condition (but we show however that the regularity in time can be higher).
\begin{proposition}\label{propRL}
     Let  $a>0$, $n=1$ or $2$, and $f\in {\mathbb Y}^n_a$ and $u^{\rm in}\in H_a^n(\RR_x^+)$.       Then there exists a unique solution $u\in {\mathbb X}_a^n$ to the nonlocal initial boundary value problem \eqref{IBVPtranspNLRL}.
Moreover, one has $u(t,\cdot)_{\vert_{x=0}}=\uu(t)$ for all $t\in \RR^+$, with $\uu(t)$ given by
$$
\uu(t)=e^{-\frac{t}{\kappa}}u^{\rm in}(0)+\int_0^t e^{-\frac{t-t'}{\kappa}}f(t',0){\rm d}t'.
$$
If in addition $f\in C^q(\RR^+_t; H^n_a(\RR^+_x) )$ for some $q\in {\mathbb N}$ then one also has $u\in C^{q+1}(\RR^+_t; H^n_a(\RR^+_x) )$.
\end{proposition}
\begin{remark}
Comparing the representation of the solution given in \eqref{exptu2} below to the representation of the solution to the initial boundary value problem \eqref{IBVPtranspNL} given in \eqref{exptu}, one can check that they are both the same if $\uu=0$, which is not surprising since one can compute
$$
\dx(\cK^0_\mu*_x u)(t,x)=(\cK^0_\mu *_x \dx u)(t,x) +\cK^0_\mu(x)\uu(t),
$$
so that the Caputo and Riemann-Liouville nonlocal initial boundary value problem coincide when $\uu=0$.
\end{remark}
\begin{proof}
{As previously done, we simply write $\cK=\cK_\mu^0$}. Taking the Laplace transform of \eqref{IBVPtranspNLRL} one readily gets
\begin{equation}\label{exptu2}
\widetilde{u}(t,p)=e^{-p\widetilde{\cK}(p) t}\widetilde{u^{\rm in}}(p)+\int_0^t e^{-p\widetilde{\cK}(p) (t-t')}\widetilde{f}(t',p){\rm d}t';
\end{equation}
by the initial value theorem, one gets that $\lim_{x\to 0^+}u(t,x)=\uu(t)$, with $\uu$ as in the statement of the theorem. \\
For all $j$ and $l$, one deduces from the above formula for $\widetilde{u}$ that
$$
p^j \widetilde{\dt^l u}=(-p\widehat{\cK}(p))^l \big[e^{-p\widehat{\cK}(p)t}p^j \widetilde{u^{\rm in}}
+\int_0^t e^{-p\widetilde\cK(p)(t-t')}p^j\widetilde{f} \big]
+\sum_{m=1}^{l}(-p\widehat{\cK}(p))^{l-m}p^j\widetilde{\dt^{m-1} f}.
$$
Replacing in this expression
$$
p^j\widetilde{v}=\widetilde{\dx^j v}+\sum_{i=0}^{j-1}p^{j-1-i}(\dx^i v)_{\vert_{x=0}}
$$
for $v=u^{\rm in}$, $\widetilde{f}$, $\widetilde{\dt^{m-1} f}$, we obtain
$$
p^j \widetilde{\dt^l u}=\sum_{i=0}^{j-1}p^{j-1-i} U_{li}(p)+F_{lj}(t,p)
$$
(using the convention that the summation is zero if $j-1<0$) with
\begin{align*}
U_{li}(p):= &(-p\widetilde{\cK}(p))^l \big(\dx^i u^{\rm in}(0) +\int_0^t e^{-p\widetilde\cK(p)(t-t')}(\dx^i f)(t',0){\rm d}t' \big)\\
&+\sum_{m=1}^l (-p\widetilde{\cK}(p))^{l-m}(\dt^{m-1}\dx^i f)_{\vert_{x=0}}
\end{align*}
and
$$
F_{lj}(t,p):=(-p\widehat{\cK}(p))^l \big[\widetilde{\dx^ju^{\rm in}}
+\int_0^t e^{-p\widetilde\cK(p)(t-t')}\widetilde{\dx^jf}  \big]+\sum_{m=1}^{l}(-p\widehat{\cK}(p))^{l-m}\widetilde{\dt^{m-1}\dx^j f}.
$$
Remarking that $\lim_{\abs{p}\to \infty} p \widehat{\cK}(p)=\kappa^{-1}$, and introducing $\uu_{li}=\lim_{\abs{p}\to\infty}U_{li}(p)$, namely,
$$
\uu_{li}=(-\kappa)^{-l} \big(\dx^i u^{\rm in}(0) +\int_0^t e^{-\frac{t-t'}{\kappa}}(\dx^i f)(t',0){\rm d}t' \big)\\
+\sum_{m=1}^l (-\kappa)^{-l+m}(\dt^{m-1}\dx^i f)_{\vert_{x=0}}
$$
(of course, $\uu_{00}=\uu$), we can write
$$
p^j \widetilde{\dt^l u}-\sum_{i=0}^{j-1}p^{j-1-i} \uu_{li}(p)=\sum_{i=0}^{j-1}p^{j-1-i} \big(U_{li}(p)-\uu_{li}\big)+F_{lj}(t,p).
$$
From Proposition \ref{lemcarSob}, we can deduce that $\dx^j\dt^l u$ belongs to $C(\RR^+_t;L^2_a(\RR^+_x))$ if the right-hand side of the above equality is in $C(\RR^+_t;{\mathcal H}^2({\mathbb C}_a))$. This is obvious for $F_{lj}$ under the assumptions made in the statement of the proposition (see Lemma \ref{lemHardy}); for the summation, the problem reduces to determine whether the mapping $p\mapsto p^{j-1}\big(p\widehat{\cK}(p)-\kappa^{-1}\big)$ belongs to
${\mathcal H}^2({\mathbb C}_a)$ or not. This mapping being holomorphic on ${\mathbb C}_a$, we just need to check that it is square integrable on $a+i\RR$. Recalling that $\widehat{\cK}(p)=\frac{1}{\sqrt{1+\kappa^2 p^2}}$, and using the fact that for all $p\in {\mathbb C}_a$ one has $\sqrt{p^2}=p$, one has $p^{j-1}\big(p\widehat{\cK}(p)-\kappa^{-1}\big)\sim -\frac{1}{2\kappa^3}p^{j-3}$ at infinity; the mapping is therefore square integrable on $a+i\RR$ if and only if $j\leq 2$, hence the results.
\end{proof}

As a corollary, we can exhibit a smoothing effect for the nonlocal transport problem \eqref{IBVPtranspNL} that does not exist for the standard transport problem \eqref{IBVPtransp}. Indeed, as one can easily check on the explicit expression \eqref{soltransp}, even if the data $u^{\rm in}$, $\uu$ and $f$ are very smooth, the solution is not $C^1(\RR^+\times \RR^+)$ if the additional compatibility condition $\dt \uu (0)=-\dx u^{\rm in}(0)+f(0,0)$ is not imposed. There is a smoothing effect for the nonlocal problem in the sense that the solution constructed in Theorem \ref{theo-pde-laplace} actually belongs to ${\mathbb X}_a^3 \subset C^2(\RR^+\times \RR^+)$ without any additional compatibility condition if the data are smooth enough. Note that using the last statement of Proposition \ref{propRL},  the proof shows that additional regularity in time on $\dx f$ would yield additional regularity in time on $\dx u$.
  \begin{corollary}\label{cororeg}
   Under the assumptions of Theorem   \ref{theo-pde-laplace}, if moreover  $f\in {\mathbb Y}^n_a$, $u^{\rm in}\in H_a^n(\RR_x^+)$ and $\uu \in W^{n,1}_{\rm loc}(\RR^+_t)$ for $n=2$ or $3$, then the solution $u$ provided by the theorem belongs to ${\mathbb X}^n_a$. 
     \end{corollary}    
\begin{proof}
Taking the space derivative of the nonlocal transport equation in \eqref{IBVPtranspNL}, it is easy to see that $v=\dx u$ solves the initial boundary value problem
$$
\begin{cases}
\dt v+\dx \big(\cK*_x v)=\dx f,\\
v_{\vert_{t=0}}=\dx u^{\rm in}.
\end{cases}
$$
It follows therefore from Proposition \ref{propRL} that $\dx u\in {\mathbb X}^{n-1}$. We are therefore left to prove that $\dt^j u\in C(\RR_t^+;L^2_a(\RR^+_x))$ for $1\leq j\leq n$; this easily follows from the observation that
$$
\dt^j u=-\cK*_x \dt^{j-1}\dx u +\dt^{j-1}f
$$
and from the fact that $\cK\in L^1_a(\RR^+)$.
\end{proof}

\subsubsection{The left-going case}

As for the right-going case in the previous section, we want to consider a nonlocal perturbation of the standard transport problem in which the
space derivative $\dx$ is replaced by a nonlocal term $\cK^0_\mu * \dx$. As recalled in \S \ref{sectbehav}, for the standard left-going transport equation, one has to consider either the initial value problem or the boundary value problem. While both cases are symmetric in the case of the standard transport equation, this is no longer the case and, as we shall see, the boundary value problem leads simpler expressions. We therefore consider here its nonlocal analogue (see Remark \ref{remamm} below for the nonlocal analogue of the initial value problem),
\begin{equation}\label{IBVPtranspNL_left}
\begin{cases}
\dt u-\cK^0_\mu*_x\dx u &=f,\\
u_{\vert_{x=0}}&=\uu.
\end{cases}
\end{equation}
As for the boundary value problem \eqref{BVPtransp} for the standard left-going transport equation, there is no compatibility condition like \eqref{CCT1} since $u^{\rm in}$ is not prescribed. On the other hand, the analysis leading to the second compatibility condition \eqref{CCT2} remains valid, and it is still necessary to have 
\begin{equation}\label{CCT2left}
\forall t\in \RR^+,\qquad \dt \uu(t)= f(t,0)
\end{equation}
in order to expect a solution $u$ that belongs to ${\mathbb X}^1$. In the statement below, we use the notation
$$
H^1_a(\RR^+_t\times \RR^+_x):=H^1(\RR^+_t;L^2_a(\RR^+_x))\cap L^2(\RR^+_t; H^1_a(\RR^+_x))
$$
(note that the assumptions on the time dependence of $f$ and $\uu$ are chosen in order to ensure the convergence of the integral term over the range $(t,+\infty)$ and that they could easily be weakened).
 \begin{theorem}\label{theo-pde-laplace_left}
      Let $a>0$, $f\in H^1_a(\RR^+_t\times \RR^+_x)$,  and $\uu \in H^1(\RR^+_t)$. Assume moreover that the compatibility condition \eqref{CCT2left} holds.
      Then there exists a unique solution $u\in {\mathbb X}^1_a$ to the nonlocal boundary value problem \eqref{IBVPtranspNL_left}, 
 and there exists $c_a>0$ such that, for all $t\in \RR^+$,
$$
\vert u(t,\cdot ) \vert_{H_a^1(\RR^+_x)}\leq  \int_t^\infty  e^{c_a(t-t')}\big[ \abs{f(t',\cdot)}_{H_a^1(\RR^+_x)}
+\abs{{\mathcal K^0_\mu}}_{L_a^2} \abs{\uu(t')}\big]
{\rm d}t'.
 $$     
        \end{theorem}
\begin{proof}
Still denoting $\cK=\cK_\mu^0$ and following the same procedure as for the proof of Theorem \ref{theo-pde-laplace}, one readily finds that
$$
\widetilde{u}(t,p)=-\int_t^\infty e^{p\widetilde{\cK}(p)(t-t')}\big( \widetilde{f}(t',p)-\widetilde{\cK}(p)\uu(t') \big){\rm d}t';
$$
as for the right-going case, one can check that the compatibility condition \eqref{CCT2left} is necessary for the continuity of the solution at $x=0$. We omit the proof which is an easy adaptation of the proof of Theorem \ref{theo-pde-laplace}.
\end{proof}
\begin{remark}
As for the standard boundary transport problem \eqref{BVPtransp}, the initial data is determined in terms of the source term $f$ and the boundary data $\uu$ by evaluating the Laplace representation formula given in the proof at $t=0$, namely,
\begin{equation}\label{CICBleft}
\widetilde{u^{\rm in}}(p)=-\int_0^\infty e^{-p\widetilde{\cK_\mu^0}(p)t'}\big( \widetilde{f}(t',p)-\widetilde{\cK_\mu^0}(p)\uu(t') \big){\rm d}t'.
\end{equation}
In the limit case $\mu=0$ (and therefore $\widehat{\cK}_\mu(p)=1$), one can check that the representation formula of the proof is equivalent to \eqref{left_bis}; the additional compatibility condition \eqref{CCT2left} that is not necessary for \eqref{BVPtransp} also disappears at the limit along a mechanism similar to the one described in Remark \ref{remexplCC}.
\end{remark}
\begin{remark}\label{remamm}
For the standard left-going transport equation, the initial value problem \eqref{left} and the boundary value problem \eqref{left_bis} can be treated in a totally symmetric case by switching the variables $t$ and $x$. The presence of the nonlocal term breaks this symmetry, and the nonlocal initial value problem would be more delicate to deal with than the boundary value problem addressed above. In particular, one would need to find $\uu$ in terms of $f$ and $u^{\rm in}$ by solving the nonlocal equation \eqref{CICBleft}.
\end{remark}

\appendix

\section{Non dimensionalization of the equations}\label{appND}

We show here how to derive the dimensionless equations of motion used throughout this paper. To begin with, the Boussinesq-Abbott system describing the propagation of weakly nonlinear waves in a fluid of mean depth $h_0$ and with a pressure $P_{\rm atm}+\underline{P}$ exerted at the surface 
($P_{\rm atm}$ is a constant reference value for the atmospheric pressure)
is given  by 
\begin{equation}\label{AbbottD}
\begin{cases}
\dt \zeta + \dx q=0,\\
(1-\frac{h_0^2}{3}\dx^2 )\dt q +\dx \big( \frac{1}{h}q^2 \big)+gh \dx \zeta=-\frac{h}{\rho}\dx \underline{P}\qquad (h=h_0+\zeta).
\end{cases}
\end{equation}
\begin{remark}
Introducing the {\it hydrodynamic pressure} $\Pi$ as
\begin{equation}\label{defHP}
\Pi=\underline{P}+\rho g\zeta,
\end{equation}
and alternative formulation of \eqref{AbbottD} is 
\begin{equation}\label{AbbottDbis}
\begin{cases}
\dt \zeta + \dx q=0,\\
(1-\frac{h_0^2}{3}\dx^2 )\dt q +\dx \big( \frac{1}{h}q^2 \big)=-\frac{h}{\rho}\dx \Pi;
\end{cases}
\end{equation}
we shall sometimes use this alternative formulation under the floating object.
\end{remark}

Let us now consider the equations for the solid. We recall that we consider here a floating object with vertical lateral walls located  at $x=\pm \ell$ ($\ell>0$) and allowed to move only vertically (heave motion). There is therefore only one degree of freedom for the motion of the solid which can be fully deduced from the signed distance $\delta(t)$ between the center of mass $G=\big(x_G,z_{G}(t)\big)$ and its equilibrium position $G_{\rm eq}=(x_G,z_{G,{\rm eq}})$, namely, $\delta=z_{G}(t)-z_{G, {\rm eq}}$.\\
Let us also assume that the water depth below the object is given at equilibrium by a nonnegative single valued function $x\mapsto h_{\rm eq}(x)$;  the part of the bottom of the object in contact with the water (the wetted surface) is therefore given at all time $t$ by the graph of the function ${\zeta}_{\rm w}$ defined as
\begin{equation}\label{eqzwD}
\zeta_{\rm w}(t,x)=\delta(t)+h_{\rm eq}(x)-h_0.
\end{equation}
Newton's equation for a body of mass $m$ that only moves vertically and subject to gravity and hydrodynamic forces is given by
\begin{equation}\label{Newton-z_GD}
m  \ddot{\delta} + mg = \displaystyle \int_{-\ell}^\ell \Pint(t,x){\rm d}x,
\end{equation}
where $\Pint(t,x)$ is the pressure exerted by the fluid on the object at the point $(x,\zeta_{\rm w}(t,x))$.
Note that at equilibrium, the pressure is hydrostatic, $\underline{P}_{\rm i}=-\rho g (h_{\rm eq}-h_0)$,  so that 
$$
m=\rho\int_{-\ell}^\ell (h_0-h_{\rm eq}(x)){\rm d}x\qquad \mbox{(Archimedes' principle)},
$$
and we can rewrite Newton's equation under the form
$$
m  \ddot{\delta} = \displaystyle \int_{-\ell}^\ell \big(\Pint(t,x)  + \rho g (h_{\rm eq}-h_0) \big)   {\rm d}x.
$$
By definition of the hydrodynamic pressure, its value $\Pi_{\rm i}$ in the interior domain $(-\ell,\ell)$ is given by
$$
\Pi_{\rm i}=\underline{P}_{\rm i}+\rho g \zeta_{\rm w},
$$
from which we infer,  using \eqref{eqzwD},
\begin{equation}\label{Newton-z_GD2}
\tau_{\rm buoy}^2  \ddot{\delta} + \delta= \displaystyle \frac{1}{2\rho g\ell} \int_{-\ell}^\ell \Pi_{\rm i}(t,x)   {\rm d}x,
\end{equation}
where $2\pi \tau_{\rm buoy}$ is the {\it buoyancy} period defined through
$$
\tau_{\rm buoy}^2=\frac{m}{2\ell \rho g}.
$$
\medbreak

We now proceed to derive dimensionless versions of \eqref{AbbottD}, \eqref{eqzwD}, \eqref{Newton-z_GD}. We recall that $h_0$ denotes the water depth at rest, and also denote by $a$ and $L$ the typical amplitude of the waves and a typical horizontal scale respectively. For the Boussinesq-Abbott equations \eqref{AbbottD}, we use the following scalings
$$
\widetilde{x}=\frac{x}{L}, \quad \widetilde{z}=\frac{z}{h_0}, \quad \widetilde{t}=\frac{t}{L/\sqrt{gh_0}}, \quad \widetilde{\zeta}=\frac{\zeta}{a}, \quad \widetilde{q}=\frac{q}{a\sqrt{gh_0}},\quad\widetilde{\underline{P}}=\frac{\underline{P}}{\rho g h_0}
$$
and consequently $\widetilde{h}=1+\eps\widetilde{\zeta}$.  We also introduce the nonlinearity and shallowness parameters $\eps$ and $\mu$ as
$$
\eps=\frac{a}{h_0},\qquad \mu=\frac{h_0^2}{L^2}.
$$
For the sake of clarity the tildes used to denote dimensionless quantities are omitted throughout this paper. The system \eqref{AbbottD} thus becomes 
\begin{equation}\label{AbbottDL}
\begin{cases}
\dt \zeta + \dx q=0,\\
(1-\frac{1}{3}\mu \dx^2 )\dt q +\eps \dx \big( \frac{1}{h}q^2 \big)+h \dx \zeta=-\frac{1}{\eps}h\dx \underline{P}\qquad (h=1+\eps\zeta).
\end{cases}
\end{equation}
\begin{remark}
The dimensionless form of the hydrodynamic pressure is naturally
$$
 \widetilde{\Pi}=\frac{\Pi}{\rho g h_0}=\widetilde{\underline{P}}+{\eps}\widetilde{\zeta},
$$
so that the dimensionless version of the alternative formulation \eqref{AbbottDbis} is (omitting the tildes)
\begin{equation}\label{AbbottDbisL}
\begin{cases}
\dt \zeta + \dx q=0,\\
(1-\frac{\mu}{3}\dx^2 )\dt q +\eps \dx \big( \frac{1}{h}q^2 \big)=-\frac{h}{\eps}\dx \Pi;
\end{cases}
\end{equation}
\end{remark}

In order to derive the dimensionless versions of \eqref{eqzwD}, \eqref{Newton-z_GD} and \eqref{Newton-z_GD}, we also need the following scalings 
$$
\widetilde{\zeta}_{\rm w}=\frac{\zeta_{\rm w}}{a},\qquad \widetilde{\delta}=\frac{\delta}{a},\qquad \widetilde{h}_{\rm eq}=\frac{h_{\rm eq}}{h_0},
\qquad \widetilde{m}=\frac{m}{2\ell\rho h_0},\qquad \widetilde{\tau}_{\rm buoy}=\frac{\tau_{\rm buoy}}{L/\sqrt{gh_0}},\qquad \widetilde{\ell}=\frac{\ell}{L}
$$
so that, omitting again the tildes for the sake of readability,  we can rewrite \eqref{eqzwD} and \eqref{Newton-z_GD} as
\begin{equation}\label{eqzwDL}
\zeta_{\rm w}(t,x)=\delta(t)+ \frac{1}{\eps} \big( h_{\rm eq}(x)-1\big).
\end{equation}
and
\begin{equation}\label{Newton-z_G-DL0}
\tau_{\rm buoy}^2  \ddot{\delta} + \frac{1}{\eps}m= \displaystyle \frac{1}{\eps}\frac{1}{2\ell}\int_{-\ell}^\ell \underline{P}_{\rm i}(t,x){\rm d}x;
\end{equation} 
note that in these dimensionless coordinates, the coordinates of the vertical sides of the object are $x=\pm \ell$ and that Archimedes' principle reads in dimensionless form as
$$
 m=\frac{1}{2\ell}\int_{-\ell}^\ell (1-h_{\rm eq}).
$$
Finally, the dimensionless version of \ref{Newton-z_GD2} is
$$
\tau_{\rm buoy}^2  \ddot{\delta} + \delta= \displaystyle \frac{1}{\eps}\frac{1}{2\ell} \int_{-\ell}^\ell \Pi_{\rm i}(t,x)   {\rm d}x.
$$

\end{document}